\documentclass[11pt]{amsart}
\usepackage[latin1]{inputenc}
\usepackage[english]{babel}
\usepackage[T1]{fontenc}
\usepackage{amssymb,amsmath,amsthm,amsfonts,enumerate}
\usepackage[autostyle]{csquotes}
\usepackage{alphalph,cite}
\usepackage{nameref}
\usepackage{verbatim}
\usepackage{mathrsfs}
\usepackage{manfnt}
\usepackage{pifont}
\usepackage{yfonts}
\usepackage[perpage]{footmisc}
\usepackage{enumitem}
\usepackage[all]{xy} 
\usepackage{wasysym}

\usepackage{stmaryrd}

\usepackage{apptools}
\usepackage{stackengine}
\usepackage{scalerel}
\usepackage{array}
\usepackage{datetime}
\usepackage{xparse} 
\usepackage{tikz-cd}
\usepackage{rotating}
\usepackage{esvect}
\usepackage[colorlinks=true, linkcolor={blue}, citecolor={magenta}, urlcolor={green}]{hyperref}
\usepackage{mathtools}
\usepackage{stackengine}

\makeatletter
\@namedef{subjclassname@2010}{ \textup{2010} Mathematics Subject Classification}
\makeatother
\theoremstyle{definition}
\newtheorem{theorem}{Theorem}[subsection]

\newtheorem{theoremletters}{Theorem}

\theoremstyle{definition}
\newtheorem{lemma}[theorem]{Lemma}
\theoremstyle{definition}
\newtheorem{corollary}[theorem]{Corollary}
\theoremstyle{definition}
\newtheorem{proposition}[theorem]{Proposition}
\theoremstyle{definition}

\theoremstyle{definition}
\newtheorem{remark}[theorem]{Remark}
\theoremstyle{definition}
\newtheorem{example}[theorem]{Example}
\theoremstyle{definition}
\newtheorem{question}[theorem]{Question}
\theoremstyle{definition}
\newtheorem{definition}[theorem]{Definition}
\numberwithin{equation}{section}

\numberwithin{equation}{section}
\theoremstyle{definition}
\newtheorem{warning}[theorem]{Warning}
\theoremstyle{definition}
\newtheorem{notation}[theorem]{Notation}
\theoremstyle{definition}

\theoremstyle{definition}
\newtheorem{claim}[theorem]{Claim}
\theoremstyle{definition}
\newtheorem{digression}[theorem]{Digression}
\theoremstyle{definition}

\DeclareMathOperator{\supp}{supp}

\DeclareMathOperator*{\esssup}{ess\,sup}

\newcommand{\norm}[1]{\left\lVert#1\right\rVert}
\newcommand{\vertiii}[1]{{\left\vert\kern-0.25ex\left\vert\kern-0.25ex\left\vert #1 
		\right\vert\kern-0.25ex\right\vert\kern-0.25ex\right\vert}}
\newcommand{\abs}[1]{\left\lvert#1\right\rvert}

\newcommand{\E}{{\bf E}}
\newcommand\xrowht[2][0]{\addstackgap[.5\dimexpr#2\relax]{\vphantom{#1}}}

\makeatletter
\def\author@andify{
	\nxandlist {\unskip ,\penalty-1 \space\ignorespaces}
	{\unskip {} \@@and~}
	{\unskip \penalty-2 \space \@@and~}
}
\makeatother

\newcommand\xqed[1]{%
	\leavevmode\unskip\penalty9999 \hbox{}\nobreak\hfill
	\quad\hbox{#1}}
\newcommand\demo{\xqed{{\Large $\blacktriangle$}}}

\makeatletter
\def\l@subsection{\@tocline{2}{0pt}{1pc}{5pc}{}} \def\l@subsection{\@tocline{2}{0pt}{2pc}{6pc}{}}
\makeatother

\begin{document}
	\title{Arithmetic, interpolation and factorization of amalgams}
	
	\author{Tomasz Kiwerski}
	\address[Tomasz Kiwerski]{Pozna\'{n} University of Technology, Institute of Mathematics, Piotrowo 3A, 60-965 Pozna\'{n}, Poland}
	\email{\href{mailto:tomasz.kiwerski@gmail.com}{\tt tomasz.kiwerski@gmail.com}}

	\author{Jakub Tomaszewski}
	\address[Jakub Tomaszewski]{Pozna\'{n} University of Technology, Institute of Mathematics, Piotrowo 3A, 60-965 Pozna\'{n}, Poland}
	\email{\href{mailto:tomaszewskijakub@protonmail.com}{\tt tomaszewskijakub@protonmail.com}}
	
	\maketitle
	
	\begin{center}
		{\it Dla Aleksandra}
	\end{center}
	
	\begin{abstract}
		Building upon Bennett's and Grosse-Erdmann's ideas falling under the conceptual umbrella of factorization of inequalities, we propose
		a unified approach towards the structure of certain Banach ideal spaces defined in terms of the least decreasing majorant.
		The key to our results, and, it seems, the main novelty in general, is the synthesis of discretization process,
		usually called the blocking technique, along with some tools from the interpolation theory.
		This blend allows us to obtain an abstract versions of several remarkable results proposed by Bennett and to show certain phenomena
		in new, somehow more complete perspective.
		Furthermore, with the help of technology we have developed, we re-prove and sometimes also improve many more recent results
		belonging to this circle of ideas.
	\end{abstract}
	
	\tableofcontents
	
	 \footnotetext[0]{
		{\it Date:} \longdate{\today}.
		
		
	}

	\section{{\bf Introduction}} \label{SECTION: Introduction}
	
	\subsection{Motivation and goal}
	
	Throughout the following we fix a Banach ideal space $X$.
	The main object of our studies will be the space $\widetilde{X}$, which is understood as a vector space of all measurable functions $f$
	such that their {\it least decreasing majorant} $\widetilde{f}$ belongs to $X$, furnished with the norm
	$\norm{f}_{\widetilde{X}} \coloneqq \Vert \widetilde{f} \Vert_X$.
	Here, $\widetilde{f}$ can be explicitly expressed as $\widetilde{f}(x) = \esssup_{t \geqslant x} \abs{f(t)}$.
	
	Dual object to $\widetilde{X}$ can be identified with the so-called {\it down space} $(X^{\times})^{\downarrow}$, where $X^{\times}$ is a (K{\" o}the) dual of $X$.
	Following Gord Sinnamon, these spaces are defined by the norm $\norm{f}_{X^{\downarrow}} \coloneqq \sup \int \abs{f(t)}g(t)dt$, where the supremum
	is taken over all positive and decreasing functions $g$ from the unit ball in $X^{\times}$ (see \cite{Si94} and \cite{HS23}; cf. \cite{LM15a} and \cite{Si07}).
	This construction obviously resembles the K{\" o}the dual construction but restricted to the cone of decreasing functions.
	There are many reasons why it should be considered natural and, most likely, also important. For example,
	$\bullet$ it leads to an improved version of the H{\" o}lder--Rogers inequality (this idea goes back to Israel Halperin and his {\it $D$-type H{\" o}lder inequality}
	with the {\it exact} norm given in terms of the level functions; see \cite{Hal53}; cf. \cite{Lor53});
	$\bullet$ can be used to describe the duals of Lorentz spaces $\Lambda_{p,w}^{\times}$ and $\Gamma_{p,w}^{\times}$ (see \cite{Hal53}, \cite{Lor53} and \cite{KM07}; cf. \cite{Ste93});
	$\bullet$ gives a general variant of {\it Sawyer's duality} theorem which can be seen as a variant of $D$-type H{\" o}lder inequality
	(the idea is that the weighted inequalities for positive functions are {\it equivalent} to the corresponding inequalities for decreasing functions;
	see \cite{GP00}, \cite{GHS96}, \cite{KM07}, \cite{PRP22} and \cite{Saw90});
	$\bullet$ supports the developments of {\it level functions} (see \cite{Si91}, \cite{Si94}, \cite{Si01} and \cite{Si07}; see also \cite{FLM16} where a comparison
	between Halperin's and Sinnamon's approach to the construction of level function is offered);
	$\bullet$ is applied to prove the boundedness of Hardy's operator and Fourier transform on weighted $L_p$ spaces (also on weighted Lorentz spaces;
	see, for example, \cite{Saw90} and \cite{Ste93}).
	
	In many situations, but not always, the space $X^{\downarrow}$ can be equipped with an equivalent norm given as
	$\norm{f}_{X^{\downarrow}} \approx \norm{x \rightsquigarrow \frac{1}{x}\int_0^x \abs{f(t)}dt}_X$ (see \cite{Si07}; cf. \cite{KMS07}).
	From this perspective, the space $X^{\downarrow}$ can be viewed as the {\it optimal domain} for Hardy's operator
	$\mathscr{H} \colon f \rightsquigarrow \left[ x \rightsquigarrow \frac{1}{x}\int_0^x f(t)dt \right]$,
	that is, as the largest Banach ideal space with the property that $\mathscr{H}$ is still bounded when acting into $X$.
	By no means is this an isolated phenomenon. For example, $\bullet$ real interpolation spaces; $\bullet$ extrapolation spaces;
	$\bullet$ Hardy spaces; $\bullet$ tent spaces; $\bullet$ Bochner spaces; and $\bullet$ Besov spaces all fall into the same pattern
	- they can be seen as the optimal domains for {\it some} positive and sublinear operators
	(studies on the structure of these spaces can be found, for example, in \cite{Ast13} and \cite{Mas91}; see also the monograph \cite{ORS08} and their references).
	Typically, the study of optimal domains $[T,X]$, where $T$ is a positive and sublinear operator, involves their representation as $L_1(\nu_X)$
	for some vector measure $\nu_X \colon \Sigma \rightarrow X$ associated with the operator $T$.
	This is not always possible, but in many situations leads to interesting results
	(see, for example, \cite{CR16}, \cite{DS07} and \cite{NP11}; cf. \cite{ORS08}). It should be mentioned that in many places in the literature
	(this paper is among them), the optimal domains $[\mathscr{H},X]$ are denoted $\mathscr{C}X$ and called {\it Ces{\' a}ro spaces}
	(see, for example, \cite{ALM19}, \cite{CR16}, \cite{KLM19}, \cite{Les15}, \cite{LM14}, \cite{LM15a} and \cite{LM16}).
	
	After roughly setting the scene, let us try to explain the context for what we are going to do here.
	The following result from Bennett's memoir \cite{Be96} should be considered as the starting point for our work: for $p > 1$, we have
	\begin{equation} \label{INTRODUCTION: Bennett factorization of H}
		\ell_p \odot g_q = ces_p.
	\end{equation}
	Here,
	\begin{equation*}
		\ell_p \odot g_q = \left\{ x = yz \colon y \in \ell_p \text{ and } z \in g_q \right\}
	\end{equation*}
	with the (quasi-)norm
	\begin{equation*}
		\norm{x}_{\ell_p \odot g_q} = \inf \left\{ \norm{y}_{\ell_p} \norm{z}_{g_q} \colon x = yz,\, y \in \ell_p \text{ and } z \in g_q \right\};
	\end{equation*}
	the space $ces_p$ is just the optimal domain $\mathscr{C}\ell_p$, that is,
	\begin{equation} \label{INTRODUCTION: norm cesp}
		ces_p = \left\{ x = \{x_n\}_{n=1}^{\infty} \colon \left\{ \frac{1}{n} \sum_{k=1}^n \abs{x_n} \right\}_{n=1}^{\infty} \in \ell_p \right\},
			\quad \norm{x}_{ces_p} = \norm{\left\{ \frac{1}{n} \sum_{k=1}^n \abs{x_n} \right\}_{n=1}^{\infty}}_{\ell_p},
	\end{equation}
	\begin{equation*}
		g_q = \left\{ x = \{x_n\}_{n=1}^{\infty} \colon \left\{ \frac{1}{n} \sum_{k=1}^n \abs{x_n}^q \right\}_{n=1}^{\infty} \in \ell_{\infty} \right\},
			\quad \norm{x}_{g_q} = \sup_{n \in \mathbb{N}} \left( \frac{1}{n} \sum_{k=1}^n \abs{x_k}^q \right)^{1/q};
	\end{equation*}
	and, finally, $q$ is the conjugate exponent to $p$, that is, $1/p + 1/q = 1$.
	It is not a mere coincidence that \eqref{INTRODUCTION: Bennett factorization of H} is usually referred to as the {\it factorization of Hardy's inequality}.
	In fact, it is straightforward to see that since the constant sequence ${\bf 1} = (1,1,...)$ belongs to $g_q$, so
	\begin{equation} \label{INTRODUCTION: H ineq}
		\ell_p = \ell_p \odot \ell_{\infty} \hookrightarrow ces_p.
	\end{equation}
	The above inclusion \eqref{INTRODUCTION: H ineq} is noting else, but a {\it qualitative version of Hardy's inequality},
	which is typically presented as
	\begin{equation*}
		\left[ \sum_{n=1}^{\infty} \left( \frac{1}{n} \sum_{k=1}^n \abs{x_k} \right)^p \right]^{1/p} \leqslant C_p \left( \sum_{k=1}^n \abs{x_k}^p \right)^{1/p}
	\end{equation*}
	for $p > 1$ and some constant $C_p > 0$ (see \cite{Har25}; cf. \cite{HLP52} and \cite{KMP07}).
	Again, this is not an isolated situation and many celebrated inequalities like, for example, Hilbert's inequality or Copson's inequality,
	behave in a similar way (briefly, this idea and its consequences are the subject of \cite{Be96}).
	More recently, factorization \eqref{INTRODUCTION: Bennett factorization of H} has been extended to $L_p$ spaces and later even to weighted
	$L_p$ spaces (see \cite{AM09} and, respectively, \cite{BMM18}; cf. \cite{Lei98}).
	
	In retrospect, factorization of Hardy's inequality \eqref{INTRODUCTION: Bennett factorization of H} fits into the much broader topic of
	{\it factorization of \enquote{function spaces}} such as
	$\bullet$ Hardy spaces \cite{CRW76};
	$\bullet$ Bergman spaces \cite{Hor77} and \cite{PZ15};
	$\bullet$ tent spaces \cite{CV00}; and
	$\bullet$ Banach function spaces \cite{Gil81}, \cite{JR76}, \cite{Rei81} and \cite{Sch10}
	(like, Lorentz spaces, Marcinkiewicz spaces and Orlicz spaces; see \cite{CS17}, \cite{KLM12}, \cite{KLM14}, \cite{LT17} and \cite{LT21}).
	From this perspective, the space $ces_p$ can not only be decomposed into the pointwise product \eqref{INTRODUCTION: Bennett factorization of H} but,
	more importantly, a rather {\it ad-hoc} construction of the space $g_q$ can be now explained by introducing the space $M(\ell_p,ces_p)$. Simply, $g_q = M(\ell_p,ces_p)$.
	Here, $M(\ell_p,ces_p)$ is the space of {\it pointwise multipliers} between $\ell_p$ and $ces_p$, that is,
	\begin{equation*}
		M(\ell_p,ces_p) = \left\{ x \colon xy \in ces_p \text{ for all } y \in \ell_p \right\},
			\quad \norm{x}_{M(\ell_p,ces_p)} = \sup_{\norm{y}_{\ell_p} \leqslant 1} \norm{xy}_{ces_p}.
	\end{equation*}
	Now, even though the factorization
	\begin{equation} \label{INTRODUCTION: lp o M(lp,cesp) = cesp}
		\ell_p \odot M(\ell_p,ces_p) = ces_p
	\end{equation}
	follows from some general facts (for instance, it is enough to know that the space $ces_p$ is strictly $p$-concave... and it is;
	see \cite[pp.~316--318]{Sch10} and \cite[Example~6.5(c)]{KKM21}), they say nothing about an explicit description of the space
	$M(\ell_p,ces_p)$. Moreover, accepting that factorization \eqref{INTRODUCTION: H ineq} is important, the question whether
	\begin{equation*}
		ces_p \odot M(ces_p,ces_q) = ces_q \quad \text{ and } \quad ces_p \odot M(ces_p, \widetilde{\ell_q}) = \widetilde{\ell_q},
	\end{equation*}
	lurks on the horizon.
	Here things are even worse, because {\it a priori} we do not know if the above factorizations hold.
	Therefore, in some sense, we are forced to compute the spaces of multipliers $M(ces_p,ces_q)$ and $M(ces_p, \widetilde{\ell_q})$.
	This, however, is not an easy task (an instructive example showing how complicated this can be is the description of the space
	$M(ces_p,\ell_1) = ces_p^{\times}$ proposed by Jagers \cite{Ja74} or a characterization of pointwise multipliers between the weighted
	variants of $Ces_p$ spaces given by Gogatishvili, Pick and {\" U}nver \cite{GPU22}).
	
	A remedy (as long as we feel good about the results modulo some constants, and that is the case here) is the {\it blocking technique}.
	It turns out, and this is one of the most important lessons from Grosse-Erdmann's book \cite{GE98},
	that on both spaces $ces_p$ and $\widetilde{\ell_p}$ one can introduce an equivalent norms
	\begin{equation} \label{INTRODUCTION: blocking technique}
		\norm{x}_{ces_p} \approx \left( \sum_{j=0}^{\infty} 2^{j(1-p)} \left[ \sum_{k = 2^j}^{2^{j+1}-1} \abs{x_k} \right]^p \right)^{1/p},
			\quad
		\norm{x}_{{\widetilde{\ell_p}}} \approx \left( \sum_{j=0}^{\infty} 2^{-j} \left[ \sup_{k = 2^j, 2^j+1, ..., 2^{j+1}-1} \abs{x_k} \right]^p \right)^{1/p}.
	\end{equation}
	In other words, up to some weights, we have
	\begin{equation} \label{INTRODUCTION: amalgams}
		ces_p = \Bigl( \bigoplus_{n=0}^{\infty} \ell_{1}^{2^n} \Bigr)_{\ell_p}
			\quad \text{ and } \quad
		\widetilde{\ell_p} = \Bigl( \bigoplus_{n=0}^{\infty} \ell_{\infty}^{2^n} \Bigr)_{\ell_p}.
	\end{equation}
	The above norms \eqref{INTRODUCTION: blocking technique} are usually referred to as norms in the {\it block form}, while the process
	of converting \eqref{INTRODUCTION: norm cesp} into \eqref{INTRODUCTION: blocking technique} (or {\it vice versa}) is called the {\it blocking technique}.
	Moreover, since we prefer to look at direct sums \eqref{INTRODUCTION: amalgams} as sequence spaces, we will call them {\it amalgams}.
	The point is that many operations that are difficult to do on $ces_p$ are easy (or at least easier) to do on $\bigl( \bigoplus_{n=0}^{\infty} \ell_{1}^{2^n} \bigr)_{\ell_p}$.
	
	In the outlined context, it is quite natural to ask the following
	
	\begin{question}
		{\it What does the blocking technique have to offer in the general context, that is, when the $\ell_p$ space is replaced
		by a Banach sequence space and, analogously, the $L_p$ space is replaced by a Banach function space?
		What type of factorization results can be obtained in this way?
		What are the consequences of this approach to the interpolation structure of these spaces?}
	\end{question}
	
	A rough goal of this paper is to propose an answer to the above question.
	To put it briefly, just as the blocking technique itself was successfully used by Grosse-Erdmann in \cite{GE98} to reproduce and,
	in some sense, explain a significant part of Bennett's memoir \cite{Be96}, our main motivation here is to create a unified theory
	that will allow us to do essentially the same with a large body of more recent results obtained by various authors.
	Moreover, but this is merely our pious wish, we would like to convince the curious reader that the blocking technique approach to the
	structure of these spaces is the \enquote{right} one.
	
	As far as we know, this paper is the first attempt to gather all this ideas \enquote{under one roof}.
	
	\subsection{Results}
	Here we want to present our main results in more detail. To avoid unnecessary complications, we will usually formulate
	them in the less generality and for non-atomic measure spaces only.
	
	At the heart of our work is Section~\ref{SECTION: BF representations}.
	Our main result reads as follows (see Theorem~\ref{Thm: Tandori sequence representation}, Theorem~\ref{Thm: Tandori function representation}
	and Corollary~\ref{COR: blocking technique CX}).
	
	\vspace{5pt} \noindent {\bf Synopsis 1. Blocking technique.}
	{\it Let $1 < p < q < \infty$. Further, let $X$ be an interpolation space between $L_p(0,\infty)$ and $L_{q}(0,\infty)$.
	Then the space $\widetilde{X}$ admits an equivalent norm in the block form, that is,}
	\begin{equation*}
		\norm{f}_{\widetilde{X}} \approx \norm{\left\{ \norm{f\chi_{\Delta_j}}_{L_\infty} \right\}_{j \in \mathbb{Z}}}_{\mathbf{E}(X)},
	\end{equation*}
	{\it where $\Delta_j \coloneqq [2^{j},2^{j+1})$ for $j \in \mathbb{Z}$ and $\mathbf{E}(X)$ is a vector space of all sequences
	$x = \{x_j\}_{j \in \mathbb{Z}}$ such that $\sum_{j \in \mathbb{Z}} x_j \chi_{\Delta_j} \in X$, endowed with the norm
	$\norm{x}_{\mathbf{E}(X)} \coloneqq \norm{\sum_{j \in \mathbb{Z}} x_j \chi_{\Delta_j}}_X$.
	Moreover, by duality, we have}
	\begin{equation*}
		\norm{f}_{X^{\downarrow}}
			\approx \norm{f}_{\mathscr{C}X}
			\approx \norm{\left\{ 2^{-j} \norm{f\chi_{\Delta_j}}_{L_1} \right\}_{j \in \mathbb{Z}}}_{\mathbf{E}(X)}.
	\end{equation*}
	\vspace{5pt} \noindent
	
	In other words, $\widetilde{X} = \bigl( \bigoplus_{j \in \mathbb{Z}} L_{\infty}(\Delta_j) \bigr)_{{\bf E}(X)}$ and
	$X^{\downarrow} = \bigl( \bigoplus_{j \in \mathbb{Z}} L_{1}(\Delta_j) \bigr)_{{\bf E}(X)(w)}$, where $w(j) = 2^{-j}$ for $j \in \mathbb{Z}$.
	Looking back, this is, of course, a generalization of Grosse-Erdmann's block form representation \eqref{INTRODUCTION: blocking technique}
	(see also Corollary~\ref{COR: reprezentacje ces_p i Ces_p}).
	
	Among many consequences of the above results (which we will discuss later), perhaps the most immediate is the fact that isomorphically
	the structure of these spaces is rather simple. More precisely, we have
	\begin{align*}
		 \widetilde{L_p} \approx \ell_p(\ell_{\infty}) \quad \text{ and }
		 	\quad L_p^{\downarrow} \approx \ell_p(L_1);
	\end{align*}
	\begin{equation*}
		\widetilde{\Lambda_\varphi} \approx \ell_1(\ell_\infty) \quad \text{ and } \quad \Lambda_{\varphi}^{\downarrow} \approx L_1;
			\quad \widetilde{M_\varphi} \approx \ell_\infty \quad \text{ and } \quad M_{\varphi}^{\downarrow} \approx \ell_\infty(\ell_1).
	\end{equation*}
	Thus, they coincide with a subclass of Lebesgue--Bochner spaces (see Propositions~\ref{PROPOSITION: ALM19} and \ref{PROP: CCX = CX}).
	
	As for Section~\ref{SECTION: Interpolation structure}, our most important result regarding the interpolation structure can be formulated
	as follows (see Theorems~\ref{THM: CM-couples} and \ref{THM: relative CM-couples}).
	
	\vspace{5pt} \noindent {\bf Synopsis 2. Interpolation structure.}
	{\it Let $X$ and $Y$ be two interpolation spaces between $L_1(0,\infty)$ and $L_{\infty}(0,\infty)$.
	Then the couples $(X,Y)$, $(\widetilde{X},\widetilde{Y})$ and $(X^{\downarrow},Y^{\downarrow})$ are a Calder{\' o}n--Mityagin couples if,
	and only if, one of them is a Calder{\' o}n--Mityagin couple.}
	\vspace{5pt} \noindent
	
	This means, in particular, that the couples $(\widetilde{L_p},L_{\infty})$ and $(L_1,L_q^{\downarrow}) = (L_1,Ces_q)$ are a Calder{\' o}n--Mityagin couples
	(see Corollaries~\ref{COR: MS06} and \ref{COR: Lesnik}). On the other hand, for $1 < p < q < \infty$, the couples $(L_p,\widetilde{L_q})$,
	$(L_p^{\downarrow}, L_q) = (Ces_p, L_q)$ and $(L_p^{\downarrow}, \widetilde{L_q}) = (Ces_p,\widetilde{L_q})$ are not Calder{\' o}n--Mityagin
	couples (see Example~\ref{THM: NOT CM-couples}).
	
	Technical side of our results (apart from, perhaps, Propositions~\ref{Prop: E komutuje z Kothe dualem} and \ref{PROP: EX <-> carrier})
	mainly covers Section~\ref{SECTION: Products and factors: Grosse-Erdmann's style}. Henceforth, we will prefer to talk about the construction
	$\mathscr{C}X$ instead of $X^{\downarrow}$ (even if this requires some additional assumptions). This is dictated by the fact that both \cite{Be96}
	and \cite{GE98}, which we are obviously inspired by, deal only with $ces_p$ and $Ces_p$. Our most relevant results here can be summarized as follows
	(see Theorem~\ref{Cor: multipliers between Cesaro and Tandori}, Remark~\ref{REMARK: M komutuje z E} and Theorem~\ref{PROP: Tandori komutuje z produktem}).
	
	\vspace{5pt} \noindent {\bf Synopsis 3. Products and factors.}
	{\it Let $X$ and $Y$ be two interpolation spaces between $L_1(0,\infty)$ and $L_{\infty}(0,\infty)$. Then
	\begin{equation*}
		M(\widetilde{X},\widetilde{Y}) = \widetilde{M(X,Y)} \quad \textit{ and } \quad M(\mathscr{C}X, \mathscr{C}Y) = \mathscr{C}M(X,Y)
	\end{equation*}
	if, and only if, the discretization construction ${\bf E}$ commutes with the pointwise multipliers construction $M$, that is,
	${\bf E}(M(X,Y)) = M({\bf E}(X),{\bf E}(Y))$. Moreover, if either we have factorization $Y = X \odot M(X,Y)$ or the spaces $X$
	and $Y$ are strongly separated, that is, $\alpha_X > \beta_Y$, then ${\bf E}(M(X,Y)) = M({\bf E}(X),{\bf E}(Y))$. On the other hand,
	we also have}
	\begin{equation*}
		\widetilde{X} \odot \widetilde{Y} = \widetilde{X \odot Y}.
	\end{equation*}
	\vspace{5pt} \noindent
	
	We will also analyze the situation in which the above assumptions are not met. For example, we will show that
	for two Orlicz sequence spaces $\ell_M$ and $\ell_N$ such that $M(\ell_M,\ell_N) = \ell_{\infty}$ (this means that
	$\ell_N \neq \ell_M \odot M(\ell_M,\ell_N)$), we have
	\begin{equation*}
		M(ces_M,ces_N) = \ell_{\infty}(w_{M \to N}),
	\end{equation*}
	where $w_{M \to N}(n) \coloneqq \norm{\sum_{k=1}^n e_k}_{\ell_M} / \norm{\sum_{k=1}^n e_k}_{\ell_N}$ for $n \in \mathbb{N}$ (see Example~\ref{przyklad Orlicze infty}).
	In this way, taking $M$ and $N$ as power functions, we will recover the striking Bennett's result that
	\begin{equation*}
		M(ces_p,ces_q) = \ell_{\infty}(n^{1/q - 1/p})
	\end{equation*}
	provided $1 < p \leqslant q < \infty$ (see Corollary~\ref{COR: Bennett M(ces_p,ces_q)}). Similarly, we will also check what happens
	when we replace Orlicz sequence spaces with Lorentz sequence spaces or Marcinkiewicz sequence spaces, that is, we compute
	\begin{equation*}
		M(\mathscr{C}m_{\varphi},\mathscr{C}m_{\psi}), \quad M(\mathscr{C}m_{\varphi},\mathscr{C}\lambda_{\psi})
			\quad \text{ and } \quad M(\mathscr{C}\lambda_{\varphi},\mathscr{C}\lambda_{\psi})
	\end{equation*}
	(see Example~\ref{EXAMPLE: CLorentz i CMarcinkiewicz a'la Bennett}; cf. Example~\ref{EXAMPLE: Lorentz i Marcinkiewicz Bennett <=>}).
	
	Finally, let us move on to the last Section~\ref{SECTION: Factorization}.
	Our results about factorization can be roughly summarized as follows (see Theorems~\ref{THM: factorization CX with Lp}, \ref{THM: Lp factorized through Tandori X},
	\ref{THEOREM: Factorization CX through CY}, \ref{THEOREM: factorization Tandori przez Tandori} and \ref{THEOREM: Mixed factorizations}; we deliberately omit here
	some special cases involving Lorentz spaces $\Lambda_{\varphi}$ and Marcinkiewicz spaces $M_{\psi}$).
	
	\vspace{5pt} \noindent {\bf Synopsis 4. Factorizations.}
	{\it Let $1 < p,q < \infty$ be given. Let $X$ and $Y$ be two interpolation spaces between $L_q(0,\infty)$ and $L_{\infty}(0,\infty)$.
	Suppose that either $X = L_p \odot M(L_p,X)$ or $L_p = X \odot M(X,L_p)$. Then}
	\begin{equation*}
		\mathscr{C}X = L_p \odot M(L_p,\mathscr{C}X) \quad \textit{ and, respectively, } \quad L_p = \widetilde{X} \odot M(\widetilde{X},L_p).
	\end{equation*}
	{\it Now, suppose that ${\bf E}(Y) = {\bf E}(X) \odot M({\bf E}(X),{\bf E}(Y))$. Then}
	\begin{equation*}
		\mathscr{C}Y = \mathscr{C}X \odot M(\mathscr{C}X,\mathscr{C}Y), \quad \widetilde{Y} = \widetilde{X} \odot M(\widetilde{X},\widetilde{Y}),
	\end{equation*}
	\begin{equation*}
		\mathscr{C}Y = \widetilde{X} \odot M(\widetilde{X},\mathscr{C}Y) \quad \textit{ and } \quad \widetilde{Y} = \mathscr{C}X \odot M(\mathscr{C}X,\widetilde{Y}).
	\end{equation*}
	\vspace{5pt} \noindent
	
	It is worth mentioning that we were able to calculate the spaces of multipliers and only then deduce the corresponding factorizations.
	Note also that the assumption ${\bf E}(Y) = {\bf E}(X) \odot M({\bf E}(X),{\bf E}(Y))$ is strictly weaker than the perhaps more natural
	condition $Y = X \odot M(X,Y)$ and this alone improves some existing results (see Proposition~\ref{PROPOSITION: faktoryzacja z E lepsza niz bez E}).
	
	\subsection{Overview of the key ideas} \label{SUBSECTION: opis techniczny}
	This outline should only be taken merely as a quick indication of the new and interesting points of our results (which mainly consists of what can be found
	in Sections~\ref{SUBSECTION: E(X)}, \ref{SUBSECTION: interpolation of E(X)}, \ref{SUBSECTION: revisiting Bukhvalov}, \ref{SUBSECTION: pointwise multipliers}
	and \ref{SUBSECTION: pointwise products}) and their proofs. We will also try to highlight some analogies with previous results. To simplify the presentation,
	we will assume that from this point onward $X$ is a rearrangement invariant function space with the Fatou property and non-trivial Boyd's indices (these assumptions
	can be relaxed and, with some modifications, everything we say below has its sequence counterparts).
	For this reason, some notation or details used below may differ slightly from what we will present later.
	
	As we have already mentioned above \eqref{INTRODUCTION: blocking technique}, the fact that the space $Ces_p$ can be represented (up to an isomorphism)
	as $\bigl( \bigoplus_{n=1}^{\infty} L_{1}(\Delta_j) \bigr)_{\ell_p}$ has basically been known since the appearance of Grosse-Erdmann's book \cite{GE98}.
	For this reason, even though the prediction that the space $\mathscr{C}X$ can be represented as $\bigl( \bigoplus_{n=1}^{\infty} L_{1}(\Delta_j) \bigr)_{\blacksquare}$
	seems to be intuitively obvious, the right choice of the space $\blacksquare$ brings unexpected complications. There are basically two reasons for this:
	we know that $\mathscr{C}\Lambda_{\varphi} \approx \bigl( \bigoplus_{n=1}^{\infty} L_{1}(\Delta_j) \bigr)_{\ell_1}$
	and $\mathscr{C}M_{\psi} \approx \bigl( \bigoplus_{n=1}^{\infty} L_{1}(\Delta_j) \bigr)_{\ell_{\infty}}$
	(see \cite[Theorem~4.1]{ALM19} and \cite[Theorem~4.4]{DS07}).
	This irrevocably eliminates the most naive choice of $\blacksquare$ as the {\it discretization} of $X$, which should be specified as a vector space
	\begin{equation*}
		{\bf D}(X) \coloneqq \left\{ x = \{x_n\}_{n=1}^{\infty} \colon \sum_{n=1}^{\infty} x_n \chi_{\diamondsuit_n} \in X \right\},
	\end{equation*}
	where $\diamondsuit_n \coloneqq [n,n+1)$ for $n \in \mathbb{N}$, endowed with the norm $\norm{x}_{{\bf D}(X)} \coloneqq \norm{\sum_{n=1}^{\infty} x_n \chi_{\diamondsuit_n}}_X$.
	Intuitively speaking, such a discretization procedure should yield the representations $\bigl( \bigoplus_{n=1}^{\infty} L_{1}(\Delta_j) \bigr)_{\lambda_{\varphi}}$
	and $\bigl( \bigoplus_{n=1}^{\infty} L_{1}(\Delta_j) \bigr)_{m_{\psi}}$, which evidently are isomorphically distinct from $\bigl( \bigoplus_{n=1}^{\infty} L_{1}(\Delta_j) \bigr)_{\ell_1}$
	and, respectively, $\bigl( \bigoplus_{n=1}^{\infty} L_{1}(\Delta_j) \bigr)_{\ell_{\infty}}$.
	Several attempts show that the right choice of $\blacksquare$ is to take the {\it dyadic discretization} of $X$, that is, a vector space
	\begin{equation*}
		\mathbf{E}(X) \coloneqq \left\{ x = \{x_j\}_{j \in \mathbb{Z}} \colon \sum_{j \in \mathbb{Z}} x_j \chi_{\Delta_j} \in X \right\},
	\end{equation*}
	where $\Delta_j \coloneqq [2^{j},2^{j+1})$ for $j \in \mathbb{Z}$, equipped with the norm $\norm{x}_{\mathbf{E}(X)} \coloneqq \norm{\sum_{j \in J} x_j \chi_{\Delta_j}}_X$
	(see Definition~\ref{Def: E(X) function}).
	
	\begin{remark}
		To the best of our knowledge, the construction $X \rightsquigarrow {\bf E}(X)$ was used for the
		first time in the early 1990s by Nigel Kalton to describe the Calder{\' o}n--Mityagin couples of rearrangement invariant spaces (see \cite{Kal92}).
		However, even more recently, Sergey Astashkin found another use for it to describe (in terms of Boyd indices) the symmetric finite representability
		of $\ell_p$'s in rearrangement invariant function spaces (see \cite{Ast22}).
		Without going into details, his idea was that the dilation operator on a rearrangement invariant space $X$ behaves like the shift operator
		on the discretization ${\bf E}(X)$.
		On the other hand, Amiran Gogatishvili and Vladimir Ovchinnikov came up with a seemingly distant application of the construction
		$X \rightsquigarrow {\bf E}(X)$ in the context of the optimal Sobolev's embeddings (see \cite{GO07}).
		They did this somewhat unconsciously by emphasizing the concept of {\it orbital equivalence} rather than the construction $X \rightsquigarrow {\bf E}(X)$.
	\end{remark}
	
	The right choice of discretization allows us to represent the space $\widetilde{X}$ (and, through duality, also the space $\mathscr{C}X$ - we will come
	back to this later) in the block form (see Theorem~\ref{Thm: Tandori sequence representation} and Theorem~\ref{Thm: Tandori function representation}).
	This looks roughly like this
	\begin{equation} \label{Introduction EQ: 1}
		\Vert \widetilde{f} \Vert_X
			\approx \norm{ \sum_{j \in J} \norm{f\chi_{\Delta_j}}_{L_{\infty}} \chi_{\Delta_j} }_X
			= \norm{\left\{ \norm{f\chi_{\Delta_j}}_{L_\infty} \right\}_{j \in \mathbb{Z}}}_{\mathbf{E}(X)}.
	\end{equation}
	In other words, $\widetilde{X} = \bigl( \bigoplus_{j \in \mathbb{Z}} L_{\infty}(\Delta_j) \bigr)_{{\bf E}(X)}$.
	
	\begin{warning}
		Pedantically speaking, we should say that both spaces $\widetilde{X}$ (with the equivalent norm given by the right-hand side of \eqref{Introduction EQ: 1})
		and $\bigl( \bigoplus_{j \in \mathbb{Z}} L_{\infty}(\Delta_j) \bigr)_{{\bf E}(X)}$ are isometrically isomorphic via the mapping
		$\Delta \colon f \rightsquigarrow \{f\chi_{\Delta_j}\}_{j \in \mathbb{Z}}$.
		However, the later space can naturally be treated as a Banach function space and, moreover, the isomorphism $\Delta$ not only preserves the lattice structure,
		but also is an isomorphism in the sense of interpolation (see \cite[Definition~2.1.12, p.~96]{BK91}). For this reason, we will forgo long-winded formulations
		and {\it par abus} we will write $\widetilde{X} = \bigl( \bigoplus_{j \in \mathbb{Z}} L_{\infty}(\Delta_j) \bigr)_{{\bf E}(X)}$.
	\end{warning}

	\begin{remark}[Wiener amalgam spaces] \label{REMARK: Wiener amalgams}
		Blocking technique in the context of harmonic analysis leads to the notion of {\it amalgams}, that is, a class of function spaces
		defined by a norm which amalgamates, or simply mixes, a local and global beheviour of a given function. More precisely, for a Banach
		function space $X$ defined on $\Omega$ and a Banach sequence space $E$ defined on $J$, their {\it Wiener amalgam $\mathscr{W}(X,E)$}
		is defined as a vector space
		\begin{equation*}
			\mathscr{W}(X,E) \coloneqq \left\{ f \in L_0(\Omega) \colon \left\{ \norm{f\chi_{\Omega_j}}_X \right\}_{j \in J} \in E \right\}
		\end{equation*}
		equipped with the norm
		\begin{equation*}
			\norm{f}_{\mathscr{W}(X,E)} \coloneqq \norm{\left\{ \norm{f\chi_{\Omega_j}}_X \right\}_{j \in J}}_E,
		\end{equation*}
		where $\{\Omega_j\}_{j \in J}$ is a decomposition of $\Omega$ into disjoint measurable subsets.
		Roughly speaking, the local data $X$ is pasted together along $E$, that is, the function $f$ belongs to the amalgam $\mathscr{W}(X,E)$
		provided it is \enquote{locally} in the space $X$ and \enquote{globally} in the space $E$. Clearly, both spaces $\mathscr{W}(X,E)$ and
		$\bigl( \bigoplus_{j \in J} X(\Omega_j) \bigr)_E$ are isometrically isomorphic.
		Historically, while working on Tauberian theorems, Norbert Wiener defined the spaces $\mathscr{W}(L_1,\ell_2)$, $\mathscr{W}(L_2,\ell_1)$,
		$\mathscr{W}(L_1,\ell_\infty)$ and $\mathscr{W}(L_\infty,\ell_1)$, but we owe the naming and the general construction to Hans Georg Feichtinger
		(see \cite{Fei90} and \cite{Wie32}; cf. \cite{Beu49}, \cite{Hed69} and \cite{Hol75}). Note also that $\mathscr{W}(L_p,\ell_p)$
		is isometrically isomorphic to $L_p$, that is, the $L_p$-norm does not distinguish between local and global properties (see \cite{Pes22} for more).
	\end{remark}
	
	A few comments about the proof of the equivalence \eqref{Introduction EQ: 1} seem in order. The argument behind the more interesting of the two inequalities
	is based on the following observation
	\begin{align*}
		\norm{f}_{\widetilde{X}}
			& \leqslant \norm{\sum_{j \in \mathbb{Z}} \norm{f\chi_{\Delta_j}}_{L_\infty}\chi_{\Delta_j}}_{\widetilde{X}} \quad (\text{since $X$ has the ideal property}) \\
			& = \norm{\left\{ \norm{f\chi_{\Delta_j}}_{L_\infty} \right\}_{j \in \mathbb{Z}}}_{\widetilde{\mathbf{E}(X)}}
					\quad (\text{by the very definition of $\widetilde{X}$ and ${\bf E}(X)$}) \\
			& \leqslant \norm{\mathscr{T}}_{\mathbf{E}(X)
					\rightarrow \mathbf{E}(X)} \norm{\left\{ \norm{f\chi_{\Delta_j}}_{L_\infty} \right\}_{j \in \mathbb{Z}}}_{\mathbf{E}(X)} \quad (\text{by interpolation (\ding{80})}),
	\end{align*}
	where the positive and sublinear operator $\mathscr{T}$ is defined for $x = \{ x_j \}_{j \in \mathbb{Z}}$ as
	\begin{equation*}
		\mathscr{T} \colon x \rightsquigarrow \mathscr{T}(x) \coloneqq \left\{ \sup_{k \geqslant j} \abs{x_k} \right\}_{j \in \mathbb{Z}}.
	\end{equation*}
	Only the interpolation part (\ding{80}) requires some further explanation. Since $X = (L_1,L_{\infty})^{\mathscr K}_{\mathscr X}$ for some parameter space $\mathscr{X}$
	(this is due to the celebrated result of Calder{\' o}n and Mityagin; see \cite{Cal64} and \cite{Mit65}) and, as is easy to see, the space ${\bf E}(X)$ is a
	$1$-complemented subspace in $X$, so the couple $({\bf E}(L_1),{\bf E}(L_{\infty}))$ is the {\it $\mathscr{K}$-subcouple} of the couple $(L_1,L_{\infty})$
	(in the sense of Petree, Pisier and Janson; see \cite{Jan93} and \cite[p.~465]{BK91}). In consequence,
	\begin{equation*}
		{\bf E}(X) = {\bf E}( (L_1,L_{\infty})^{\mathscr K}_{\mathscr X} ) = \left( {\bf E}(L_1),{\bf E}(L_{\infty}) \right)^{\mathscr K}_{\mathscr X}.
	\end{equation*}
	This means, that we only need to know that $\mathscr{T}$ is bounded when acting from ${\bf E}(L_1)$ into ${\bf E}(L_1)$ and, respectively, form ${\bf E}(L_{\infty})$
	into ${\bf E}(L_{\infty})$. The latter is very simple, because ${\bf E}(L_{\infty}) \equiv \ell_{\infty}$. For the remaining one, since ${\bf E}(L_1) \equiv \ell_1(w)$
	with $w(j) = 2^{j}$ for $j \in \mathbb{Z}$, so the boundedness of $\mathscr{T} \colon {\bf E}(L_1) \rightarrow {\bf E}(L_1)$ will follow once we can show the inequality
	\begin{equation} \label{INTRODUCTION: principle of duality}
		\sum_{j \in \mathbb{Z}} 2^j \left(\sup\limits_{k \geqslant j} \abs{x_k} \right) \leqslant C \sum_{j \in \mathbb{Z}} 2^j \abs{x_j}.
	\end{equation}
	Surprisingly, with $C = 2$ this is an immediate consequence of Goldman, Heinig and Stepanov's result (see \cite[Proposition~2.1]{GHS96}; cf. \cite[Lemma~3.1(i)]{GP03}).
	Our interpolation argument is thus complete.
	
	\begin{remark}
		In the context of \eqref{INTRODUCTION: principle of duality}, note that the \enquote{principle of duality} from \cite{GHS96},
		which through the notion of {\it discretizing sequences} leads, for example, to the description of the space $\Gamma_{p,w}^{\times}$,
		were further developed by Gogatishvili and Pick \cite{GP03} in the form of \enquote{discretization} and \enquote{anti-discretization}
		of rearrangement invariant norms.
		We will not go into intricate technical details of these constructions and refer the interested reader directly to \cite{GHS96}
		and \cite{GP03} for more information.
		However, let us only mention that the main innovation of \cite{GP03} is that certain conditions from \cite{GHS96} expressed with the help of
		discretizing sequences (for which direct verification is somehow ineffective) can be formulated in a more manageable integral forms
		(compare Theorems~3.1, 3.2 and 3.3 from \cite{GHS96} with Theorems~4.2, 5.1 and, respectively, 6.2 from \cite{GP03}).
		While the term \enquote{discretization} is intuitively clear, \enquote{anti-discretization} refers to the process just described.
		Parenthetically speaking, the blocking technique is also used in \cite{GP03}.
	\end{remark}
	
	Back to $\mathscr{C}X$. We know, thanks to the duality theory proposed by Sinnamon \cite{Si94} (cf. \cite{KMS07}) and Le{\' s}nik and Maligranda
	\cite{LM15a}, that the space $\widetilde{X^{\times}}$ is a (K{\" o}the) dual to $\mathscr{C}X$.
	We take this fact as a \enquote{black box} here.
	However, together with \eqref{Introduction EQ: 1}, it allows us to deduce that
	\begin{equation} \label{Introduction EQ: 2}
		\Vert f \Vert_{\mathscr{C}X}
			\approx \norm{\left\{ 2^{-j} \norm{ f\chi_{\Delta_j}}_{L_1} \right\}_{j \in \mathbb{Z}}}_{\mathbf{E}(X)}.
	\end{equation}
	In other words, $\mathscr{C}X = \bigl( \bigoplus_{j \in \mathbb{Z}} L_{1}(\Delta_j) \bigr)_{{\bf E}(X)(w)}$, where $w(j) \coloneqq 2^{-j}$
	for $j \in \mathbb{Z}$ or, equivalently, $\mathscr{C}X = \bigl( \bigoplus_{j \in \mathbb{Z}} L_{1}(\Delta_j,W_j) \bigr)_{{\bf E}(X)}$,
	where $W_j(t) \coloneqq 2^{-j}$ for $t \in \Delta_j$ and $j \in \mathbb{Z}$ (see, again, Theorems~\ref{Thm: Tandori sequence representation}
	and \ref{Thm: Tandori function representation}).

	Block form representations \eqref{Introduction EQ: 1} and \eqref{Introduction EQ: 2} allow us to look much deeper into the structure of those spaces.
	This, bearing in mind the context of \cite{GE98}, is not in itself a special surprise. From a technical point of view, however, the reason for this can
	be explained as follows: the six functors listed below
	\begin{center}
		\begin{table}[ht]
			\begin{tabular}{ | c | c | c | } 
				\hline \xrowht{15pt}
				\text{K{\" o}the dual} & $\times \colon X \rightsquigarrow X^{\times}$ & see Section~\ref{SUBSECTION: Foreground roles} \\
				\hline \xrowht{15pt}
				\text{interpolation functor} & ${\bf F} \colon \vv{X} \rightsquigarrow {\bf F}(\vv{X})$ & see Section~\ref{SECTION: interpolation} \\
				\hline \xrowht{15pt}
				\text{$p$-convexification/concavification} & $(p) \colon X \rightsquigarrow X^{(p)}$ & see Section~\ref{SECTION: ri spaces} \\
				\hline \xrowht{15pt}
				\text{pointwise multipliers} & $M \colon (X,Y) \rightsquigarrow M(X,Y)$ & see Section~\ref{SECTION: pointwise multipliers} \\
				\hline \xrowht{15pt}
				\text{pointwise products} & $\odot \colon (X,Y) \rightsquigarrow X \odot Y$ & see Section~\ref{SECTION: pointwise products} \\
				\hline \xrowht{15pt}
				\text{symmetrization} & $\bigstar \colon X \rightsquigarrow X^{\bigstar}$ & see Section~\ref{SECTION: symmetization} \\
				\hline
			\end{tabular}
		\end{table}
	\end{center}
	commute (with some minor exceptions and, maybe, additional assumptions) with
	\begin{center}
		\begin{table}[ht]
			\begin{tabular}{ | c | c | c | } 
				\hline \xrowht{15pt}
				amalgam & $\mathbf{A} \colon \left( J, X_j, E \right)
					\rightsquigarrow \mathbf{A}\left( J, X_j, E \right) \coloneqq \bigl( \bigoplus_{j \in J} X_j \bigr)_{E}$ & see Definition~\ref{DEF: direct sum} \\
				\hline \xrowht{15pt}
				discretization & ${\bf E} \colon X \rightsquigarrow {\bf E}(X)$ & see Definition~\ref{DEF: E(X)} \\
				\hline
			\end{tabular}
		\end{table}
	\end{center}
	More precisely, we have
	\begin{center}
		\begin{table}[ht]
			\begin{tabular}{ | c | c | } 
				\hline \xrowht{15pt}
				$\mathbf{F}$ (sometimes) commute with $\mathbf{A}$ & see Theorem~\ref{THM: a'la Bukhvalov} (cf. Remark~\ref{REMARK: Maligranda on CL-construction}) \\
				\hline \xrowht{15pt}
				Calder{\' o}n product commute with $\mathbf{A}$ & see Theorem~\ref{THM: Calderon product of amalgams} \\
				\hline \xrowht{15pt}
				$M$ commute with $\mathbf{A}$ & see Theorem~\ref{Thm: komutowanie M z cdot} (see also Proposition~\ref{Prop: Podstawowe wlasnosci sum prostych}) \\
				\hline \xrowht{15pt}
				$\odot$ commute with $\mathbf{A}$ & see Theorem~\ref{PROP: Pointwise multipliers of amalgams} \\
				\hline \xrowht{15pt}
				$(p)$ commute with $\mathbf{A}$ & see the proof of Theorem~\ref{PROP: Pointwise multipliers of amalgams} \\
				\hline
			\end{tabular}
		\end{table}
	\end{center}
	and, moreover,
	\begin{center}
		\begin{table}[ht]
			\begin{tabular}{ | c | c | }
				\hline \xrowht{15pt}
				$\bigstar$ annihilates with $\mathbf{E}$ & see Lemma~\ref{PROP: EX <-> carrier} \\
				\hline \xrowht{15pt}
				$\mathbf{F}$ commute with $\mathbf{E}$ & see Theorem~\ref{Proposition: E komutuje z interpolacja} \\
				\hline \xrowht{15pt}
				$M$ (sometimes) commute with $\mathbf{E}$ & see the proof of Theorem~\ref{Cor: multipliers between Cesaro and Tandori} (cf. Remark~\ref{REMARK: M komutuje z E}) \\
				\hline \xrowht{15pt}
				$\odot$ commute with $\mathbf{E}$ & see Theorem~\ref{PROP: E komutuje z produktem} \\
				\hline \xrowht{15pt}
				$(p)$ commute with $\mathbf{E}$ & see the proof of Theorem~\ref{PROP: Tandori komutuje z produktem} \\
				\hline
			\end{tabular}
		\end{table}
	\end{center}	
	
	Once established, the above (rather incomplete) list of rules create a kind of {\it arithmetic} that allows us to obtain the most important
	results of Sections~\ref{SECTION: Interpolation structure}, \ref{SECTION: Products and factors: Grosse-Erdmann's style} and \ref{SECTION: Factorization}
	(cf. Synopsis~2, 3 and 4).
	We would like to briefly say more about this.
	However, due to the accumulation of technical details, we decided to limit ourselves to illustrating the general method with a rather simple example
	(which, however, is general enough, because can be seen as an \enquote{integral} analogue of Bennett's factorization \eqref{INTRODUCTION: Bennett factorization of H};
	see \cite[Proposition~1]{AM09} and Theorem~\ref{THM: factorization CX with Lp}).
	
	\begin{example}[S. V. Astashkin and L. Maligranda, 2009] \label{EXAMPLE: AM introduction}
		Fix $p > 1$. To factorize the integral version of Hardy's inequality $L_p \hookrightarrow Ces_p$ we need to find a space, say $\spadesuit$, such that
		\begin{equation} \label{INTRODUCTION: factorization AM}
			L_p \odot \spadesuit = Ces_p.
		\end{equation}
		We know that $\spadesuit = M(L_p, Ces_p)$. The point is, however, to find a more explicit description of $\spadesuit$.
		We will do this in three steps.
		
		{\bf Step 1. Discretization.} First, we need to represent both spaces $L_p$ and $Ces_p$ as a direct sums. Although, quite trivially,
		$L_p = \bigl( \bigoplus_{j \in \mathbb{Z}} L_p(\Delta_j) \bigr)_{\ell_p}$, to discretize $Ces_p$ we need to invoke \eqref{Introduction EQ: 2}.
		Anyway, $Ces_p = \bigl( \bigoplus_{j \in \mathbb{Z}} L_1(\Delta_j, W_j) \bigr)_{{\bf E}(L_p)}$, where $W_j(t) = 2^{-j}$ for $t \in \Delta_j$ and $j \in \mathbb{Z}$.
		
		{\bf Step 2. Arithmetic.} Now we have to do some calculations. We will use some of the above-mentioned results regarding the constructions
		$\bigoplus$ and ${\bf E}$ (along with the well-known facts about the functors $M$ and $\odot$; see \cite{KLM14}, \cite{KLM19} and \cite{Sch10}).
		We have
		\begin{align*}
			M(L_p, Ces_p)
				& = M\left[ \Bigl( \bigoplus_{j \in \mathbb{Z}} L_p(\Delta_j) \Bigr)_{\ell_p}, \Bigl( \bigoplus_{j \in \mathbb{Z}} L_1(\Delta_j, W_j) \Bigr)_{{\bf E}(L_p)} \right]
					\quad (\text{by Step 1}) \\
				& = \Bigl( \bigoplus_{j \in \mathbb{Z}} M(L_p(\Delta_j), L_1(\Delta_j, W_j)) \Bigr)_{M(\ell_p, {\bf E}(L_p))} \quad (\text{since $M$ commute with $\mathbf{A}$ and $\mathbf{E}$}).
		\end{align*}
		Let us divide our attention in two. On the one hand,
		\begin{align*}
			M(L_p(\Delta_j), L_1(\Delta_j, W_j))
				& \equiv M(L_p, L_1)(\Delta_j,W_j) \quad (\text{since $M(X,Y(W)) = M(X,Y)(W)$}) \\
				& \equiv [L_p]^{\times}(\Delta_j,W_j) \quad (\text{because $M(X,L_1) \equiv X^{\times}$}) \\
				& \equiv L_q(\Delta_j,W_j),
		\end{align*}
		where $q$ is the conjugate exponent to $p$, that is, $1/p + 1/q = 1$. On the other,
		\begin{align*}
			M(\ell_p, {\bf E}(L_p))
				& \equiv M(\ell_1^{(p)}, {\bf E}(L_1)^{(p)}) \quad (\text{since $\mathbf{E}$ commute with $(p)$}) \\
				& \equiv M(\ell_1,{\bf E}(L_1))^{(p)} \quad (\text{because $M$ commute with $(p)$}) \\
				& \equiv M(\ell_1, \ell_1(w))^{(p)} \quad (\text{of course, ${\bf E}(L_1) \equiv \ell_1(w)$, where $w(j) = 2^{j}$}) \\
				& \equiv [M(\ell_1,\ell_1)(w)]^{(p)} \quad (\text{because $M(X,Y(w)) = M(X,Y)(w)$}) \\
				& \equiv \ell_{\infty}(w^{1/p}) \quad (\text{since $M(X,X) = L_{\infty}$}),
		\end{align*}
		where $w^{1/p}(j) = 2^{j/p}$ for $j \in \mathbb{Z}$. Putting things together, we get
		\begin{align*}
			M(L_p, Ces_p)
				& = \Bigl( \bigoplus_{j \in \mathbb{Z}} L_q(\Delta_j,W_j) \Bigr)_{\ell_{\infty}(w^{1/p})}\\
				& = \Bigl( \bigoplus_{j \in \mathbb{Z}} L_q(\Delta_j,W_j^{1/q}) \Bigr)_{\ell_{\infty}} \quad (\text{since $2^{-j}2^{j/p} = 2^{-j(1-1/p)} = 2^{-j/q}$}),
		\end{align*}
		where $W_j^{1/q}(t) = 2^{-j/q}$ for $t \in \Delta_j$ and $j \in \mathbb{Z}$.
	
		{\bf Step 3. Anti-discretization.} It remains to describe the space $\bigl( \bigoplus_{j \in \mathbb{Z}} L_q(\Delta_j,W_j^{1/q}) \bigr)_{\ell_{\infty}}$
		in the language of known constructions. Or, to put it another way, reverse the discretization process. For this, one can observe that
		\begin{align*}
			M(L_p, Ces_p)
			& = \Bigl( \bigoplus_{j \in \mathbb{Z}} L_q(\Delta_j,W_j^{1/q}) \Bigr)_{\ell_{\infty}} \quad (\text{by Step 2}) \\
			& = \Bigl( \bigoplus_{j \in \mathbb{Z}} [L_1(\Delta_j,W_j)]^{(q)} \Bigr)_{[\ell_{\infty}]^{(q)}} \quad (\text{since $[\ell_{\infty}]^{(q)} \equiv \ell_{\infty}$}) \\
			& = \left[ \Bigl( \bigoplus_{j \in J} L_1(\Delta_j,W_j) \Bigr)_{{\bf E}(L_{\infty})} \right]^{(q)} \quad (\text{because $\mathbf{A}$ commute with $(q)$}) \\
			& = [Ces_{\infty}]^{(q)} \quad (\text{again, using \eqref{Introduction EQ: 2}}).
		\end{align*}
		Consequently, $\spadesuit = M(L_p, Ces_p) = [Ces_{\infty}]^{(q)}$ and we are done.
	\end{example}

	\begin{remark}[About Example~\ref{EXAMPLE: AM introduction}]
		Using the nomenclature from \cite{AM09} (which has its roots in \cite{Be96}), we should denote the space $[Ces_{\infty}]^{(q)}$ as $G_q$.
		For this reason, factorization \eqref{INTRODUCTION: factorization AM} takes the form
		\begin{equation*}
			L_p \odot G_q = Ces_p.
		\end{equation*}
		For historical reasons, the space $G_q$ should most likely be called the {\it Korenblyum, Kre{\u \i}n and Levin space}
		(see \cite{KKL48}; cf. \cite[Example~1.3]{LZ65}). Note, however, that the last step, which we called \enquote{anti-discretization} above,
		was done {\it a posteriori} (in the sense that we already knew what space we were looking for). In general, finding a natural description
		of the space obtained after the discretization process is not at all obvious (this problem was already pointed out by Grosse-Erdman; see \cite[p.~5]{GE98}).
		Later, we will propose an appropriate generalization of the space $G_q$ (see Remark~\ref{REMARK: CqX construction}).
	\end{remark}
	
	\section{{\bf Statements and Declarations}}
	
	\subsection{Acknowledgments}
	The core ideas of this work were born in embryonic form during a coffee break at the conference dedicated to the memory of
	Professor Pawe{\l} Doma{\' n}ski in 2018. Since then, we have informally announced some of our results on at least a few occasions.
	
	We would also like to express our heartfelt thanks to Lech Maligranda for all the materials, especially those from his
	private collection, made available to us. His comments about interpolation proved, as usual, indispensable.
	We also thank Karol Le{\' s}nik for stimulating discussions and valid suggestions about factorization problems.
	Finally, it is also a pleasure to thank Pawe{\l} Kolwicz for all the translations from Russian that he was kind enough
	to do for us.
	
	\subsection{Funding}
	Our research was carried out at the Pozna{\' n} University of Technology (grant number 0213/SBAD/0116).
	
	\subsection{Conflict of Interest}
	The authors declare that they have no known competing financial interest or personal relationship that could have appeared
	to influence the work reported in this paper.
	
	\subsection{Data Availibility}
	All data generated or analysed during this study are included in this article.
	
	\section{{\bf Toolbox}} \label{Section: Toolbox}
	
	Below for the reader's convenience, we provide in detail the most important definitions, terminology and some essential facts which we will use later.
	Other, possibly unfamiliar, definitions and concepts will be introduced in the sections in which they are used.
	
	\subsection{Notation and terminology}
	Most of our notation and terminology deserves to be called standard and is basically a melange of what is used in the monographs by
	Brudny{\u \i} and Krugljak \cite{BK91}, and Lindenstrauss and Tzafriri \cite{LT79}.
	
	\begin{notation}[Vinogradov's notation]
		Given two quantities, say $A$ and $B$, depending (maybe) on certain parameters, we will write $A \approx B$ understanding that there exist
		absolute constants (that is, independent of all involved parameters), say $c, C > 0$, such that $cA \leqslant B \leqslant CA$.
	\end{notation}

	\begin{notation}[Definition]
		Defined terms are always distinguished from the rest of the text by {\bf bold} font (also when, instead of providing a full definition,
		we only refer to the literature). Sometimes (but there is no rule here), we specify them additionally in the form of {\bf Definition}.
	\end{notation}

	\begin{notation}[Digression]
		Some additional comments, (hopefully) interesting but not necessary for the general understanding of this work have been placed in specific
		parts of the text titled {\bf Digression}.
	\end{notation}

	\begin{notation}
		By tradition, we will use \enquote{halmos} $\blacksquare$ at the end of each proof. Anyway, we find it useful to also use the symbol
		{\large $\blacktriangle$} to clearly indicate the end of each Remark and Digression.
	\end{notation}

	\begin{notation}
		Let, as usual, $\mathbb{N} = \{1,2, ...\}$, $\mathbb{Z}$ and $\mathbb{R}$ denote the set of natural numbers, integers and reals, respectively.
		Moreover, we set $\mathbb{Z}_+ = \{0,1,2,...\}$ and $\mathbb{Z}_{-} = \{..., -2, -1\}$.
	\end{notation}
	
	\subsection{Foreground roles} \label{SUBSECTION: Foreground roles}
	Let $(\Omega, \Sigma, \mu)$ be a complete and $\sigma$-finite measure space, that is, $\Sigma$ does not have atoms of infinite measure
	and any subset of a set of measure zero is measurable. Further, let $L_0(\Omega,\Sigma,\mu)$ (briefly, $L_0(\Omega)$ or even $L_0$,)
	be the set consisting of all equivalence classes, by the Kolmogorov quotient modulo equality almost everywhere, of locally integrable
	real or complex valued functions on $\Sigma$.
	
	A Banach space $X$ is called a {\bf Banach ideal space} (in another nomenclature a {\bf K{\" o}the space}) if the following two conditions hold:
	\begin{itemize}
		\item[$\bullet$] $X$ is a linear subspace of $L_0(\Omega,\Sigma,\mu)$;
		\item[$\bullet$] if $\abs{f(\omega)} \leqslant \abs{g(\omega)}$ almost everywhere on $\Omega$, with $f$ measurable
		and $g \in X$, then $f \in X$ and $\norm{f} \leqslant \norm{g}$ (this is the so-called {\bf ideal property}).
	\end{itemize}
	To make things less annoying, we will also assume (unless we expressly write otherwise) that
	\begin{itemize}
		\item[$\bullet$] for every $A \in \Sigma$ with $\mu(A) < \infty$ the characteristic function $\chi_{A}$ belongs to $X$.
	\end{itemize}

	\begin{notation}[Banach function and sequence spaces]
		Hereafter, whenever we use the term {\bf Banach sequence space} we mean a Banach ideal space over a purely atomic measure space,
		while, the name {\bf Banach function space} is reserved for a Banach ideal spaces over a non-atomic measure space. Aside from the
		fact that every sequence is a function, this nomenclature seems to be well established in the existing literature. Anyway, it follows
		from Maharam's theorem that every Banach ideal space is, roughly speaking, a mixture of a certain Banach sequence and Banach function space.
	\end{notation}
	
	\begin{notation}[Continuous inclusions]
		Due to the closed graph theorem, a formal inclusion of two Banach ideal spaces, say $X$ and $Y$, is a bounded operator and to clearly
		emphasize this fact we will write $X \hookrightarrow Y$ instead of just $X \subset Y$. Moreover, in some situations, we will use the
		symbol $X \overset{C}{\hookrightarrow} Y$ to indicate that the norm of the inclusion $X \subset Y$ is not greater than $C > 0$.
	\end{notation}

	\begin{notation}
		The symbols $X = Y$ and $X \equiv Y$ mean that the spaces $X$ and $Y$ are the same as sets and their norms are equivalent, respectively, equal.
		Moreover, we will write $X \approx Y$ if the spaces $X$ and $Y$ are isomorphic.
	\end{notation}
	
	Given a Banach ideal space $X$ defined on $(\Omega,\Sigma,\mu)$ the {\bf K{\" o}the dual space $X^{\times}$} (or the {\bf associated space})
	of $X$ is defined as a vector space of all measurable functions $f$ such that $\int_{\Omega} \abs{fg} d\mu$ is finite for all $g \in X$,
	equipped with the operator norm $\norm{f}_{X^{\times}} \coloneqq \sup_{\norm{g}_X \leqslant 1} \int_{\Omega} \abs{fg} d\mu$.
	Recall that always we have the inclusion $X \overset{1}{\hookrightarrow} X^{\times \times}$, but $X \equiv X^{\times \times}$ if, and only if,
	the norm in $X$ has the so-called {\bf Fatou property}, meaning that the closed unit ball in $X$ is also closed in $L_0(\Omega,\Sigma,\mu)$
	with respect to the topology of local convergence in measure.
	
	A Banach ideal space $X$ is said to be {\bf rearrangement invariant space} (briefly, {\bf r.i. space}) if it posses the following additional property,
	namely, for any two measurable functions, say $f$ and $g$, such that
	$\mu\left( \left\{ \omega \in \Omega \colon \abs{f(\omega)} > \lambda \right\} \right) = \mu\left( \left\{ \omega \in \Omega \colon \abs{g(\omega)} > \lambda \right\} \right)$
	for all $\lambda \geqslant 0$ (such functions are called {\bf equimeasurable}) and $f \in X$, we have $g \in X$ and $\norm{f}_X = \norm{g}_X$.
	In particular, $\norm{f}_X = \norm{f^{\star}}_X$, where $f^{\star}$ is the {\bf non-increasing rearrangement} of $f$, that is,
	$f^{\star}(t) \coloneqq \inf \left\{ \lambda > 0 \colon \mu\left( \left\{ \omega \in \Omega \colon \abs{f(\omega)} > \lambda \right\} \right) \leqslant t \right\}$
	for $t > 0$.
	
	\begin{notation}[Rearrangement invariant sequence and function spaces]
		Due to Luxemburg's representation theorem (see \cite[Theorem~4.10, p.~62]{BS88}) it is enough to consider only r.i. spaces defined either
		on a set of positive integers $\mathbb{N}$ with the counting measure $\#$ or a half-line $(0,\infty)$ with the usual Lebesgue measure $m$.
		Henceforth, we will refer to this situations by writing that $X$ is a {\bf rearrangement invariant sequence space} or, respectively,
		a {\bf rearrangement invariant function space}.
	\end{notation}
	
	We refer to the books by Bennett and Sharpley \cite{BS88}, Kre{\u \i}n, Petunin and Semenov \cite{KPS82}, Lindenstrauss and Tzafriri \cite{LT79}
	and Maligranda \cite{Mal04} for a much more complete treatment of Banach ideal spaces and, in particular, rearrangement invariant spaces.
	
	\subsection{Kalton's zone}
	Although we will not deal with quasi-Banach spaces {\it per se}, the road to results about Banach spaces often leads through
	a little more exotic world of quasi-norms and quasi-Banach spaces. In this topic, however, we will limit ourselves to the bare
	minimum while referring to \cite{Kal03} for more.
	
	A {\bf quasi-norm} $\norm{\cdot}$ on a given, real or complex, vector space $X$ is a functional $x \rightsquigarrow \norm{x}$ which
	satisfies the usual norm axioms, except that the familiar triangle inequality is replaced by a more perverse version of it, namely,
	\begin{equation*}
		\norm{f + g} \leqslant \Delta \bigl(\norm{f} + \norm{g} \bigr)
	\end{equation*}
	for all $f, g \in X$ and with some constant\footnote{Usually called the {\bf modulus of concavity} of $\norm{\cdot}$.} $\Delta \geqslant 1$
	independent of $f$ and $g$. A {\bf quasi-normed space} is a vector space equipped with a quasi-norm. The quasi-norm $\norm{\cdot}$ induces
	a linear topology on $X$, namely, the sets $\left\{f \in X \colon \norm{f} < \frac{1}{n} \right\}$ with $n \in \mathbb{N}$ form a countable
	base of neighbourhoods of the origin. Therefore, $(X,\norm{\cdot})$ is a locally bounded topological vector space.
	
	Thanks to the Aoki--Rolewicz theorem, each quasi-norm $\norm{\cdot}$ with the modulus of concavity $\Delta$ can be equivalently replaced by
	the so-called {\bf $p$-norm} $\sharp \, \cdot \, \sharp$, where $0 < p \leqslant 1$ is given as a solution to the equation $\Delta = 2^{1/p - 1}$,
	given as
	\begin{equation*}
		\sharp \, f \, \sharp \coloneqq \inf \left\{ \left( \sum_{i=1}^n \norm{f_i}^p \right)^{1/p} \colon f = \sum_{i=1}^n f_i \right\}.
	\end{equation*}
	In this situation, for two measurable functions $f$ and $g$, we have
	\begin{equation*}
		\sharp \, f + g \, \sharp^p \leqslant \sharp \, f \, \sharp^p + \sharp \, g \, \sharp^p.
	\end{equation*}
	Consequently, the above $p$-norm $\sharp \, \cdot \, \sharp$ defines an invariant metric by $d(f,g) \coloneqq \sharp \, f - g \, \sharp^p$.
	In other words, quasi-normed spaces are metrizable. A quasi-normed space $(X,\norm{\cdot})$ is called a {\bf quasi-Banach space}
	if it is complete for this metric.
	By replacing the concept of \enquote{norm} with \enquote{quasi-norm} (in the appropriate definitions from Section~\ref{SUBSECTION: Foreground roles}),
	it is not difficult to guess the proper meaning hidden under the term {\bf quasi-Banach ideal space} or {\bf rearrangement invariant quasi-Banach space}.
	
	\begin{warning}
		The unit ball in the quasi-Banach space need no longer be convex, thus tools such as the Hahn--Banach theorem fall out of the available arsenal.
		For this reason, a quasi-Banach space might have trivial dual (just like $L_p$ with $0 < p < 1$) and a quasi-norm can be discontinuous with respect
		to the self-induced topology\footnote{The use of an equivalent $p$-norm eliminates this disadvantage.}. However, the closed graph theorem,
		the open mapping theorem and Banach--Steinhaus's theorem are still available.
	\end{warning}
	
	\subsection{Boyd indices} Let $s > 0$. For a measurable function, say $f$, on $(0,\infty)$ the {\bf dilation operator}
	$D_s$ is defined as $D_s \colon f \rightsquigarrow [t \rightsquigarrow f(t/s)]$.
	Dilation operators $D_s$ are bounded on any r.i. function space $X$ with
	\begin{equation} \label{INEQ: D_s < max 1,s}
		\norm{D_s}_{X \rightarrow X} \leqslant \max\{1,s\}
	\end{equation}
	(see \cite[p.~148]{BS88}, \cite[pp.~96--98]{KPS82} or \cite[pp.~130--131]{LT79}). However, they can also be bounded on some Banach function
	spaces which are {\it not} rearrangement invariant. For example, $D_s$ is bounded on $L_p(w)$ with $\norm{D_s}_{L_p(w) \rightarrow L_p(w)} \leqslant s^{1/p + \alpha}$
	for all $s > 0$. Here, $w(t) = t^{\alpha}$, $1 \leqslant p < \infty$ and $-1/p < \alpha < 1 - 1/p$.
	
	Let $X$ be a quasi-Banach function space with the property that the dilation operators $D_s$ are bounded on $X$ for all $s > 0$.
	Then the {\bf Boyd indices} $\alpha_X$ and $\beta_X$ of $X$ are defined by\footnote{With the convention that, if $\norm{D_s}_{X \rightarrow X} = 1$
	for some $s > 1$, we put $\alpha_X = \infty$, and, similarly, if $\norm{D_s}_{X \rightarrow X} = 1$ for all $0 < s < 1$, we put $\beta_X = \infty$.}
	\begin{equation*}
		\alpha_X \coloneqq \lim\limits_{s \rightarrow \infty} \frac{\log s}{\log \norm{D_s}_{X \rightarrow X}} = \sup\limits_{s > 1} \frac{\log s}{\log \norm{D_s}_{X \rightarrow X}}
	\end{equation*}
	and
	\begin{equation*}
		\beta_X \coloneqq \lim\limits_{s \rightarrow 0^{+}} \frac{\log s}{\log \norm{D_s}_{X \rightarrow X}} = \inf\limits_{0 < s < 1} \frac{\log s}{\log \norm{D_s}_{X \rightarrow X}}.
	\end{equation*}
	By the very definition, $0 < \alpha_X \leqslant \beta_X \leqslant \infty$.
	Moreover, for any $\varepsilon > 0$ there is $C = C(\varepsilon,X) > 0$ such that for all $f \in X$ we have
	\begin{equation} \label{EQ: D_s < CMAXs^1/p,s^1/q}
		\norm{D_s}_{X \rightarrow X} \leqslant C \max \left\{ s^{\frac{1}{\alpha_X}+\varepsilon}, s^{\frac{1}{\beta_X}-\varepsilon} \right\}.
	\end{equation}
	Additionally, if $X$ is a r.i. function space, then
	\begin{equation} \label{EQ: Boydy X i dualu X}
		\frac{1}{\alpha_X} + \frac{1}{\beta_{X^{\times}}} = 1 = \frac{1}{\beta_X} + \frac{1}{\alpha_{X^{\times}}}
	\end{equation}
	(see \cite[Proposition~2.b.2, p.~131]{LT77} and \cite[Theorem~4.11, p.~106]{KPS82}). Let us also mention that for $X = L_p$
	with $0 < p \leqslant \infty$ we have\footnote{In several other places in the literature, headed by Boyd's article \cite{Bo69}
	itself, the indices of $X$ are taken as reciprocals of the above introduced indices $\alpha_X$ and $\beta_X$; exactly as
	in \cite[p.~131]{LT79}, we belive that our choice is more natural.} $\alpha_X = \beta_X = p$.
	
	As mentioned in \cite[p.~131]{LT79} (see also \cite[p.~165]{KPS82}), the Boyd indices can be defined also for sequence spaces.
	The only difference is that in this case the dilation operators $D_m$ and $D_{1/m}$ with $m \in \mathbb{N}$ are defined in the following way
	\begin{equation*}
		D_m(x) \coloneqq \left\{ x_{[\frac{m-1+n}{m}]} \right\}_{n=1}^\infty = \{ \underbrace{x_1, x_1, ..., x_1}_{m}, \underbrace{x_2, x_2, ...,x_2}_{m}, x_3, ... \}
	\end{equation*}
	and
	\begin{equation*}
		D_{1/m}(x) \coloneqq \left\{ \frac{1}{m} \sum_{k=(n-1)m+1}^{nm} x_k \right\}_{n=1}^\infty
			= \left\{ \frac{1}{m}\sum_{k=1}^m x_k, \frac{1}{m}\sum_{k=m+1}^{2m} x_k, ..., \frac{1}{m}\sum_{(n-1)m+1}^{nm}x_k, ... \right\}.
	\end{equation*}
	
	\begin{notation}
		The Boyd indices of a given (quasi-)Banach ideal space $X$ are said to be {\bf non-trivial} provided $\alpha_X > 1$ and $\beta_X < \infty$.
	\end{notation}
	
	\begin{remark}
		The usefulness of Boyd's indices is mainly due the fact that they make it possible to reduce the characterization of r.i.
		spaces on which certain known operators are bounded to simple inequalities. Of the most famous examples, let us mention the
		following ones:
		\begin{itemize}
			\item[$\bullet$] Hilbert's transform $H \colon f \rightsquigarrow [x \rightsquigarrow \frac{1}{\pi}\text{ p.v.} \int_{0}^{\infty} \frac{f(t)}{x-t}dt]$
			is bounded on a r.i. function space $X$ if, and only if, Boyd's indices of $X$ are non-trivial;
			\item[$\bullet$] Hardy--Littlewood's maximal operator
			$\mathscr{M} \colon f \rightsquigarrow [x \rightsquigarrow \sup_{\varepsilon > 0} \frac{1}{2\varepsilon} \int_{x-\varepsilon}^{x+\varepsilon} \abs{f(t)}dt]$
			is bounded on a r.i. function space $X$ if, and only if, $\alpha_X > 1$;
			\item[$\bullet$] Hardy's operator $\mathscr{H} \colon f \rightsquigarrow [x \rightsquigarrow \frac{1}{x}\int_0^x f(t)dt]$ is bounded on a r.i. function space $X$ if,
			and only if, $\alpha_X > 1$;
			\item[$\bullet$] Copson's operator $\mathscr{C} \colon f \rightsquigarrow [x \rightsquigarrow \int_x^{\infty} \frac{f(t)}{t}dt]$ is bounded on a r.i. function space $X$ if,
			and only if, $\beta_X < \infty$.
		\end{itemize}
		For the details we refer to \cite{BS88} and \cite{KPS82}.
		\demo
	\end{remark}
	
	The {\bf lower dilation exponent $\gamma_\varphi$} and the {\bf upper dilation exponent $\delta_\varphi$} of an arbitrary positive function
	$\varphi$ on $(0,\infty)$ are defined as
	\begin{equation} \label{dilation indices}
		\gamma_\varphi \coloneqq \lim\limits_{s \rightarrow \infty} \frac{\log s}{ \log \, \sup_{t > 0} \frac{\varphi(st)}{\varphi(t)} }
	\end{equation}
	and
	\begin{equation}
		\delta_\varphi \coloneqq \lim\limits_{s \rightarrow 0^{+}} \frac{\log s}{ \log \, \sup_{t > 0} \frac{\varphi(st)}{\varphi(t)} }.
	\end{equation}
	For a quasi-concave function\footnote{A non-negative function, say $\varphi$, defined on the half-line $(0,\infty)$ is said to be {\bf quasi-concave}
	if $\varphi$ is non-decreasing, vanishes only at zero and $t \rightsquigarrow \frac{\varphi(t)}{t}$ is non-increasing. Clearly, every non-negative concave function
	on $(0,\infty)$ that vanishes only at the origin is quasi-concave, but not {\it vice versa} (see \cite[p.~69]{BS88}).} $\varphi$ on $(0,\infty)$,
	we have $1 \leqslant \gamma_\varphi \leqslant \delta_\varphi \leqslant \infty$ (see \cite[Theorem~4.12, p.~107]{KPS82} or \cite[pp.~81--95]{Ma89}).
	
	\subsection{Elements on interpolation spaces} \label{SECTION: interpolation}
	Recall that the pair $\vv{X} \coloneqq (X_0,X_1)$ of Banach spaces is said to form a {\bf Banach couple} (briefly, a {\bf couple}) if both spaces $X_0$ and $X_1$
	can be linearly and continuously embedded into a common Hausdorff topological vector space (see \cite[2.9.3, p.~37]{BL76}).
	A Banach space $X$ is called an {\bf intermediate space} for a couple
	$\vv{X}$ if
	\begin{equation*}
		\Delta(\vv{X}) \hookrightarrow X \hookrightarrow \Sigma(\vv{X}),
	\end{equation*}
	where the {\bf intersection space} $\Delta(\vv{X}) \coloneqq X_0 \cap X_1$ is equipped with the norm $\norm{f}_{\Delta(\vv{X})} \coloneqq \max\left\{ \norm{f}_{X_0},\norm{f}_{X_1} \right\}$,
	while the {\bf sum space} $\Sigma(\vv{X}) \coloneqq X_0 + X_1$ is equipped with the norm $\norm{f}_{\Sigma(\vv{X})} \coloneqq \inf\left\{ \norm{g}_{X_0} + \norm{h}_{X_1} \right\}$,
	where the infimum taken over all decompositions of the form $f = g + h$ with $g \in X_0$ and $h \in X_1$.
	
	An operator $T \colon \Sigma(\vv{X}) \rightarrow \Sigma(\vv{Y})$ is called a {\bf bounded operator from a couple $\vv{X}$ into the couple $\vv{Y}$}
	if $T$ acts as a bounded operator from $X_0$ into $Y_0$ and from $X_1$ into $Y_1$. In this situation, we will briefly write $T \colon \vv{X} \rightarrow \vv{Y}$.
	Moreover, the norm of $T \colon \vv{X} \rightarrow \vv{Y}$ is defined to be
	\begin{equation*}
		\norm{T}_{\vv{X} \rightarrow \vv{Y}} \coloneqq \max \left\{ \norm{T}_{X_0 \rightarrow Y_0}, \norm{T}_{X_1 \rightarrow Y_1} \right\}.
	\end{equation*}

	The intermediate space $X$ for a couple $\vv{X}$ is called an {\bf interpolation space} relative to the couple $\vv{X}$ if for any operator
	$T \colon \Sigma(\vv{X}) \rightarrow \Sigma(\vv{X})$ we have $T(X) \subset X$. The closed graph theorem guarantees that
	$\norm{T}_{X \rightarrow X} \leqslant C \norm{T}_{\vv{X} \rightarrow \vv{X}}$ for some $C \geqslant 1$. If $C = 1$ we say that $X$ is an
	{\bf exact interpolation space}.
	
	A functor $\mathbf{F}$ acting from the category of all Banach couples into the category of Banach spaces is called an {\bf interpolation functor}
	if for every couple $\vv{X}$ a Banach space $F(\vv{X})$ is an intermediate space for $\vv{X}$ and $\mathbf{F}(T) = T$ for every $T \colon \vv{X} \rightarrow \vv{Y}$.
	The simplest interpolation functors are $\Delta \colon \vv{X} \rightsquigarrow \Delta(\vv{X})$, $P_0 \colon \vv{X} \rightsquigarrow X_0$,
	$P_1 \colon \vv{X} \rightsquigarrow X_1$ and $\Sigma \colon \vv{X} \rightsquigarrow \Sigma(\vv{X})$.
	
	To give other examples of interpolation functors, recall that {\bf Peetre's $\mathscr{K}$-functional} of a couple $\vv{X} = (X_0,X_1)$
	is defined for $f \in \Sigma(\vv{X})$ and $t > 0$ as follows
	\begin{equation*}
		\mathscr{K}(f,t;X_0,X_1) \coloneqq \inf\left\{ \norm{f_0}_{X_0} + \norm{f_1}_{X_1} \colon f = f_0 + f_1 \text{ with } f_0 \in X_0 \text{ and } f_1 \in X_1 \right\}.
	\end{equation*}
	Moreover, a {\bf parameter of the real $\mathscr{K}$-method} is a Banach function space, say $\mathscr{X}$, defined over the half-line
	$(0,\infty)$ equipped with the multiplicative Haar measure $dt/t$ that contains the function $t \rightsquigarrow \min\{1,t\}$.
	Then the {\bf real $\mathscr{K}$-method} with parameter $\mathscr{X}$ is the interpolation functor that takes the couple $\vv{X}$
	to the space $(X_0,X_1)^{\mathscr{K}}_{\mathscr{X}}$ consisting of all $f \in \Sigma(\vv{X})$ such that
	\begin{equation*}
		\norm{f}_{(X_0,X_1)^{\mathscr{K}}_{\mathscr{X}}} \coloneqq \norm{t \rightsquigarrow \mathscr{K}(f,t;X_0,X_1)}_{\mathscr{X}} < \infty.
	\end{equation*}
	In particular, the choice of the weighted $L_p$-space $L_p(w)$, where $w(t) = t^{-\theta}$ for $t > 0$, $0 < \theta < 1$ and $1 \leqslant p \leqslant \infty$,
	in place of $\mathscr{X}$ leads to the well-known {\bf Lions--Petree real method} $(\cdot,\cdot)_{\theta,p}$.

	An interpolation functor $\mathbf{F}$ is called the {\bf $\mathscr{K}$-method functor} if there exists a Banach function space, say $\mathscr{X}$,
	as above such that $\mathbf{F}(X_0,X_1) = (X_0,X_1)^{\mathscr{K}}_{\mathscr{X}}$.
	
	Another celebrated interpolation functor is the so-called {\bf Calder{\' o}n first method} or the {\bf lower method} (simply, the {\bf complex method})
	$[\cdot,\cdot]_{\theta}$ with $0 \leqslant \theta \leqslant 1$. We will not remind its construction here, because in practice, that is, when
	we actually want to calculate something, it boils down to the {\it Calder{\' o}n product}\footnote{Below we will recall the construction of
	an even more general space, the so-called {\it Calder{\' o}n--Lozanovski{\u \i} space} $\varrho(X,Y)$ (for this, see \eqref{INEQ: CL construction}).}
	$X^{1-\theta}Y^{\theta}$: if $X$ and $Y$ are two complex Banach ideal spaces defined on the same measure space, then
	\begin{equation*}
		[X,Y]_{\theta} \equiv X^{1-\theta}Y^{\theta}
	\end{equation*}
	provided one of the spaces $X$ or $Y$ has an order continuous norm (see \cite{Cal64}).
	
	\begin{remark}
		For completeness, let us also note that the complex method is an offspring of Olof Thorin's proof of Marcel Riesz's \enquote{convexity theorem} from 1939
		(see \cite[pp.~34--36]{BK91} and \cite[Note 1.7.3 on p.~20]{BL76} and \cite[pp.~183--198]{BS88}), whereas the real method is a reincarnation of the J{\' o}zef
		Marcinkiewicz's ideas in the form of the $\mathscr{K}$-functional proposed by Jaak Petree in 1963 (interestingly, Marcinkiewicz also announced his result
		in 1939; see \cite[Chapter~1]{BK91}, \cite[Chapters 1 and 3]{BL76} and \cite[pp.~216--228]{BS88} for more information on the evolution of this circle of ideas).
		\demo
	\end{remark}
	
	Much more detailed study of interpolation spaces, interpolation functors and interpolation methods can be found, for example, in the books by
	Bergh and L{\" o}fstr{\" o}m \cite{BL76}, Brudny{\u \i} and Krugljak \cite{BK91}, and Triebel \cite{Tri78} (see also Bennett and Sharpley \cite{BS88},
	Kre{\u \i}n, Petunin and Semenov \cite{KPS82} and Maligranda \cite{Mal04}).
	
	\subsection{Classical protagonists of rearrangement invariant spaces} \label{SECTION: ri spaces}
	Let $1 \leqslant p < \infty$ and $1 \leqslant q < \infty$. Recall that the {\bf Lorentz space} $L_{p,q}(0,\infty) = L_{p,q}$,
	consists of all equivalence classes, modulo equality almost everywhere, of real-valued Lebesgue measurable functions such that
	\begin{equation} \label{NORM Lpq}
		\norm{f}_{L_{p,q}}
			\coloneqq \left[ \int_0^{\infty} \left( t^{1/p}f^{\star}(t) \right)^{q} \frac{dt}{t} \right]^{1/q}
			= \left[ \frac{p}{q} \int_0^{\infty} \left( f^{\star}(t) \right)^{q} d\left( t^{q/p} \right) \right]^{1/q} < \infty
	\end{equation}
	
	Complementary, the {\bf weak $L_p$-space}, denoted $L_{p,\infty}(0,\infty) = L_{p,\infty}$, is the collection of all functions as above so that
	\begin{equation} \label{NORM Lpinfty}
		\norm{f}_{L_{p,\infty}} \coloneqq \sup\limits_{t > 0} t^{1/p} f^{\star}(t)  < \infty.
	\end{equation}

	In general, both expressions (\ref{NORM Lpq}) and (\ref{NORM Lpinfty}) are not necessarily a norms but merely
	a quasi-norms, giving rise to {\it quasi}-Banach spaces rather than traditionally favored Banach spaces. Nevertheless,
	thanks to Hardy's inequality, at least when $1 < p < \infty$ and $1 \leqslant q < \infty$, these quasi-norms are
	equivalent to norms (see, for example, \cite[Theorem~1.9.12, p.~59]{BK91}; cf. \cite[Excersise~7, p.~16]{BL76}).
	
	Of course, $L_{p,p} \equiv L_p$. The spaces $L_{p,q}$ possess the Fatou property and, moreover, when $1 \leqslant q \leqslant p < \infty$ or $p = q = \infty$,
	are rearrangement invariant with both Boyd indices $\alpha_{L_{p,q}}$ and $\beta_{L_{p,q}}$ equal to $p$ (see \cite[Theorem~1.9.9, p.~55]{BK91}).
	The scale of Lorentz spaces $\left\{ L_{p,q} \right\}_{1 \leqslant q < \infty}$ \enquote{increase} as the exponent $q$ increases
	(see \cite[formula (1.9.17), p.~58]{BK91}; cf. \cite[pp.~142--143]{LT79}), that is, for $1 \leqslant p \leqslant \infty$ and $1 \leqslant q < r \leqslant \infty$ we have
	\begin{equation*}
		L_{p,1} \hookrightarrow L_{p,q} \overset{C}{\hookrightarrow} L_{p,r} \hookrightarrow L_{p,\infty},
	\end{equation*}
	where $C = \left( q/p \right)^{1/q-1/r}$.
	Moreover, the Lorentz spaces $L_{p,q}$ have a natural description as an interpolation space obtained via the Lions--Peetre method applied
	to a couple of Lebesgue spaces; more precisely (see, for example, \cite[Theorem~5.2.1, p.~109]{BL76}), for $1 \leqslant p,q,r,s \leqslant \infty$
	with $r < s$ and $1/p = (1-\theta)/r + \theta/s$, we have
	\begin{equation*}
		L_{p,q} = (L_r,L_s)_{\theta,q}.
	\end{equation*}

	Given a concave function $\varphi \colon (0,\infty) \rightarrow (0,\infty)$ which is non-negative and vanishes only at zero, we define
	the {\bf Lorentz space} $\Lambda_{\varphi}(0,\infty) = \Lambda_{\varphi}$ to be the space of all measurable functions, say $f$, for which
	the following Riemann--Stieltjes integral
	\begin{equation*}
		\norm{f}_{\Lambda_{\varphi}}
			\coloneqq \int_0^{\infty} f^{\star}(t)d\varphi(t)
			= \norm{f}_{L_{\infty}} \varphi(0^+) + \int_0^{\infty} f^{\star}(t)\varphi'(t)dt
	\end{equation*}
	is finite, where $\varphi(0^+) \coloneqq \lim_{t \rightarrow \infty} \varphi(t)$ and $\varphi'$ is the derivative\footnote{Since the function $\varphi$
	is concave, so its derivative $\varphi'$ exists outside a countable set of points.} of $\varphi$. Of course, if $\varphi(t) = t^{1/p}$ with $1 \leqslant p < \infty$,
	then the space $\Lambda_{\varphi}$ coincide, up to the equality of norms, with the Lorentz space $L_{p,1}$ introduced above.
	
	Likewise, for a quasi-concave function $\varphi \colon (0,\infty) \rightarrow (0,\infty)$, by the {\bf Marcinkiewicz space} $M_{\varphi}(0,\infty) = M_{\varphi}$
	we understand the vector space of all functions $f$ as above, with
	\begin{equation*}
		\norm{f}_{M_{\varphi}} \coloneqq \sup\limits_{t > 0} \frac{\varphi(t)}{t} \int_0^t f^{\star}(s)ds  < \infty.
	\end{equation*}
	Due to Boyd's theorem \cite{Bo69}, Hardy's operator $f \rightsquigarrow \left[ x \rightsquigarrow \frac{1}{x}\int_0^{x} f(t)dt \right]$
	is bounded on $M_{\varphi}$ with $\alpha_{M_{\varphi}} = \gamma_{\varphi} > 1$ and, in consequence,
	\begin{equation*}
		\norm{f}_{M_{\varphi}} \approx \sup\limits_{t > 0} \varphi(t) f^{\star}(t).
	\end{equation*}
	In particular, for $\varphi(t) = t^{1/p}$ with $1 < p < \infty$, the space $M_{\varphi}$ coincide, up to the equivalence of norms, with $L_{p,\infty}$.
	
	The above constructions of Lorentz and Marcinkiewicz spaces are dual to each other in the sense that $(\Lambda_{\varphi})^{\times} = M_{\psi}$
	and $(M_{\varphi})^{\times} \equiv \Lambda_{\psi}$, where $\psi(t) \equiv t/\varphi(t)$ for $t > 0$ (see, for example, \cite[Theorem~5.2, p.~112]{KPS82}).
	Moreover, it is easy to see that the lower and upper Boyd indices of Lorentz and Marcinkiewicz spaces coincide with the lower and upper dilation exponent, respectively.
	
	Let us denote by $\mathcal{U}$ the set of all non-negative, concave and positively homogeneous functions
	$\varrho \colon [0,\infty) \times [0,\infty) \rightarrow [0,\infty)$ which vanish only at $(0,0)$. 
	For a function $\varrho$ from $\mathcal{U}$ and two Banach ideal spaces $X$ and $Y$, both defined on the same measure space,
	by the {\bf Calder{\' o}n--Lozanovski{\u \i} space $\varrho(X,Y)$} (or the {\bf Calder{\' o}n--Lozanovski{\u \i} construction})
	we understand a vector space consisting of all measurable functions, say $f$, such that
	\begin{equation} \label{INEQ: CL construction}
		\abs{f(t)} \leqslant \lambda \varrho(\abs{g(t)},\abs{h(t)})
	\end{equation}
	for some $\lambda > 0$ and $g \in \text{Ball}(X)$, $h \in \text{Ball}(Y)$. Moreover, the norm $\norm{f}_{\varrho(X,Y)}$ of a function $f$
	from $\varrho(X,Y)$ is defined as the infimum over all $\lambda > 0$ for which the above inequality (\ref{INEQ: CL construction}) holds.
	It can be shown that
	\begin{equation*}
		\norm{f}_{\varrho(X,Y)} = \inf \bigl\{ \max \left\{ \norm{g}_X, \norm{h}_Y \right\}
			\colon \abs{f(t)} \leqslant \varrho(\abs{g(t)},\abs{h(t)}) \text{ for some } g \in X \text{ and } h \in Y \bigr\}.
	\end{equation*}
	Vladimir I. Ovchinnikov showed that the above construction is an interpolation functor (but not necessary exact)
	on the category of Banach ideal spaces with the Fatou property.
	The importance of Calder{\' o}n--Lozanovski{\u \i} spaces is partly due to the fact that it unifies many other remarkable constructions.
	For example, in the case of power functions, that is, when $\varrho(s,t) = s^{1-\theta}t^{\theta}$ with $0 \leqslant \theta \leqslant 1$,
	the space $\varrho(X,Y)$ coincide with the celebrated Calder{\' o}n interpolation construction $X^{1-\theta}Y^{\theta}$ defined already
	in 1964 (see \cite{Cal64}). In particular, the {\bf $p$-convexification $X^{(p)}$} of $X$, where $1 < p < \infty$, is a special case of
	the Calder{\' o}n product, namely,
	\begin{equation*}
		X^{1/p}(L_{\infty})^{1-1/p} \equiv X^{(p)} \coloneqq \left\{ f \in L_0 \colon \abs{f}^p \in X \right\}
	\end{equation*}
	and $\norm{f}_{X^{(p)}} \coloneqq \norm{\abs{f}^p}^{1/p}_X$. On the other hand, if $\varrho(s,t) = tM^{-1}(s/t)$ for $t > 0$, where $M^{-1}$
	is the right-continuous inverse of an Young function\footnote{As usual, by an {\bf Young function $M$} we shall mean a convex function
	$M \colon [0,\infty) \rightarrow [0,\infty]$ which vanish at zero and is neither identically equal to zero nor infinity. Moreover, we put
	\begin{equation}
		a_M \coloneqq \sup \left\{ u \geqslant 0 \colon M(u) = 0 \right\}
	\end{equation}
	and
	\begin{equation}
		b_M \coloneqq \inf \left\{ u \geqslant 0 \colon M(u) = \infty \right\}.
	\end{equation}
	} $M$, then the space $\varrho(X,L_{\infty})$ coincide with the
	{\bf generalized Orlicz space\footnote{In many places in the literature, the space $X_M$ is called the {\it Calder{\' o}n--Lozanovski{\u \i}
	space} itself. We, however, will stick to the name \enquote{generalized Orlicz space}, to avoid some inconsistency in nomenclature and,
	possibly, subsequent misunderstandings.} $X_M$}, which is defined as
	\begin{equation*}
		X_M \coloneqq \left\{f \in L_0 \colon M\left( \frac{\abs{f}}{\lambda} \right) \in X \text{ for some } \lambda > 0 \right\}
	\end{equation*}
	and equipped with the {\bf Luxemburg--Nakano norm}
	\begin{equation*}
		\norm{f}_{X_M} \coloneqq \inf \left\{ \lambda > 0 \colon \norm{M\left( \frac{\abs{f}}{\lambda} \right)}_X \leqslant 1 \right\}.
	\end{equation*}
	In particular, the space $\varrho(L_1,L_{\infty})$ coincides up to the equality of norms with the familiar {\bf Orlicz space $L_M$}.
	Of course, if $M(u) = u^p$ with $1 \leqslant p < \infty$, then the Orlicz space $L_M$ is nothing else but the classical Lebesgue space $L_p$.
	The Boyd indices of Orlicz spaces $L_M$ coincide with the so-called {\bf Matuszewska--Orlicz indices} $\alpha_M$ and $\beta_M$
	(see \cite[Proposition~2.b.5, p.~139]{LT79} and \cite[Chapter~11]{Mal04}).
	
	\begin{remark}[Sequence spaces]
		All the above constructions, of course, have their discrete counterparts; and so, the {\bf Lorentz sequence space} $\ell_{p,q}$ is the
		collection of all sequences $x = \{x_n\}_{n=1}^{\infty}$ with
		\begin{equation*}
			\norm{x}_{\ell_{p,q}} \coloneqq \left[ \sum_{n=1}^{\infty} x_n^{\star q} \left( (n+1)^{q/p} - n^{q/p} \right) \right]^{1/q} < \infty
		\end{equation*}
		and, moreover, the {\bf weak $\ell_p$-space} $\ell_{p,\infty}$ consists of all sequences $x = \{x_n\}_{n=1}^{\infty}$ such that
		\begin{equation*}
			\norm{x}_{\ell_{p,\infty}} \coloneqq \sup\limits_{n \in \mathbb{N}} n^{1/p} x_n^{\star} < \infty.
		\end{equation*}
		By the {\bf Lorentz sequence space $\lambda_{\varphi}$}, respectively, the {\bf Marcinkiewicz sequence space $m_{\varphi}$}, we understand
		the space of all sequences $x = \{x_n\}_{n=1}^{\infty}$ for which
		\begin{equation*}
			\norm{x}_{\lambda_{\varphi}} \coloneqq \sum_{n=1}^{\infty} x_n^{\star} \bigl( \varphi(n+1) - \varphi(n) \bigr) < \infty,
		\end{equation*}
		respectively,
		\begin{equation*}
			\norm{x}_{m_{\varphi}} \coloneqq \sup\limits_{n \in \mathbb{N}} \frac{\varphi(n)}{n} \sum_{k=1}^n x_k^{\star} < \infty.
		\end{equation*}
		Finally, the {\bf Orlicz sequence space} $\ell_{M}$, where $M$ is a Young function, is a vector space consisting of all sequences,
		say $x = \{x_n\}_{n=1}^{\infty}$, such that $\sum_{n=1}^{\infty} M(\abs{x_n}/\lambda) < \infty$ for some $\lambda > 0$ equipped with
		the so-called {\bf Luxemburg--Nakano norm}
		\begin{equation*}
			\norm{x}_{\ell_{M}} \coloneqq \inf \left\{ \lambda > 0 \colon \sum_{n=1}^{\infty} M\left( \frac{\abs{x_n}}{\lambda} \right) \leqslant 1 \right\}.
		\end{equation*}
		In particular, if $M(u) = u^p$ with $1 \leqslant p < \infty$, then $\ell_M \equiv \ell_p$. Moreover, it is easy to see that $\ell_M = \ell_{\infty}$
		if, and only if, $a_M > 0$.
		\demo
	\end{remark}

	\begin{notation}
		Whenever we consider the function and sequence case {\it together}, we will, unless stated otherwise, prefer to use the function
		space notation. In all these situations, after some obvious modifications, function' arguments can be converted into their sequence
		equivalents and {\it vice versa}. Note, however, that contrary to G. H. Hardy's popular maxim, what goes for functions does not
		necessarily go for sequences in general, so some caution seems advisable
		(see, for example, \cite{BGE06} or Question~\ref{QUESTION: Multipliers Cesaro--Orlicz function}).
	\end{notation}

	For general properties of Banach ideal spaces, rearrangement invariant spaces and, in particular, Lorentz, Marcinkiewicz and Orlicz spaces,
	we refer to the books by Bennett and Sharpley \cite{BS88}, Kre{\u \i}n, Petunin and Semenov \cite{KPS82}, Lindenstrauss and Tzafriri \cite{LT77}, \cite{LT79}
	and Maligranda \cite{Ma89}.

	For more information about Banach lattices and, even more generally, Riesz spaces see, for example, Meyer-Nieberg \cite{MN91} and Zaanen \cite{Za83}.
	
	\subsection{Pointwise multipliers} \label{SECTION: pointwise multipliers}
	Let $X$ and $Y$ be two Banach ideal spaces both defined on the same measure space $(\Omega,\Sigma,\mu)$. The space of
	{\bf pointwise multipliers} $M(X,Y)$ between $X$ and $Y$ is defined as a vector space
	\begin{equation*}
		M(X,Y) \coloneqq \left\{ f \in L_0(\Omega,\Sigma,\mu) \colon fg \in Y \text{ for all } g \in X \right\}
	\end{equation*}
	and furnished with the natural in this situation operator norm $\norm{f}_{M(X,Y)} \coloneqq \sup_{\norm{g}_X \leqslant 1} \norm{fg}_Y$
	(after all, every function $f \in M(X,Y)$ induces the multiplication operator $M_f \colon X \rightarrow Y$ given as $M_f \colon g \rightsquigarrow fg$).
	With this norm $M(X,Y)$ becomes a Banach ideal space itself (see \cite[Proposition~2]{MP89}). Moreover, if $X$ and $Y$ are two r.i. spaces,
	then $M(X,Y)$ is a r.i. space as well (see \cite[Theorem~2.2]{KLM12} and the first part of the proof of Lemma~4.3 in \cite{KT23}).
	
	\begin{digression}
		To shed a familiar light on the above construction, let us mention that the space $M(X, L_1)$ coincide, even up to the equality
		of norms, with the K{\" o}the dual $X^{\times}$, so we can naturally think of the space $M(X,Y)$ as the generalized K{\" o}the dual
		of $X$ (for more papers contemplating this approach see, for example, \cite{CDSP08} and \cite{MP89}; there, by analogy with K{\" o}the dual,
		the space $M(X,Y)$ is denoted somewhat different by $X^Y$ and is called the {\bf $Y$-dual space} of $X$).
		\demo
	\end{digression}
	
	Again, by the very definition,
	\begin{equation} \label{INEQ: Holder-Rogers inequality}
		\norm{fg}_Y \leqslant \norm{f}_X \norm{g}_{M(X,Y)}
	\end{equation}
	and there are good reasons\footnote{Since $M(X,L_1) \equiv X^{\times}$, so the classical H{\" o}lder--Rogers inequality corresponds
	to the situation when $Y \equiv L_1$.} to call the above inequality the {\bf generalized H{\" o}lder--Rogers inequality}.
	
	\begin{example}
		It seems prudent to note now a few very basic facts, which one should be aware of. First of all, let $1 \leqslant q < p < \infty$ and set
		$1/r = 1/q - 1/p$. Then $M(L_p,L_q) \equiv L_r$ and $M(\ell_p,\ell_q) \equiv \ell_r$ (see \cite[Example~1.6]{Cro69}, \cite[Proposition~3]{MP89} or \cite[Example~3.1]{Nak17}).
		Moreover, the space $M(L_q,L_p)$ is trivial, but $M(\ell_q,\ell_p) \equiv \ell_{\infty}$ (see, e.g., \cite[Theorem~2]{MP89}). It is also worth noting
		that in general $M(X,X) \equiv L_{\infty}$ (see, for example, \cite[Theorem~1]{MP89}). Many more sophisticated results\footnote{These results, apart from
			isolated exceptions, are never isometric in nature.} of this type, that is, identifying the space $M(X,Y)$ for different configurations
		of Banach ideal spaces $X$ and $Y$, can be found, for example, in \cite{KLM12}, \cite{KLM14}, \cite{KT23}, \cite{LT17}, \cite{LT21}, \cite{MP89} and \cite{Sch10};
		see also \cite{Nak17} and references given there.
	\end{example}
	
	\subsection{Pointwise products} \label{SECTION: pointwise products}
	Since the space $M(X,Y)$ of pointwise multipliers acting between two Banach ideal spaces $X$ and $Y$ may informally be regarded
	as a \enquote{pointwise quotient} of the space $Y$ by $X$, so in some sense an opposite construction is the {\bf pointwise product
	space $X \odot Y$} defined by
	\begin{equation*}
		X \odot Y \coloneqq \left\{ gh \colon g \in X \text{ and } h \in Y \right\},
	\end{equation*}
	and equipped with the quasi-norm(!)
	\begin{equation*}
		\norm{f}_{X \odot Y} \coloneqq \inf \bigl\{ \norm{g}_X \norm{h}_Y \colon f = gh,\, g \in X \text{ and } h \in Y \bigr\}.
	\end{equation*}
	Note that the product $X \odot Y$ is a vector space due to the ideal property of $X$ and $Y$. Moreover, we have
	\begin{equation*}
		\norm{f}_{X \odot Y} = \norm{\abs{f}}_{X \odot Y} = \inf \bigl\{ \norm{g}_X \norm{h}_Y \colon \abs{f} = gh \text{, } g \in X_+ \text{ and } h \in Y_+ \bigr\},
	\end{equation*}
	where $X_+$ and $Y_+$ are the positive cones in $X$ and, respectively, $Y$ (see \cite[Proposition~1]{KLM14}).
	Note also that $X \odot Y$ is a r.i. quasi-Banach ideal space provided both spaces $X$ and $Y$ were r.i. spaces too.
	
	\begin{example}
		For further reference, it is worth noting that if $1 \leqslant p,q \leqslant \infty$, then $L_p \odot L_q \equiv L_r$, where $1/r = 1/p + 1/q$.
		Even more generally, $X^{(p)} \odot X^{(q)} \equiv X^{(r)}$ (see \cite[Lemma~1]{Mal04} and \cite[Lemma~2.21(i)]{ORS08}).
	\end{example}
	
	The so-called {\bf cancellation property} for multipliers will often prove helpful
	(see \cite[Theorem~4]{KLM14}; cf. \cite[p.~72--74]{Be96}, \cite[p.~46]{GE98} and \cite[Lemma~4(i)]{KLM19})
	
	\begin{theoremletters}[Cancellation property] \label{THM: cancellation property for multipliers}
		{\it For any three Banach ideal spaces, say $E$, $F$ and $G$, with $G$ having the Fatou property, the following holds}
		\begin{equation} \tag{$\div$} \label{CANCELLATION PROPERTY}
			M(E \odot F, E \odot G) \equiv M(F,G).
		\end{equation}
	\end{theoremletters}
	
	The construction $(X,Y) \rightsquigarrow X \odot Y$ is inextricably linked with the complex interpolation method. More precisely,
	the product space $X \odot Y$ of two Banach ideal spaces $X$ and $Y$ can be represented as the $\frac{1}{2}$-concavification of the
	Calder{\' o}n space $X^{1/2}Y^{1/2}$, that is to say,
	\begin{equation} \label{EQ: product as 1/2-concavification of Calderon product}
		X \odot Y \equiv \left( X^{1/2}Y^{1/2} \right)^{(1/2)}
	\end{equation}
	(see \cite[Theorem~2.1]{Sch10}; cf. \cite[Theorem~1(iv)]{KLM14}).
	
	\begin{digression}
		If we choose the right perspective, namely, we will think of Banach ideal spaces as modules over $L_{\infty}$,
		then one can argue, that the pointwise product $\odot$ is the \enquote{true} tensor product in the category of quasi-Banach
		modules over $L_{\infty}$ (see \cite[Theorem~1]{CS14}; cf. \cite{OT72}).
		\demo
	\end{digression}
	
	Because the embedding $Z \hookrightarrow X \odot Y$ implies that any function $f$ from $Z$ can be {\bf factorized}
	as the pointwise product $f = gh$ with $g \in X$ and $h \in Y$, so it is not meaningless to call the identification
	of the product space in the form $X \odot Y = Z$ a {\bf factorization}.
	Factorization is a powerful tool with many deep and profound applications in harmonic analysis, interpolation theory and operator theory.
	For example, $\bullet$ Coifman, Rochberg and Weiss (weak) factorization of Hardy spaces in several variables with applications to commutators
	and boundedness of certain Hankel operators \cite{CRW76}, $\bullet$ Cohn and Verbitsky factorization of tent spaces with applications to complex
	interpolation and boundedness of Hankel operators acting in pairs of Triebel spaces \cite{CV00} and $\bullet$ Pau and Zhao (weak) factorization of
	weighted Bergman spaces \cite{PZ15} (see also \cite{CFG17} and \cite{Kal92}).
	
	For more information about the spaces of pointwise products we refer to \cite{Bun87}, \cite{CS14}, \cite{KLM14}, \cite{Sch10}
	and references given there.
	
	\subsection{Factorization}
	Celebrated {\bf Lozanovski{\u \i}'s factorization} theorem (see \cite[Theorem~6]{Loz69}; cf. \cite{Gil81}, \cite{JR76}, \cite[Example~6, p.~185]{Mal04} and \cite{Rei81})
	says that for any $\varepsilon > 0$ each function $f$ from $L_1$ can be written as a pointwise product
	of two functions, say $g$ and $h$, one from $X$ and the other from $X^{\times}$, in such a way that
	\begin{equation} \label{PRE: Lozanovski factorization}
		\norm{f}_{L_1} \leqslant \norm{g}_X \norm{h}_{X^{\times}} \leqslant (1+\varepsilon) \norm{f}_{L_1}.
	\end{equation}
	Moreover, we can set $\varepsilon = 0$ in the above inequality as long as we know that $X$ has the Fatou property.

	On the one hand, we can simply interpret Lozanovski{\u \i}'s result \eqref{PRE: Lozanovski factorization} as the commutativity of the diagram
	\begin{equation*}
		\begin{tikzcd}[column sep=small]
			& X \arrow[dr, dashrightarrow, "M_h"] & \\
			L_{\infty} \arrow[ur, dashrightarrow, "M_g"] \arrow{rr}{M_f} & & L_1
		\end{tikzcd}
	\end{equation*}
	together with the estimate
	\begin{equation*}
		\norm{M_g \colon L_{\infty} \rightarrow X} \norm{M_h \colon X \rightarrow L_1} \leqslant (1+\varepsilon)\norm{M_f \colon L_{\infty} \rightarrow L_1}.
	\end{equation*}
	
	On the other, using the language of pointwise products and pointwise multipliers, we can express Lozanovski{\u \i}'s
	result \eqref{PRE: Lozanovski factorization} in the following form
	\begin{equation*}
		X \odot M(X, L_1) = L_1.
	\end{equation*}
	Due to this, the short road leads to the much more general question: {\it For which pairs of Banach ideal spaces, say $X$ and $Y$, the following factorization
	\begin{equation} \label{EQ : factorization}
		X \odot M(X,Y) = Y
	\end{equation}
	holds?} We will refer to this kind of questions as the {\bf factorization problems}. Note, however, that it is very
	simple to find a pair of Banach ideal spaces $X$ and $Y$ such that \eqref{EQ : factorization} does not hold.
	For example, take $X = \ell_p$ and $Y = \ell_q$ with $1 \leqslant p < q < \infty$. Then
	\begin{equation*}
		\ell_p \odot M(\ell_p, \ell_q) = \ell_p \odot \ell_{\infty} = \ell_p \neq \ell_q.
	\end{equation*}
	Some more sophisticated examples of this kind can be found in \cite[Example~2]{KLM14} and \cite[Proposition~A.1]{KT23}.
	
	Even in some concrete classes of Banach function or sequence spaces, {\it en route} to the solution of the factorization problem problem
	is bumpy and, sometimes, also only one-way. With future needs in mind, we will now roughly summarize what we know about the factorization
	problem in the case when both involved spaces comes from the class of Lorentz spaces $L_{p,q}$, Orlicz spaces $L_M$ and Musielak--Orlicz spaces $L_{\Phi}$.
	
	\begin{theoremletters}[Factorization of Lorentz spaces $L_{p,q}$] \label{THEOREM: Factorization of Lpq}
		{\it Let $1 \leqslant p,q,r,s \leqslant \infty$. Then the Lorentz space $L_{p,q}$ can be factorized through another Lorentz space
		$L_{r,s}$, that is,}
		\begin{equation*}
			L_{p,q} \odot M(L_{p,q},L_{r,s}) = L_{r,s}
		\end{equation*}
		{\it if, and only if, either $p > r$ and $q \geqslant s$, or $p = r$ and $q = s$.}
	\end{theoremletters}
	\begin{proof}[\enquote{Proof}]
		This follows from \cite{CS17} in tandem with \cite[Theorem~4]{KLM19}.
	\end{proof}

	To say something about the factorization of Orlicz spaces we need a little more preparation. For a given two Young functions, say $M$ and $N$,
	by the {\bf generalized Young conjugate}\footnote{This construction is intended to naturally generalize the well-known construction of the
	{\it Young conjugate function} appearing in the theory of Orlicz spaces (see \cite[Chapter~8]{Ma89}; cf. \cite[pp.~77--78]{Ma89}). Recall
	that the K{\" o}the dual of the Orlicz space $L_M$ or, what comes to the same thing, the space of pointwise multipliers from the Orlicz space
	$L_M$ into $L_1$, can be described as follows: $L_M^{\times} = M(L_M,L_1) = L_{M^*}$, where $M^*(t) \coloneqq \sup_{s > 0} \left\{ ts - M(s) \right\}$
	is the Young's conjugate function to $M$ (cf. \cite[Theorem~9.1, p.~63]{Ma89}; see also \cite[Remark~1, p.~54]{Ma89}).
	Now, it is evident that $M^* \equiv \text{id} \ominus M$.} of $M$ with respect to $N$ we will understand another Young function $N \ominus M$
	defined in the following way
	\begin{equation*}
		(N \ominus M)(t) \coloneqq \sup_{s > 0} \left\{ N(ts) - M(s) \right\}.
	\end{equation*}
	Let us also note that the description of the space of pointwise multipliers between two Orlicz spaces consists of many partial results of various
	authors scattered over nearly half a century (we will only mention \cite{And60} and \cite{MP89}, referring to \cite{LT17} and \cite{LT21} for a
	much more detailed discussion).
	The complete answer to the factorization problem for Orlicz spaces is provided in the following
		
	\begin{theoremletters}[Factorization of Orlicz spaces] \label{THM: Faktoryzacja Orliczy LT17}
		{\it Let $M$ and $N$ be two Young functions. Then the Orlicz space $L_M$ can be factorized through another Orlicz space $L_N$, that is,}
		\begin{equation*}
			L_N \odot M(L_N,L_M) = L_M
		\end{equation*}
		{\it if, and only if,}
		\begin{equation*}
			N^{-1}(M \ominus N)^{-1} \approx M^{-1} \quad \textit{ for the appropriate values.}
		\end{equation*}
	\end{theoremletters}
	\begin{proof}[\enquote{Proof}]
		This is a consequence of \cite[Theorem~3]{DR00} and \cite[Theorem~2]{LT17}.
	\end{proof}

	In the remaining case concerning the factorization of Musielak--Orlicz spaces, we will only conclude that the main ideas about the construction of the generalized
	Young conjugate function from \cite{DR00} and \cite{LT17}, transfer, modulo some technical details, to the Musielak--Orlicz set up. Therefore, in the light of
	Theorem~\ref{THM: Faktoryzacja Orliczy LT17}, the following result is expected.

	\begin{theoremletters}[Factorization of Musielak--Orlicz spaces] \label{THM: Faktoryzacja MO}
		{\it Let $(\Omega,\Sigma,\mu)$ be a complete and $\sigma$-finite measure space. Further, let $\Phi$ and $\Psi$ be two Musielak--Orlicz functions
		both defined on $\Omega \times (0,\infty)$. Then the Musielak--Orlicz space $L_{\Psi}$ can be factorized through another Musielak--Orlicz space
		$L_{\Phi}$, that is,}
		\begin{equation*}
			L_{\Phi} \odot M(L_{\Phi},L_{\Psi}) = L_{\Psi},
		\end{equation*}
		{\it provided}
		\begin{equation} \label{MO}
			{\Psi}^{-1}(\Phi \ominus \Psi)^{-1} \approx {\Phi}^{-1} \quad \textit{ for the appropriate values.}
		\end{equation}
		{\it However, in general, the converse implication does not hold.}
	\end{theoremletters}
	\begin{proof}[\enquote{Proof}]
		This is just a juxtaposition of \cite[Corollary~13]{LT21} together with \cite[Example~14]{LT21}. However, for all the details and a precise definition
		of the object $\Phi \ominus \Psi$ from \eqref{MO}, we refer directly to \cite{LT21}.
	\end{proof}
	
	For various reasons, the problem of factorization of Banach ideal spaces, or even more generally, Banach lattices,
	has attracted many researchers over the years (see, for example, \cite{AM09}, \cite{Be96}, \cite{CS17}, \cite{GE98}, \cite{KLM14},
	\cite{KLM19}, \cite{Pis79} and \cite{Sch10}). Some remarkable (weak) factorization results for spaces without the ideal property, like Bergman
	spaces, Hardy spaces or tent spaces, can be found, for example, in \cite{CRW76}, \cite{CV00}, \cite{Hor77}, \cite{JR76} and \cite{PZ15}.
	Moreover, there is an elegant connection, already noted by Kalton, between complex interpolation, Lozanovski{\u \i}'s factorization
	and the so-called twisted sums (see \cite{CS14}, \cite{CS17}, \cite{CFG17}, \cite{Kal92} and references therein).
	
	\subsection{Symmetrizations} \label{SECTION: symmetization}
	For a quasi-Banach function or sequence space $X$, the {\bf symmetrization $X^{\bigstar}$} of $X$ is defined as a set
	\begin{equation*}
		X^{\bigstar} \coloneqq \left\{ f \in L_0 \colon f^{\star} \in X \right\}
	\end{equation*}
	together with the functional $\norm{f}_{X^{\bigstar}} \coloneqq \norm{f^{\star}}_X$. It can be shown that $X^{\bigstar}$ is a vector space if,
	and only if, the dilation operator $D_2$ is bounded on $X^{d}$, that is, the {\bf cone of non-negative and decreasing functions} from $X$ (see \cite[Corollary~1]{KLM19};
	cf. \cite[Lemma~1.4]{KR09}). In this case, of course, $(X^{\bigstar},\norm{\cdot}_{X^{\bigstar}})$ is a rearrangement invariant quasi-Banach space.
	Moreover, it is straightforward to see that the space $X^{\bigstar}$ is non-trivial if, and only if, $\chi_{(0,a)} \in X$ for some $a > 0$
	(see, for example, \cite[Lemma~3]{KLM19}).
	
	\begin{remark}
		The reason why the construction $X \rightsquigarrow X^{\bigstar}$ pops up when examining the structure of r.i. spaces can be found in the
		simple observation that many classical rearrangement invariant spaces (notably Lorentz and Marcinkiewicz spaces) can be regarded as symmetrizations
		of weighted $L_p$-spaces. To be more precise,
		\begin{equation*}
			\Lambda_{\varphi} \equiv \left[ L_1(\varphi') \right]^{\bigstar}, \quad L_{p,q} \equiv \left[ L_q\left( t^{\frac{1}{p} - \frac{1}{q}} \right) \right]^{\bigstar}
			\quad \text{ and } \quad M_{\varphi} \equiv \left[ L_{\infty}(\varphi) \right]^{\bigstar};
		\end{equation*}
		\begin{equation*}
			\lambda_{\varphi} \equiv \bigl[ \ell_1(\varphi(n+1)-\varphi(n)) \bigr]^{\bigstar}, \quad \ell_{p,q} \equiv \left[ \ell_q\left( n^{\frac{1}{p} - \frac{1}{q}} \right) \right]^{\bigstar}
			\quad \text{ and } \quad m_{\varphi} \equiv \left[ \ell_{\infty}(\varphi) \right]^{\bigstar}.
		\end{equation*}
		\demo
	\end{remark}
	
	A much more complete study of the properties of this construction can be found, for example, in \cite{KLM19} and \cite{KR09}.
	
	\subsection{Spaces related to decreasing functions} \label{Section: Introduction Cesaro and Tandori}
	Among all decreasing functions that lie above (in the sense of a natural partial order) a given non-negative Lebesgue measurable
	function $f$ defined on the half-line there is a unique least one, namely, the {\bf least decreasing majorant} of $f$.
	We will denote it here by $\widetilde{f}$.
	Explicitly, this object can be described by the following point-wise formula
	\begin{equation*}
		\widetilde{f}(x) = \esssup\limits_{t \geqslant x} \abs{f(t)} \quad \text{ for } \quad x > 0.
	\end{equation*}
	Clearly enough, the least decreasing majorant $\widetilde{f}$ of a given function $f$ is non-negative and, well, decreasing.
	Moreover, the map $f \rightsquigarrow \widetilde{f}$ is a sublinear, that is, $\widetilde{f + g} \leqslant \widetilde{f} + \widetilde{g}$.
	The least decreasing majorant of decreasing function is, of course, the same function, which means that the map $f \rightsquigarrow \widetilde{f}$
	is also an idempotent.
	
	Fix a Banach function space $X$.
	By the {\bf Tandori function space}\footnote{According to the best of our knowledge, the name of these spaces was proposed
	by Le{\' s}nik and Maligranda in 2016 (see \cite[p.~767]{LM16}, where the authors refer to 1955 Tandori's paper \cite{Ta55}).} $\widetilde{X}$
	we will understand here a vector space of those measurable functions, say $f$, such that $\widetilde{f} \in X$. It is natural to equip $\widetilde{X}$
	with the norm $\norm{f}_{\widetilde{X}} \coloneqq \Vert \widetilde{f} \Vert_X$. Now, it is routine to verify that $\widetilde{X}$ is again a Banach function space.
	Moreover, the space $\widetilde{X}$ is non-trivial provided $X$ contains a positive and decreasing function.
	
	By the {\bf down space} $X^{\downarrow}$ we mean a Banach function space defined by the functional
	$\norm{f}_{X^{\downarrow}} \coloneqq \sup \int_{0}^{\infty} \abs{f(t)}g(t) dt$,
	where the supremum is taken over all positive and decreasing functions $g$ with $\norm{g}_{X^{\times}} \leqslant 1$.
	The functional $\norm{\cdot}_{X^{\downarrow}}$ is a norm provided $\chi_{(0,a)} \in X^{\times}$ for each $a > 0$.
	
	Finally, the {\bf Ces{\' a}ro function space} $\mathscr{C}X$ is defined as an optimal domain for Hardy's operator
	$\mathscr{H} \colon f \rightsquigarrow [x \rightsquigarrow \frac{1}{x} \int_0^x \abs{f(t)}dt]$, that is,
	as the largest (in the sense of inclusion) Banach function space with the property that $\mathscr{H}$ is
	bounded when acting into $X$. Explicitly, the space $\mathscr{C}X$ is the collection of all measurable functions $f$
	such that the norm
	$\norm{f}_{\mathscr{C}X} \coloneqq \norm{x \rightsquigarrow \frac{1}{x} \int_0^x \abs{f(t)}dt}_X$ is finite.
	It is straightforward to see that the space $\mathscr{C}X$ is non-trivial provided $\frac{1}{x}\chi_{(a,\infty)} \in X$
	for some $a > 0$.
	
	It is immediate to see that
	\begin{equation*}
		\widetilde{X} \hookrightarrow X \hookrightarrow X^{\times\times} \hookrightarrow X^{\downarrow}
	\end{equation*}
	and, in general, each of the inclusions is strict. Both spaces $\widetilde{X}$ and $X^{\downarrow}$ are related by duality, namely,
	$(X^{\downarrow})^{\times} = \widetilde{X^{\times}}$ (see \cite[Theorem~6.7]{Si94}, \cite[Theorem~5.7]{Si01} and \cite[Theorem~2.1]{Si07};
	cf. \cite[Theorem~D]{KMS07} and \cite[Theorem~2]{LM15a}). Moreover, for a r.i. function spaces $X$ with the Fatou property and such
	that $\alpha_X > 1$, we have $X^{\downarrow} = \mathscr{C}X$ (see \cite[Theorem~3.1]{Si01}).
	
	\begin{remark}[Sequence spaces]
		Let $X$ be a Banach sequence space. As always, the above constructions have their sequence counterparts. And so,
		\begin{equation*}
			\norm{x}_{\widetilde{X}} \coloneqq \norm{ \left\{ \sup_{k \geqslant n} \abs{x_k} \right\}_{n=1}^{\infty} }_X,
		\end{equation*}
		where $\widetilde{x} \coloneqq \left\{ \sup_{k \geqslant n} \abs{x_k} \right\}_{n=1}^{\infty}$ is the {\bf least decreasing majorant}
		of a sequence $x = \{x_n\}_{n=1}^{\infty}$,
		\begin{equation*}
			\norm{x}_{X^{\downarrow}} \coloneqq \sup \sum_{n=1}^{\infty} \abs{x_n} y_n,
		\end{equation*}
		where the supremum is taken over all positive and decreasing functions $y = \{y_n\}_{n=1}^{\infty}$ with $\norm{y}_{X^{\times}} \leqslant 1$,
		and
		\begin{equation*}
			\norm{x}_{\mathscr{C}X} \coloneqq \norm{ \left\{ \frac{1}{n} \sum_{k=1}^n \abs{x_k} \right\}_{n=1}^{\infty} }_X,
		\end{equation*}
		where $\mathscr{H} \colon x \rightsquigarrow \left[ n \rightsquigarrow \frac{1}{n} \sum_{k=1}^n \abs{x_k} \right]$ with $n \in \mathbb{N}$
		is the {\bf discrete Hardy operator}.
		\demo
	\end{remark}
	
	Over the last two decades, research on the structure of the above-mentioned spaces (which we will collectively call
	{\bf spaces related to decreasing functions}) has enjoyed unwavering interest from many different viewpoints like, for example,
	$\bullet$ isomorphic structure \cite{Ast17}, \cite{ALM19}, \cite{AM09}, \cite{CR16}, \cite{DS07} and \cite{KT17};
	$\bullet$ isometric structure \cite{KKM21} and \cite{KKT22};
	$\bullet$ interpolation structure \cite{AM13}, \cite{Les15}, \cite{LM16}, \cite{MS06} and \cite{MS17};
	$\bullet$ products, factors and factorization \cite{AM09}, \cite{BMM18}, \cite{GPU22} and \cite{KLM19};
	$\bullet$ duality \cite{KMS07}, \cite{LM15a}, \cite{NP11}, \cite{Si01} and \cite{Si03};
	$\bullet$ operator ideals \cite{Wal20}
	(see also survey papers \cite{AM14}, \cite{FLM16} and \cite{Si07}).
	For more information, we also refer to two monographs \cite{Be96} and \cite{GE98}.
	
	\section{{\bf Blocking technique via interpolation}} \label{SECTION: BF representations}
	
	In this section, we introduce the concept of discretization, analyze the interpolation properties of this construction and provide
	the block form representations.
	
	\subsection{Amalgams} Let us recall the following definition.
	
	\begin{definition}[Amalgams] \label{DEF: direct sum}
		Let $E$ be a Banach sequence space defined on $J$. Further, let $\{X_j\}_{j \in J}$ be a family of Banach spaces.
		By the {\bf $E$-direct sum} of the family $\{X_j\}_{j \in J}$ (abbreviated, {\bf direct sum} if the context leaves no doubt)
		we will understand here a vector space
		\begin{equation*}
			\Bigl( \bigoplus_{j \in J} X_j \Bigr)_{E} \coloneqq \left\{\{x_j\}_{j \in J} \in \prod_{j \in J}X_j \colon \left\{ \norm{x_j}_{X_j} \right\}_{j \in J} \in E \right\}
		\end{equation*}
		equipped with the norm
		\begin{equation*}
			\norm{\{x_j\}_{j \in J}}_{\left( \bigoplus_{j \in J} X_j \right)_{E}} \coloneqq \norm{\sum_{j \in J}\norm{x_j}_{X_j}e_j}_E.
		\end{equation*}
		Here $\{e_j\}_{j \in J}$ is the canonical basis of the space $c_0(J)$.
		Only a brief reflection is enough to see that when equipped with the coordinatewise-defined addition and scalar multiplication
		$\bigl( \bigoplus_{j \in J} X_j \bigr)_E$ become a Banach space itself. Moreover, remembering what was said in Remark~\ref{REMARK: Wiener amalgams},
		we will collectively call the class of spaces $\bigl( \bigoplus_{j \in J} X_j \bigr)_E$ {\bf amalgams}.
	\end{definition}

	\begin{remark}[Direct integrals and K{\" o}the--Bochner spaces] \label{REMARK: Bochner construction}
		Note that the space $\bigl( \bigoplus_{j \in J} X_j \bigr)_{E}$ is only a special case of a much more general construction of the so-called {\bf $E$-direct integral}
		$\bigl( \int_{\Omega}^{\oplus} X_{\omega} d\mu \bigr)_{E}$ (here $E$ can be a Banach function space defined on a complete and decomposable measure space
		$(\Omega,\Sigma,\mu)$; see \cite[Chapter~6]{HLR91} for details). Many of our results can be translated to this more abstract context. On the other hand,
		if the family $\{ X_j \}_{j \in J}$ of Banach spaces is \enquote{constant}, that is, $X_j = X$ for all $j \in J$ and some Banach space $X$,
		then the space $\bigl( \bigoplus_{j \in J} X \bigr)_E$ coincide with the {\bf K{\" o}the--Bochner space} $E(X)$ of strongly $E$-summable sequences from $X$
		(see, for example, \cite{Lin04}).
		\demo
	\end{remark}

	\begin{remark}[Amalgams as Banach ideal spaces] \label{REMARK: Direct sum = BFS}
		Let $E$ be a Banach sequence space on $J$ and let $\{X_j\}_{j \in J}$ be a family of Banach ideal spaces. Suppose that $X_j$ is defined on
		$(\Omega_j,\Sigma_j,\mu_j)$ for $j \in J$. Then the space $\bigl( \bigoplus_{j \in J} X_j \bigr)_E$ can be seen as a Banach ideal space
		defined on $(\Omega,\Sigma,\mu)$, where $\Omega$ is a {\bf disjoint union} $\bigsqcup_{j \in J} \Omega_j$ of $\Omega_j$'s and the measure $\mu$
		is defined, for $A \subset \bigsqcup_{j \in J} \Omega_j$, as $\mu(A) \coloneqq \sum_{j \in J} \mu_j(A \cap \Omega_j)$.
		Occasionally, it is conceptually helpful to think about a member of $\bigl( \bigoplus_{j \in J} X_j \bigr)_E$ as an infinite matrix
		in which the $j^{\text{th}}$ row is inhabited by the functions living in $X_j$. For instance, the space
		$\bigl( \bigoplus_{n=1}^{\infty} \ell_p^n \bigr)_{\ell_q}$ can be seen as the space of all lower triangular matrices
		\begin{equation*}
			\begin{bmatrix}
				x_{11} & 0 & 0 & 0 & \dots \\
				x_{21} & x_{22} & 0 & 0 & \dots \\
				x_{31} & x_{32} & x_{33} & 0 & \dots \\
				\vdots & \vdots & \vdots & \vdots & \ddots
			\end{bmatrix}
		\end{equation*}
		such that the expression $\sum_{i=1}^{\infty} \bigl( \sum_{j=1}^n \abs{x_{ij}}^p \bigr)^{q/p}$ is finite.
		\demo
	\end{remark}

	\begin{remark}[Amalgams as optimal domains]
		Fix a Banach sequence space $E$ on $J$. Let $S \colon f \rightsquigarrow \bigl\{ \norm{f(j)}_{X_j} \bigr\}_{j \in J}$, where $X_j$'s are
		Banach ideal spaces for $j \in J$, be a positive and sublinear operator. A quick look at Definition~\ref{DEF: direct sum} and the diagram
		\begin{equation*}
			\begin{tikzcd}
				& E           \\
				X \arrow[ru, "S"] \arrow[r, hook] & \Bigl( \bigoplus_{j \in J} X_j \Bigr)_{E} \arrow[u, "S"']
			\end{tikzcd}
		\end{equation*}
		is enough to see that $\bigl( \bigoplus_{j \in J} X_j \bigr)_{E}$ is nothing else but $[S,E]$, that is, the optimal domain
		for $S$ and $E$ (see \cite{ORS08} for more).
		\demo
	\end{remark}
	
	Some basic properties of direct sums are summarized below in the form of a handy proposition.
	
	\begin{proposition}[Basic properties of amalgams] \label{Prop: Podstawowe wlasnosci sum prostych}
		{\it Let $E$ and $F$ be two Banach sequence spaces defined on $J$. Further, let $\{X_j\}_{j \in J}$ and $\{Y_j\}_{j \in J}$ be two families
		of Banach ideal spaces.}
		\begin{enumerate}
			\item [(a)] {\it Suppose that $E$ is rearrangement invariant. Further, suppose that $\sup_{n,m \in \mathbb{N}} d(X_n \oplus X_m, X_{n+m}) < \infty$,
				where $d(X_n \oplus X_m, X_{n+m})$ stands for the Banach--Mazur distance between $X_n \oplus X_m$ and $X_{n+m}$.
				Then the space $\bigl( \bigoplus_{n=1}^{\infty} X_n \bigr)_E$ is isomorphic to $\bigl( \bigoplus_{k=1}^{\infty} X_{n_k} \bigr)_E$
				for any unbounded sequence $\{n_k\}_{k=1}^{\infty}$ of positive integers. In particular, if $1 \leqslant p \leqslant \infty$ and $a > 1$,
				then the space $\bigl( \bigoplus_{n=1}^{\infty} \ell_p^n \bigr)_E$ is isomorphic to $\bigl( \bigoplus_{n=1}^{\infty} \ell_p^{a^n} \bigr)_E$.}
			\item [(b)] {\it Let $\left\{ T_j \colon X_j \rightarrow Y_j \right\}_{j \in J}$ be a sequence of isomorphisms.
				Suppose that $\left\{ \norm{T_j \colon X_j \rightarrow Y_j} \right\}_{j \in J}$ belongs to $M(E,F)$ and
				$\{ \Vert T^{-1}_j \colon Y_j \rightarrow X_j \Vert \}_{j \in J}$ belongs to $M(F,E)$.
				Then the space $\bigl( \bigoplus_{j \in J} X_j \bigr)_E$ is isomorphic to $\bigl( \bigoplus_{j \in J} Y_j \bigr)_F$.
			 	In particular, the space $\bigl( \bigoplus_{j \in J} X_j \bigr)_E$ is isomorphic to $\bigl( \bigoplus_{j \in J} Y_j \bigr)_E$
		 		provided the sequence $\left\{ \norm{T_j \colon X_j \rightarrow Y_j} \right\}_{j \in J}$ is bounded.}
			\item [(c)] {\it The K{\" o}the dual of the space $\bigl( \bigoplus_{j \in J} X_j \bigr)_E$ is isometrically isomorphic to
				$\bigl( \bigoplus_{j \in J} X_j^{\times} \bigr)_{E^{\times}}$. In particular, the K{\" o}the dual of the space
				$\bigl( \bigoplus_{n=1}^\infty \ell_p^{n} \bigr)_E$ is isometrically isomorphic to $\bigl( \bigoplus_{n=1}^\infty \ell_q^{n} \bigr)_{E^{\times}}$.
				Here $1 \leqslant p,q \leqslant \infty$ with $1/p + 1/q = 1$.}
		\end{enumerate}
	\end{proposition}
	\begin{proof}
		(a) This follows directly from Casazza, Kottman and Lin's result \cite[Theorem~6]{CKL77} (see also \cite[Proposition~4.1]{AA17}).
		
		(b) Let us define the mapping $\bigoplus_{j \in J} T_j \colon \bigl( \bigoplus_{j \in J} X_j \bigr)_E \rightarrow \bigl( \bigoplus_{j \in J} Y_j \bigr)_F$
		in the following way
		\begin{equation*}
			\bigoplus_{j \in J} T_j \colon \{x_j\}_{j \in J} \rightsquigarrow \left\{ T_j(x_j) \right\}_{j \in J}.
		\end{equation*}
		Then, for $x = \{x_j\}_{j \in J}$ from $\bigl( \bigoplus_{j \in J} X_j \bigr)_E$, it is straightforward to see that
		\begin{align*}
			\norm{\bigoplus_{j \in J} T_j(x)}_{\left( \bigoplus_{j \in J} Y_j \right)_F}
				& = \norm{\left\{ \norm{T_j(x_j)}_{Y_j} \right\}_{j \in J}}_F \\
				& \leqslant \norm{\left\{ \norm{T_j}_{X_j \rightarrow Y_j} \norm{x_j}_{X_j} \right\}_{j \in J}}_F \\
				& \leqslant \norm{\left\{ \norm{T_j}_{X_j \rightarrow Y_j} \right\}_{j \in J}}_{M(E,F)} \norm{\left\{ \norm{x_j}_{X_j} \right\}_{j \in J}}_E,
		\end{align*}
		where the first inequality follows from the ideal property of the space $F$ and the second one is due to the general form
		of the H{\" o}lder--Rogers inequality \eqref{INEQ: Holder-Rogers inequality}. This, however, means that
		\begin{equation*}
			\norm{\bigoplus_{j \in J} T_j}_{\bigl( \bigoplus_{j \in J} X_j \bigr)_E \rightarrow \left( \bigoplus_{j \in J} Y_j \right)_F}
				\leqslant \norm{\left\{ \norm{T_j}_{X_j \rightarrow Y_j} \right\}_{j \in J}}_{M(E,F)}.
		\end{equation*}
		Analogously, one can show that
		\begin{equation*}
			\norm{\bigoplus_{j \in J} T^{-1}_j}_{\bigl( \bigoplus_{j \in J} Y_j \bigr)_F \rightarrow \left( \bigoplus_{j \in J} X_j \right)_E}
			\leqslant \norm{\left\{ \norm{T^{-1}_j}_{Y_j \rightarrow X_j} \right\}_{j \in J}}_{M(F,E)},
		\end{equation*}
		where the mapping $\bigoplus_{j \in J} T^{-1}_j \colon \bigl( \bigoplus_{j \in J} Y_j \bigr)_F \rightarrow \bigl( \bigoplus_{j \in J} X_j \bigr)_E$
		is defined as
		\begin{equation*}
			\bigoplus_{j \in J} T^{-1}_j \colon \{y_j\}_{j \in J} \rightsquigarrow \left\{ T^{-1}_j(y_j) \right\}_{j \in J}.
		\end{equation*}
		The rest is obvious, because $M(E,E) \equiv \ell_{\infty}(J)$ (cf. \cite[Proposition~3]{Pel60}).
		
		(c) This is folklore (see, for example, \cite[pp.~175--178]{Lau01} and \cite[pp.~14--15]{DK16}).
		Anyway, since we will prove a more general variant of this result later (see Theorem~\ref{Thm: komutowanie M z cdot}),
		so we will take advantage of this situation and show how this can be deduced without referring to the mentioned results.
		Indeed, remembering about Remark~\ref{REMARK: Direct sum = BFS} and using Theorem~\ref{Thm: komutowanie M z cdot}, we have
		\begin{align*}
			\biggl[ \Bigl( \bigoplus_{j \in J} X_j \Bigr)_E \biggr]^{\times}
				& \equiv M\biggl[ \Bigl( \bigoplus_{j \in J} X_j \Bigr)_E, L_1\Bigl( \bigsqcup_{j \in J} \Omega_j \Bigr) \biggr] \\
				& \equiv M\biggl[ \Bigl( \bigoplus_{j \in J} X_j \Bigr)_E, \Bigl( \bigoplus_{j \in J} L_1(\Omega_j) \Bigr)_{\ell_1(J)} \biggr] \\
				& \equiv \biggl[ \bigoplus_{j \in J} M \bigl( X_j,L_1(\Omega_j) \bigr) \biggr]_{M(E,\ell_1(J))}
					\equiv \Bigl( \bigoplus_{j \in J} X_j^{\times} \Bigr)_{E^{\times}}.
		\end{align*}
	\end{proof}

	\begin{digression}
		The isomorphic structure of spaces $\bigl( \bigoplus_{j \in J} X_j \bigr)_{E}$ can be quite non-obvious and surprising. For example,
		as Hagler and Stegall showed (see \cite[Theorem~1]{HS73}), the space $\bigl( \bigoplus_{n=1}^{\infty} \ell_1^n \bigr)_{\ell_{\infty}}$
		contains an isomorphic copy of $L_1(0,1)$. 
		For this (among many others) reasons, deciding whether the spaces $\bigl( \bigoplus_{j \in J} X_j \bigr)_{E}$ and $\bigl( \bigoplus_{j \in J} Y_j \bigr)_{F}$
		are (or are not) isomorphic is usually difficult. On the trivial side, it is clear that both spaces $\bigl( \bigoplus_{j \in J} X_j \bigr)_{E}$
		and $\bigl( \bigoplus_{j \in J} X_j \bigr)_{E(w)}$ are even isometrically isomorphic. If, however, the isomorphism between the spaces $E$
		and $F$ is more exotic, it does not have to be so clear. For instance, even though the spaces $\ell_p$ and $\bigl( \bigoplus_{n = 1}^\infty \ell_2^n \bigr)_{\ell_p}$
		with $2 < p < \infty$ are isomorphic\footnote{Actually, this isomorphism comes from what we now call Pe{\l}czy{\' n}ski's decomposition method,
		so it fully deserves the name \enquote{exotic} (see \cite[Proposition~7]{Pel60}).}, it was proved by Draga and Kochanek that the spaces
		$\bigl( \bigoplus_{n = 1}^\infty c_0 \bigr)_{\ell_p}$ and $\bigl( \bigoplus_{n = 1}^\infty c_0 \bigr)_{L}$ with $L = \bigl( \bigoplus_{n=1}^\infty \ell_2^n \bigr)_{\ell_p}$
		are not isomorphic (essentially, because they have different Szlenk index; see \cite[Example~5.13]{DK16} for details).
		\demo
	\end{digression}

	\subsection{Discretization} \label{SUBSECTION: E(X)}
	The key definition is (cf. \cite[p.~468]{BK91} and \cite[p.~254]{Kal92})
	
	\begin{definition}[Discretization] \label{DEF: E(X)}
		Let $X$ be either a Banach function space or a Banach sequence space.
		With the space $X$ we can associate a Banach sequence space $\mathbf{E}(X)$ defined on $\mathbb{J}$,
		where either $J = \mathbb{Z}$ provided $X$ is a Banach function spaces or $\mathbb{J} = \mathbb{Z}_+$ provided $X$ is a Banach sequence space,
		defined as a vector space
		\begin{equation*} \label{Def: E(X) function}
			\mathbf{E}(X) \coloneqq \left\{ x = \{x_j\}_{j \in \mathbb{J}} \colon \sum_{j \in \mathbb{J}} x_j \chi_{\Delta_j} \in X \right\}
		\end{equation*}
		and equipped with the norm  $\norm{x}_{\mathbf{E}(X)} \coloneqq \norm{\sum_{j \in \mathbb{J}} x_j \chi_{\Delta_j}}_X$.
		Here, $\Delta_j$ is the {\bf $j^{\text{th}}$ dyadic interval}, that is, $\Delta_j \coloneqq [2^{j},2^{j+1})$ for $j \in \mathbb{J}$.
		To the construction $X \rightsquigarrow \mathbf{E}(X)$ we will further refer as the {\bf dyadic discretization} of $X$ (briefly,
		the {\bf $\Delta$-discretization} or just {\bf discretization} if the context is clear).
	\end{definition}

	Before we go any further, note the following very simple but instructive example, which will keep our intuition on the right track.
	
	\begin{example} \label{EXAMPLE : ELp computations}
		Let $1 \leqslant p,q \leqslant \infty$ with $1/p + 1/q = 1$. Suppose that $1 \leqslant p < \infty$. Then ${\bf E}(L_p) \equiv \ell_p(W_p)$, where $W_p(j) = 2^{j/p}$
		for $j \in \mathbb{Z}$, and ${\bf E}(\ell_p) \equiv \ell_p(w_p)$, where $w_p(j) = 2^{j/p}$ for $j \in \mathbb{Z}_+$.
		In the remaining case, when $p = \infty$, ${\bf E}(L_{\infty}) \equiv \ell_{\infty}(\mathbb{Z})$ and ${\bf E}(\ell_{\infty}) \equiv \ell_{\infty}$.
		Let us now turn our attention toward duality.
		Again, suppose that $1 \leqslant p < \infty$. Then ${\bf E}(L_p)^{\times} \equiv \ell_q(M_p)$, where $M_p = 2^{-j/p}$ for $j \in \mathbb{Z}$,
		and ${\bf E}(\ell_p)^{\times} \equiv \ell_q(m_p)$, where $m_p(j) = 2^{-j/p}$ for $j \in \mathbb{Z}_+$.
		Moreover, since $W_q M_q \equiv 1$ and $M_p M_q \equiv M_1$, so clearly $\ell_q(M_p) \equiv \ell_q(W_qM_qM_p) \equiv {\bf E}(L_q)(M_1)$.
		Thus, ${\bf E}(L_p)^{\times} \equiv {\bf E}(L_p^{\times})(M_1)$ and, analogously, ${\bf E}(\ell_p)^{\times} \equiv {\bf E}(\ell_p^{\times})(m_1)$.
		Now, suppose that $p = \infty$. We have
		\begin{equation*}
			{\bf E}(L_{\infty})^{\times} \equiv \ell_{\infty}^{\times} \equiv \ell_1 \equiv \ell_1(W_1M_1) \equiv {\bf E}(L_1)(M_1) \equiv {\bf E}(L_{\infty}^{\times})(M_1).
		\end{equation*}
		Similarly, one can show that ${\bf E}(\ell_{\infty})^{\times} \equiv {\bf E}(\ell_{\infty}^{\times})(m_1)$.
	\end{example}
	
	A few clarifying remarks about Definition~\ref{Def: E(X) function} seem in order. 
	
	\begin{remark} \label{REMARK: EX as sequence and function space}
		We can consider the space $\mathbf{E}(X)$ as a Banach sequence space or as a Banach function space,
		which contains all measurable functions, say $f$, which can be represented in the form
		$f(t) = \sum_{j \in \mathbb{J}} a_n \chi_{\Delta_j}(t)$
		and is equipped with the norm inherited from $X$. Both points of view are actually isometrically isomorphic
		via the mapping $\{x_j\}_{j \in \mathbb{J}} \rightsquigarrow \sum_{j \in \mathbb{J}} x_j \chi_{\Delta_j}$.
		We do not intend to be particularly pedantic about distinguishing between them.
		\demo
	\end{remark}
	
	\begin{remark} \label{Remark: EX jest i nie jest r.i.}
		Let $X$ be a rearrangement invariant space.
		With few exceptions\footnote{Most notably, $\mathbf{E}(L_{\infty}) \equiv \ell_{\infty}(\mathbb{Z})$ and $\mathbf{E}(\ell_{\infty}) \equiv \ell_{\infty}$
		(cf. Example~\ref{EXAMPLE : ELp computations}). One can even show that if $X$ is r.i. space, then ${\bf E}(X)$ is r.i. if, and only if, $X = L_{\infty}$.}
		the space $\mathbf{E}(X)$ is not rearrangement invariant when we consider it as defined over $(\mathbb{J},2^{\mathbb{J}},\#)$,
		where $\#$ is the counting measure on $\mathbb{J}$.
		However, if we think of $\mathbf{E}(X)$ as of a Banach sequence space defined over $(\mathbb{J},2^{\mathbb{J}},2^{\#})$,
		where $2^{\#} \coloneqq \sum_{j \in \mathbb{J}} 2^{j}\delta_j$ and $\delta_{j_0}$ is the Dirac measure at $j_0 \in \mathbb{J}$,
		then $\mathbf{E}(X)$ is rearrangement invariant.
		\demo
	\end{remark}

	\begin{digression} \label{Remark: EX is complemented in X}
		Note that as long as the {\bf averaging operator} $\text{Ave} \colon X \rightarrow X$ defined by
		\begin{equation} \label{EQ: Ave operator definition}
			\text{Ave} \colon f \rightsquigarrow \left[ t \rightsquigarrow \sum_{j \in \mathbb{J}} \left( 2^{-j} \int_{\Delta_j} f(s)ds \right)\chi_{\Delta_j}(t) \right]
		\end{equation}
		is bounded on $X$, the space ${\mathbf{E}}(X)$ is a $\lambda$-complemented subspace of $X$ with the constant $\lambda$ equal to $\norm{\text{Ave}}_{X \rightarrow X}$.
		This is the case\footnote{Since the averaging operator $\text{Ave}$ is a contraction on $L_1$ and $L_{\infty}$, so due to the Calder{\' o}n--Mityagin interpolation
		theorem, it is also a contraction on every interpolation space between $L_1$ and $L_{\infty}$.}, for example, when $X$ is an interpolation space between
		$L_1$ and $L_{\infty}$.
		This is a particularly nice situation, because then the above construction gives rise to the functor
		$\mathbf{E} \colon \{ \textit{Banach ideal spaces} \} \rightarrow \{ \textit{Banach sequence spaces} \}$,
		which makes the following diagram commutative
		\begin{equation*}
			\begin{tikzcd}
				X \arrow[r, "T"] \arrow[d, "\mathbf{E}"']
				& Y \arrow[d, "\mathbf{E}"] \\
				\mathbf{E}(X) \arrow[r, dashrightarrow, "\mathbf{E}(T)"]
				& \mathbf{E}(Y)
			\end{tikzcd}
		\end{equation*}
		Here,
		\begin{equation*}
			\mathbf{E}(T) \colon x = \{x_j\}_{j \in \mathbb{J}} \rightsquigarrow
				\left\{ 2^{-j}\int_{\Delta_j} T\left( \sum_{j \in \mathbb{J}} x_j\chi_{\Delta_j} \right)(s) ds \right\}_{j \in \mathbb{J}}.
		\end{equation*}
		Anticipating the facts a little, this observation can be considered the main reason why the construction $X \rightsquigarrow {\bf E}(X)$ has so many good properties.
		\demo
	\end{digression}
	
	The following result can be found without proof on page 254 in Kalton's paper \cite{Ka93a}. For the sake of completeness, we will give a proof of a somewhat more general fact.
	
	\begin{proposition}[Duality and discretization] \label{Prop: E komutuje z Kothe dualem}
		{\it Let $X$ be a Banach ideal space such that the averaging operator} $\text{Ave}$ {\it (defined in \eqref{EQ: Ave operator definition})
		is bounded on $X$. Then}
			\begin{equation*}
				\mathbf{E}(X)^{\times} \equiv \mathbf{E}( X^{\times} )(w),
			\end{equation*}
		{\it where $w(j) = 2^{-j}$ for $j \in \mathbb{J}$.}
	\end{proposition}
	\begin{proof}
		Take $x = \{x_j\}_{j \in \mathbb{J}}$ form $\mathbf{E}(X)^{\times}$. Let us note the fact that the averaging operator $\text{Ave}$
		is an isometry on $L_1$. Indeed, for $f \in L_1$, we have
		\begin{align*}
			\norm{f}_{L_1}
				& = \sum_{j \in \mathbb{J}} \int_{\Delta_j} \abs{f(t)}dt \\
				& = \sum_{j \in \mathbb{J}} 2^j \left( 2^{-j}\int_{\Delta_j} \abs{f(t)}dt \right) \\
				& = \norm{ \sum_{j \in \mathbb{J}} \left( 2^{-j}\int_{\Delta_j} \abs{f(t)}dt \right)\chi_{\Delta_j} }_{L_1}
				= \norm{\text{Ave}(f)}_{L_1}.
		\end{align*}
		Therefore, it is straightforward to see that
		\begin{align*}
			\norm{x}_{\mathbf{E}\left( X^{\times} \right)(w)}
				& = \sup\limits_{\norm{f}_X \leqslant 1} \norm{f \cdot \sum_{j \in \mathbb{J}} 2^{-j}x_j\chi_{\Delta_j}}_{L_1} \\
				& = \sup\limits_{\norm{f}_X \leqslant 1} \norm{\text{Ave}\left( f\cdot \sum_{j \in \mathbb{J}} 2^{-j}x_j\chi_{\Delta_j} \right)}_{L_1} \\
				& = \sup\limits_{\norm{f}_X \leqslant 1}
					\norm{\sum_{j \in \mathbb{J}} \left( 2^{-j} \int_{\Delta_j} f(t)dt \right)\chi_{\Delta_j} \cdot \sum_{j \in \mathbb{J}} 2^{-j}x_j\chi_{\Delta_j}}_{L_1} \\
				& = \sup\limits_{ \substack{ y = \{y_j\}_{j \in \mathbb{J}} \\ \norm{\sum_{j \in \mathbb{J}} y_j\chi_{\Delta_j}}_X \leqslant 1 } }
					\norm{\sum_{j \in \mathbb{J}} y_j\chi_{\Delta_j} \cdot \sum_{j \in \mathbb{J}} 2^{-j}x_j\chi_{\Delta_j}}_{L_1} \\
				& = \sup\limits_{ \substack{ y = \{y_j\}_{j \in \mathbb{J}} \\ \norm{\sum_{j \in \mathbb{J}} y_j\chi_{\Delta_j}}_X \leqslant 1 } }
					\norm{\sum_{j \in J} x_j y_j 2^{-j}\chi_{\Delta_j}}_{L_1} \\
				& = \sup\limits_{\norm{y}_{\mathbf{E}(X)} \leqslant 1} \norm{xy}_{\ell_1(\mathbb{J})} \\
				& = \norm{x}_{\mathbf{E}(X)^{\times}}.
		\end{align*}
		That's all.
	\end{proof}

	\begin{remark}[About Proposition~\ref{Prop: E komutuje z Kothe dualem}]
		Perhaps, it may come as a little surprise that the space $\mathbf{E}\left( X^{\times} \right)(w)$, being the K{\" o}the dual of $\mathbf{E}(X)$,
		is in general not rearrangement invariant (neither with respect to $(\mathbb{J},2^{\textbf{}},\#)$ nor to $(\mathbb{J},2^{\mathbb{J}},2^{\#})$).
		After all, the space $\mathbf{E}(X)$ is rearrangement invariant with respect to the measure $2^{\#}$ (cf. Remark~\ref{Remark: EX jest i nie jest r.i.}).
		The culprit for all the confusion is the fact that the measure space $(\mathbb{J},2^{\mathbb{J}},2^{\#})$ fails to be resonant (in the sense of Bennett and Sharpley;
		see, in this order, Definition~2.3, Theorem~2.7 and Proposition~4.2 in \cite{BS88}).
		\demo
	\end{remark}
	
	Next result was shown by Hern{\' a}ndez and Kalton in \cite{HK95} for Banach function spaces (see Propositions~2.2 and 2.3 therein;
	note also that the proof of the first one is actually attributed to Stephen Montgomery--Smith).
	Because, from our point of view, it is of fundamental importance and the presentation of arguments in \cite{HK95} is quite laconic, we will detail
	all the steps of the proof in the remaining case of Banach sequence spaces.
	
	\begin{proposition}[Discretization of symmetrization] \label{PROP: EX <-> carrier}
		{\it Let $X$ be a Banach ideal space with the Fatou property. Suppose that $\alpha_X > 1$. Then, for $f \in X^{\bigstar}$, we have}
		\begin{equation} \label{KALTON-equvalent norm}
			\norm{f}_{X^{\bigstar}} \approx \norm{ \left\{  f^{\star\star}  \left( 2^j \right) \right\}_{j \in \mathbb{J}} }_{\mathbf{E}(X)}.
		\end{equation}
		{\it Moreover, if $\beta_X < \infty$, then also}
		\begin{equation} \label{EX^s = EX}
			\mathbf{E}( X^{\bigstar} ) = \mathbf{E}(X).
		\end{equation}
	\end{proposition}
	\begin{proof}
		Let $X$ be a Banach sequence space with the Fatou property such that $\alpha_X > 1$. We will start by showing the formula\footnote{It seems worth noting that \eqref{KALTON-equvalent norm}
		means (among other things) that $X^{\bigstar}$ can be renormed to be a Banach space. In general, $X^{\bigstar}$ is only a quasi-Banach space, and that only with the additional
		assumption that $D_2$ is bounded on $X^d$ (see \cite[Corollary~1]{KLM19}). Without any additional assumptions, $X^{\bigstar}$ may not even be a linear space.} \eqref{KALTON-equvalent norm}.
	
		To do this, we will first show that $f^\star\leqslant \sum_{j=0}^{\infty} f^{\star\star}  \left( 2^j \right) \chi_{\Delta_j}$.
		Fix $n in \mathbb{N}$ and let $j\in\mathbb{Z}_+$ be such that $2^j\leqslant n< 2^{j+1}$. We have
		\begin{align*}
			f^{\star\star}(2^j) &=2^{-j}\sum_{k=1}^{2^j}f^{\star}\left( k \right)\geqslant 2^{-j} 2^j f^{\star}\left( 2^j \right)=f^{\star}\left( 2^j \right)\geqslant f^{\star}\left( n \right).
		\end{align*}
		However, this clearly implies the first inequality
		\begin{equation*}
			\norm{f}_{X^{\bigstar}}\leqslant\norm{ \left\{  f^{\star\star}  \left( 2^j \right) \right\}_{j=0}^{\infty} }_{\mathbf{E}(X)}.
		\end{equation*}
		To prove the second inequality observe that
		\begin{align*}
			f^{\star\star}(2^j) &=2^{-j}\sum_{k=1}^{2^j}f^{\star}\left( k \right)\leqslant 2^{-j}\sum_{k=0}^{j}2^{j-k}f^{\star}\left( 2^{j-k} \right)=\sum_{k=0}^{j}2^{-k}f^{\star}\left( 2^{j-k} \right).
		\end{align*}
		In consequence,
		\begin{align*}
			\sum_{j=0}^\infty f^{\star\star}  \left( 2^j \right) \chi_{\Delta_j}\left( n \right)
				&\leqslant\sum_{j=0}^\infty \sum_{k=0}^{j}2^{-k}f^{\star}\left( 2^{j-k} \right) \chi_{\Delta_j}\left( n \right) \\
				& =\sum_{k=0}^\infty 2^{-k}\sum_{j=k}^{\infty}f^{\star}\left( 2^{j-k} \right) \chi_{\Delta_j}\left( n \right) \\
				& =\sum_{k=0}^\infty 2^{-k}\sum_{j=0}^{\infty}f^{\star}\left( 2^{j} \right) \chi_{\Delta_{j+k}}\left( n \right) \\
				& \stackrel{(\spadesuit)}{\leqslant}\sum_{k=0}^\infty 2^{-k} \left(D_{2^{k+1}}f^{\star}\right)\left( n \right).
		\end{align*}
		Only $(\spadesuit)$ requires a more detailed explanation. Fix $n,k\in\mathbb{N}$. Note that
		\begin{equation*}
			\sum_{\left\{ j \in \mathbb{Z}_+ \colon 2^{k+j}\leqslant n<2^{k+j+1} \right\}} f^{\star}\left( 2^{j} \right) \chi_{\Delta_{j+k}}\left( n \right)=f^{\star}\left( 2^{j} \right).
		\end{equation*}
		Since the set of all $j \in \mathbb{Z}_+$ with $2^{k+j}\leqslant n<2^{k+j+1}$ coincide with the set of all $j \in \mathbb{Z}_+$
		with $2^{j-1}\leqslant 2^{-k-1}n<2^{j}$, so
		\begin{equation*}
			f^{\star}\left( 2^{j} \right) \leqslant f^{\star}\left( 2^{-k-1}n \right) = \left(D_{2^{k+1}}f^{\star}\right)\left( n \right)
		\end{equation*}
		and $(\spadesuit)$ follows. Take $1/\alpha_X < \alpha < 1$. Remembering about (\ref{EQ: D_s < CMAXs^1/p,s^1/q}), we have
		\begin{align*}
			\norm{ \left\{  f^{\star\star}  \left( 2^j \right) \right\}_{j=0}^{\infty} }_{\mathbf{E}(X)}
			&\leqslant\norm{\sum_{k=0}^\infty 2^{-k} D_{2^{k+1}}f^{\star}}_X \\
			& \leqslant\sum_{k=0}^\infty 2^{-k} \norm{D_{2^{k+1}}f^{\star}}_X\\
			&\leqslant\sum_{k=0}^\infty 2^{-k} \norm{D_{2^{k+1}}}_{X\to X}\norm{f^{\star}}_X \\
			& \leqslant\sum_{k=0}^\infty 2^{-k} 2^{\alpha(k+1)}\norm{f^{\star}}_X \\
			& = C\norm{f}_{X^{\bigstar}}.
		\end{align*}
		The proof of \eqref{KALTON-equvalent norm} has been completed.
		
		Assume additionally that $\beta_X<\infty$ and, due to \eqref{KALTON-equvalent norm}, that $X^{\bigstar}$ is a Banach sequence space. Take $f\in \E(X)_+$
		with finite support, that is, $f=\sum_{j=0}^{\infty} a_j \chi_{\Delta_j}\in X$ with $a_j \geqslant 0$ and $\#(\supp(f)) < \infty$. For arbitrary $f\in \E(X)$
		the formula (\ref{EX^s = EX}) will easily follow from the Fatou property of $X$. Coming back to the point, put
		$$
		g\coloneqq\sum\limits_{n=0}^\infty D_{2^{-n}}f=f+\sum\limits_{n=1}^\infty D_{2^{-n}}f.
		$$
		We need to check if $g\in\E(X)$. Take $0< \beta < 1/\beta_X$. Again, by virtue of (\ref{EQ: D_s < CMAXs^1/p,s^1/q}),
		\begin{align*}
			\norm{g}_X
				& =\norm{\sum\limits_{n=0}^\infty D_{2^{-n}}f}_X \\
				& \leqslant\sum\limits_{n=0}^\infty \norm{D_{2^{-n}}f}_X \\
				& \leqslant\sum\limits_{n=0}^\infty \norm{D_{2^{-n}}}_{X\to X}\norm{f}_X \\
				& \leqslant\sum\limits_{n=0}^\infty 2^{-\beta n}\norm{f}_X \\
				& \leqslant C\norm{f}_X.
		\end{align*}
		Moreover, since $f$ was chosen positive, so $f\leqslant g$. Combining all these facts together, we infer that
		\begin{equation} \label{EQ: A}
			\norm{f}_X\approx\norm{g}_X \quad \text{ and } \quad \norm{g}_{X^{\bigstar}} \approx \norm{f}_{X^{\bigstar}}.
		\end{equation}
		However, it is not difficult to see that $g$ is non-increasing, so
		\begin{equation} \label{EQ: AB}
			\norm{g}_X\approx\norm{g}_{X^{\bigstar}}.
		\end{equation}
		A quick inspection of (\ref{EQ: A}) and (\ref{EQ: AB}) leaves no doubt that the proof has just been completed.
	\end{proof}
	
	In view of the above Proposition~\ref{PROP: EX <-> carrier}, and following Grafakos and Kalton (see \cite[p.~154]{GK01}), we will extract the following class of spaces.
	
	\begin{definition}[Carrier spaces]
		Let $Q$ be a quasi-Banach space with the Fatou property. Suppose that the dilation operators $D_s$ are bounded on $X$
		with $\norm{D_s}_{X \rightarrow X} \leqslant Cs^{\kappa}$ for some $\kappa, C > 0$ and all $s > 0$. We will say that $Q$
		is a {\bf carrier space} for $X$ whenever $X = Q^{\bigstar}$.
	\end{definition}
	
	Because of \cite[Proposition~2.2]{HK95} and Proposition~\ref{PROP: EX <-> carrier}, $Q^{\bigstar}$ is always a r.i. space with the Fatou property.
	Roughly speaking, the mapping $X \rightsquigarrow \{Q \colon X = Q^{\bigstar}\}$ behaves like the \enquote{inverse} to $Q \rightsquigarrow Q^{\bigstar}$.
	
	\begin{example} \label{EXAMPLE : carrier spaces}
		It is clear that every r.i. space is a carrier space for itself. Moreover, since it is clear that
		\begin{equation*}
			[L_1(\varphi')]^{\bigstar} = \Lambda_{\varphi}, \quad [L_\infty(\varphi)]^{\bigstar} = M_{\varphi}, \quad \text{ and } \quad L_{\Phi}^{\bigstar} = \Lambda_{M,w},
		\end{equation*}
		so we can also say that a carrier space for a Lorentz space $\Lambda_{\varphi}$ is the weighted Lebesgue space $L_1(\varphi')$,
		a carrier space for the Marcinkiewicz space $M_{\varphi}$ is $L_\infty(\varphi)$ and a carrier space for the Orlicz--Lorentz space
		$\Lambda_{M,w}$ is the Musielak--Orlicz space $L_{\Phi}$ with $\Phi(t,u) = w(t)M(u)$. An analogous conclusion can easily be drawn
		for the sequence counterparts of these spaces.
	\end{example}
	
	\subsection{On interpolation and discretization} \label{SUBSECTION: interpolation of E(X)}
	
	The following preparatory result is essential.
	
	\begin{proposition} \label{Prop: wlasnosci E}
		{\it Let $\vv{X} = (X_0,X_1)$ be a couple of Banach ideal spaces. Suppose that the averaging operator} $\text{Ave}$ {\it (defined in \eqref{EQ: Ave operator definition})
		is bounded on both $X_0$ and $X_1$. Then ${\bf E}(\vv{X}) \coloneqq ({\bf E}(X_0),{\bf E}(X_1))$ is a Banach couple. Moreover, both functors $\Delta$ and $\Sigma$ commute
		with $\mathbf{E}$, that is to say,}
		\begin{equation} \label{EQ: E(XiY)=EXiEY}
			{\bf E}(\Delta(\vv{X})) = \Delta({\bf E}(\vv{X})) \quad \textit{ and } \quad {\bf E}(\Sigma(\vv{X})) = \Sigma({\bf E}(\vv{X})).
		\end{equation}
		{\it At last, for any interpolation space $X$ relative to the couple $\vv{X}$, we have\footnote{Here, we identify a sequence
		$x = \{x_j\}_{j \in \mathbb{J}}$ form $\mathbf{E}(X)$ with the function $\sum_{j \in \mathbb{J}} x_j\chi_{\Delta_j}$ (cf. Remark~\ref{REMARK: EX as sequence and function space}).}}
		\begin{equation} \label{EQ: EX = X i E(L_1+L_infty)}
			{\bf E}(X) = X \cap {\bf E}(\Sigma(\vv{X})).
		\end{equation}
	\end{proposition}
	\begin{proof}
		As for (\ref{EQ: E(XiY)=EXiEY}), the equality ${\bf E}(\Delta(\vv{X})) = \Delta({\bf E}(\vv{X}))$ is clear, so we will focus on the remaining one.
		It is straightforward to see that $\Sigma({\bf E}(\vv{X})) \hookrightarrow {\bf E}(\Sigma(\vv{X}))$.
		To show the second inclusion take $x = \{ x_j \}_{j \in \mathbb{J}}$ from ${\bf E}(\Sigma(\vv{X}))$. This means that $\sum_{j \in \mathbb{J}} x_j \chi_{\Delta_j} = g + h$
		for some $g \in X_0$ and $h \in X_1$. Let us consider the operator
		\begin{equation*}
			\text{ave} \colon f \rightsquigarrow \left[ t \rightsquigarrow \left\{ 2^{-j}\int_{\Delta_j} f(t) dt \right\}_{j \in \mathbb{J}} \right].
		\end{equation*}
		Since, due to our assumptions, the operator $\text{Ave}$ is bounded on both $X_0$ and $X_1$, so it is clear that both operator
		$\text{ave} \colon X_0 \rightarrow {\bf E}(X_0)$ and $\text{ave} \colon X_1 \rightarrow {\bf E}(X_1)$ are also bounded.
		Therefore, we have
		\begin{equation*}
			x = \text{ave} \left( \sum_{j \in \mathbb{J}} x_j \chi_{\Delta_j} \right)
				= \text{ave}(g + h)
				= \underbrace{\text{ave}(g)}_{\in {\bf E}(X_0)} + \underbrace{\text{ave}(h)}_{\in {\bf E}(X_1)}.
		\end{equation*}
		This, however, means that $x \in \Sigma({\bf E}(\vv{X}))$. In consequence, ${\bf E}(\Sigma(\vv{X})) \overset{C}{\hookrightarrow} \Sigma({\bf E}(\vv{X}))$
		with $C = \max \left\{ \norm{\text{Ave}}_{X_0 \rightarrow X_0}, \norm{\text{Ave}}_{X_1 \rightarrow X_1} \right\}$.
		
		Let us now turn to the proof of (\ref{EQ: EX = X i E(L_1+L_infty)}). Since ${\bf E}(X) \overset{1}{\hookrightarrow} X$ and $X {\hookrightarrow} \Sigma(\vv{X})$,
		so ${\bf E}(X) {\hookrightarrow} X \cap {\bf E}(\Sigma(\vv{X}))$. Now, take $x = \{x_j\}_{j \in \mathbb{J}}$ from $X \cap {\bf E}(\Sigma(\vv{X}))$ and observe
		that $x = \sum_{j \in \mathbb{J}} x_j \chi_{\Delta_j} \in X \cap \Sigma(\vv{X})$. However, since $X \cap \Sigma(\vv{X}) = X$, so $x \in {\bf E}(X)$.
		This means that $X \cap {\bf E}(\Sigma(\vv{X})) {\hookrightarrow} {\bf E}(X)$. With this, we have finished the proof.
	\end{proof}

	The fact that the construction $X \rightsquigarrow {\bf E}(X)$ commutes with any interpolation functor
	$\vv{X} \rightsquigarrow {\bf F}(\vv{X})$ should be considered the main result archived in this section.
	
	\begin{theorem}[Interpolation of discretization] \label{Proposition: E komutuje z interpolacja}
		{\it Let $\vv{X} = (X_0,X_1)$ be a couple of Banach ideal spaces such that the averaging operator} $\text{Ave}$ {\it (defined in \eqref{EQ: Ave operator definition})
		is bounded on both spaces $X_0$ and $X_1$. Then, for any interpolation functor ${\bf F}$, we have}
		\begin{equation*}
			{\bf F}({\bf E}(\vv{X})) = {\bf E}( {\bf F}(\vv{X}) ).
		\end{equation*}
	\end{theorem}
	\begin{proof}
		Since, thanks to our assumptions, the averaging operator $\text{Ave}$ is bounded on both $X_0$ and $X_1$, so it is clear the space ${\bf E}(\Sigma(\vv{X}))$
		is a complemented subspace of $\Sigma(\vv{X})$ (comp. with Digression~\ref{Remark: EX is complemented in X}). With this in mind, we have
		\begin{align*}
			{\bf F}({\bf E}(\vv{X}))
				& = {\bf F}(X_0 \cap {\bf E}(\Sigma(\vv{X})), X_1 \cap {\bf E}(\Sigma(\vv{X}))) \quad (\text{in view of Proposition~\ref{Prop: wlasnosci E}}) \\
				& = {\bf F}(\vv{X}) \cap {\bf E}(\Sigma(\vv{X})) \quad (\text{using \cite[Theorem~1, p.~118]{Tri78}}) \\
				& = {\bf E}({\bf F}(\vv{X})) \quad (\text{again, due to Proposition~\ref{Prop: wlasnosci E}}).
		\end{align*}
	\end{proof}
	
	\subsection{Block form representations} \label{SUBSECTION: Block representation}
	Everything we have done so far leads us here. This section extends results of Grosse-Erdmann \cite{GE98}.
	
	\begin{theorem}[Block form representation - sequence spaces] \label{Thm: Tandori sequence representation}
		{\it Let $X$ be a Banach sequence space with the Fatou property. Suppose that either $X$ has non-trivial Boyd indices
		or $X$ is rearrangement invariant. Then, we have}
			\begin{equation*}
				\norm{x}_{\widetilde{X}} \approx \norm{ \left\{ \norm{x\chi_{\Delta_j}}_{\ell_{\infty}} \right\}_{j=0}^{\infty} }_{\mathbf{E}(X)}
			\end{equation*}
		{\it and}
			\begin{equation*}
				\norm{x}_{X^{\downarrow}} \approx \norm{ \left\{ 2^{-j} \norm{x\chi_{\Delta_j}}_{\ell_{\infty}} \right\}_{j=0}^{\infty} }_{\mathbf{E}(X)}.
			\end{equation*}
	\end{theorem}
	\begin{proof}
		We will break up the argument into several steps.
		
		{\bf $1^{\text{st}}$ step.} Since, due to \eqref{KALTON-equvalent norm}, the space $X^{\bigstar}$ is (up to an equivalent norm)
		a Banach sequence space,
		\begin{equation*}
			\widetilde{X} = \widetilde{X^{\bigstar}} \quad \text{ and } \quad {\bf E}(X^{\bigstar}) = {\bf E}(X),
		\end{equation*}
		where the second equality follows from Proposition~\ref{PROP: EX <-> carrier}, so without the loss of generality we can assume that the space $X$ is rearrangement invariant.
		
		{\bf $2^{\text{nd}}$ step.} Now, our goal is to show that
		\begin{equation*}
			\vertiii{x}_{\widetilde{X}}
				\coloneqq \norm{ \left\{ \norm{ x\chi_{\Delta_j} }_{\ell_{\infty}} \right\}_{j=0}^{\infty} }_{\mathbf{E}(X)}
		\end{equation*}
		introduces an equivalent norm on the space $\widetilde{X}$. To be precise, we intend to prove that
		\begin{equation} \tag{$\spadesuit$} \label{Step 1: d_BM <= 4}
			\norm{\text{id} \colon \left( \widetilde{X}, \norm{\cdot}_{\widetilde{X}} \right) \rightarrow \left( \widetilde{X}, \vertiii{\cdot}_{\widetilde{X}} \right)}
				\norm{\text{id} \colon \left( \widetilde{X}, \vertiii{\cdot}_{\widetilde{X}} \right) \rightarrow \left( \widetilde{X}, \norm{\cdot}_{\widetilde{X}} \right)}
				\leqslant 4.
		\end{equation}
		
		To see the first inequality, take $x = \{x_j\}_{j=1}^\infty$ from $\widetilde{X}$. Moreover, note that
		\begin{equation} \label{EQ: D1/2}
			D_{1/2}\left( \left\{ \underbrace{a_1, a_1}_{\text{twice}}, \underbrace{a_2, a_2, a_2, a_2}_{4\text{-times}}, ..., \underbrace{a_j, ..., a_j}_{2^j\text{-times}}, ... \right\} \right)
				= \left\{ \underbrace{a_1}_{\text{once}}, \underbrace{a_2, a_2}_{\text{twice}}, ..., \underbrace{a_j, ..., a_j}_{2^{j-1}\text{-times}}, ... \right\}
		\end{equation}
		and
		\begin{equation} \label{EQ: tandori -> X -> linfty}
			\widetilde{X} \overset{c}{\hookrightarrow} X \overset{C}{\hookrightarrow} \ell_\infty,
		\end{equation}
		where $c = 1$ and $C = \varphi_{X^{\times}}(1)$ (see \cite[Corollary~6.8, p.~78]{BS88}).
		Keeping this in mind, we have
		\begin{align*}
			\vertiii{x}_{\widetilde{X}} & \leqslant \norm{ x\chi_{\Delta_0} }_X
			+ \norm{ \sum_{j=1}^{\infty}  \norm{x\chi_{\Delta_j}}_{\ell_{\infty}} \chi_{\Delta_j} }_{X} \\
			& \leqslant \varphi_X(1) \norm{x}_{\ell_\infty}
				+ \norm{D_{1/2}}_{X \rightarrow X}
					\norm{ \sum_{j=1}^{\infty} \norm{x\chi_{\Delta_j}}_{\ell_{\infty}} \chi_{\Delta_{j-1}} }_X \quad (\text{by \eqref{EQ: D1/2}}) \\
			& \leqslant \varphi_X(1) \varphi_{X^{\times}}(1) \norm{x}_{\widetilde{X}}
				+ \norm{ \sum_{j=1}^{\infty} \left( \sup\limits_{k \geqslant j} \abs{x_k} \right) e_j }_X
					\quad (\text{using \eqref{EQ: tandori -> X -> linfty} and \eqref{INEQ: D_s < max 1,s}}) \\
			& = 2\norm{x}_{\widetilde{X}}.
		\end{align*}
		
		Now, let us turn to the proof of the second inequality. We are going to use interpolation theory here. To do this,
		define a sublinear operator $\mathscr{T}$ in the following way
		\begin{equation*}
			\mathscr{T} \colon \sum_{j=1}^{\infty} x_j e_j \rightsquigarrow \sum_{j=1}^{\infty} \left( \sup\limits_{k \geqslant j} \abs{x_k} \right) e_j.
		\end{equation*}
		We want to show that the operator $\mathscr{T}$ is bounded on both spaces $\ell_1(w)$ and $\ell_{\infty}$, where $w(j) = 2^{-j}$ for $j \in \mathbb{Z}_+$.
		For this it actually suffices to show that
		\begin{equation} \label{EMB: not obvious}
			\ell_1(w) \hookrightarrow \widetilde{\ell_1(w)}
		\end{equation}
		and, respectively,
		\begin{equation} \label{EMB: obvious}
			\ell_{\infty} \hookrightarrow \widetilde{\ell_{\infty}}
		\end{equation}
		Since (\ref{EMB: obvious}) is obvious, so let us focus on (\ref{EMB: not obvious}).
		Note the following inequalities in the spirit of Gol'dman, Heinig and Stepanov (see\footnote{Note, however, that the proof presented in \cite{GHS96}
		contains a misprint - the expression \enquote{$\sum_{m \in \mathbb{Z}} [a_m\sigma_m]^2$} in the last line should contain the exponent $p$, not $2$.}
		\cite[Proposition~2.1(a)]{GHS96})
		\begin{align*}
			\sum_{j=0}^\infty \left(\sup\limits_{k \geqslant j} \abs{x_{k}}\right) 2^{j}
			& \leqslant \sum_{j=0}^\infty \left(\sum_{k=j}^\infty \abs{x_k}\right) 2^{j}
				\leqslant \sum_{j=0}^\infty \abs{x_j} \sum_{k=0}^j 2^{k} \\
			& = \sum_{j=0}^\infty \abs{x_j}2^{j} \sum_{k=0}^j 2^{-j} 2^{k} = \sum_{j=0}^\infty \abs{x_j} 2^{j} \sum_{k=0}^j 2^{-k+1} \\
			& \leqslant \left( \sum_{k=0}^\infty 2^{-k+1} \right)\left( \sum_{j=0}^\infty \abs{x_j} 2^{j} \right)
				\leqslant 2 \sum_{j=0}^\infty \abs{x_j} 2^{-j}.
		\end{align*}
		This means that
		\begin{equation*}
			\ell_1(w) \overset{2}{\hookrightarrow} \widetilde{\ell_1(w)}
		\end{equation*}
		and (\ref{EMB: not obvious}) follows.
		Next, as one can easily see, ${\bf E}(\ell_1) \equiv \ell_1(w)$ and ${\bf E}(\ell_{\infty}) \equiv \ell_{\infty}$, so the operator $\mathscr{T}$
		is bounded on both spaces ${\bf E}(\ell_1)$ and ${\bf E}(\ell_{\infty})$ with
		\begin{equation} \label{Tand norma na E}
			\norm{\mathscr{T}}_{{\bf E}(\ell_1) \rightarrow {\bf E}(\ell_1)} \leqslant 2
				\quad \text{ and } \quad
					\norm{\mathscr{T}}_{{\bf E}(\ell_{\infty}) \rightarrow {\bf E}(\ell_{\infty})} = 1.
		\end{equation}
		Moreover, since the couple $(\ell_1, \ell_{\infty})$ is a Calder{\' o}n--Mityagin couple, so $X = (\ell_1, \ell_{\infty})^{\mathscr{K}}_{\mathscr{X}}$ for some
		parameter space $\mathscr{X}$. Thus, invoking Theorem~\ref{Proposition: E komutuje z interpolacja}, we infer that
		\begin{equation*}
			{\bf E}(X) = {\bf E}\left[ (\ell_1, \ell_{\infty})^{\mathscr{K}}_{\mathscr{X}} \right]
				= \left( {\bf E}(\ell_1), {\bf E}(\ell_{\infty}) \right)^{\mathscr{K}}_{\mathscr{X}} = \left( \ell_1(w), \ell_{\infty} \right)^{\mathscr{K}}_{\mathscr{X}}.
		\end{equation*}
		Consequently, using \eqref{Tand norma na E}, we have
		\begin{align*}
			\norm{x}_{\widetilde{X}}
				& \leqslant \norm{ \sum_{j=0}^{\infty} \norm{ x\chi_{\Delta_j} }_{\ell_{\infty}} \chi_{\Delta_j} }_{\widetilde{X}} \\
				& = \norm{ \left\{ \norm{ x\chi_{\Delta_j} }_{\ell_{\infty}} \right\}_{j=0}^{\infty} }_{\widetilde{{\bf E}(X)}} \\
				& \leqslant \norm{\mathscr{T}}_{{\bf E}(X) \rightarrow {\bf E}(X)} \norm{ \left\{ \norm{ x\chi_{\Delta_j} }_{\ell_{\infty}} \right\}_{j=0}^{\infty} }_{{\bf E}(X)} \\
				& \leqslant \max \left\{ \norm{\mathscr{T}}_{{\bf E}(\ell_1) \rightarrow {\bf E}(\ell_1)}, \norm{\mathscr{T}}_{{\bf E}(\ell_{\infty}) \rightarrow {\bf E}(\ell_{\infty})} \right\} 	
					\vertiii{x}_{\widetilde{X}} \\
				& \leqslant 2 \vertiii{x}_{\widetilde{X}}.
		\end{align*}
		This completes the proof of \eqref{Step 1: d_BM <= 4}.
		
		{\bf $3^{\text{rd}}$ step.} It remains to justify the last part about $X^{\downarrow}$. Due to \cite[Theorem~2.1]{Si07}
		(see also \cite[Theorem~D]{KMS07}) and Proposition~\ref{PROP: EX <-> carrier}, we have
		\begin{equation*} \label{EQ: redukcja V down}
			X^{\downarrow}
				= \left( \widetilde{X^{\times}} \right)^{\times}
				= \left( \widetilde{ \left[ X^{\times} \right]^{\bigstar}} \right)^{\times}
				= \left( \widetilde{ \left[ X^{\bigstar} \right]^{\times}} \right)^{\times}
				= \left[ X^{\bigstar} \right]^{\downarrow},
		\end{equation*}
		where the third equality follows from \cite[Corollary~1.6]{KM07} (see also \cite[Theorem~2]{KLM19}).
		Thus, just as before, without the loss of generality we can assume that the space $X$ is rearrangement invariant.
		Keeping this in mind, we have
		\begin{align*}
			X^{\downarrow}
				& = \left( \widetilde{X^{\times}} \right)^{\times} \\
				& = \left[ \biggl( \bigoplus_{j=0}^\infty \ell^{2^{j}}_\infty \biggr)_{{\bf E}(X^{\times})} \right]^{\times}
					\quad (\text{in view of \eqref{Step 1: d_BM <= 4}}) \\
				& = \biggl( \bigoplus_{j=0}^\infty \ell^{2^{j}}_1 \biggr)_{{\bf E}(X^{\times})^{\times}}
					\quad (\text{by Proposition~\ref{Prop: Podstawowe wlasnosci sum prostych}(c)}) \\
				& = \biggl( \bigoplus_{j=0}^\infty \ell^{2^{j}}_1 \biggr)_{{\bf E}(X^{\times\times})(w)} \quad (\text{using Proposition~\ref{Prop: E komutuje z Kothe dualem}})\\
				& = \biggl( \bigoplus_{j=0}^\infty \ell^{2^{j}}_1 \biggr)_{{\bf E}(X)(w)} \quad (\text{due to the Fatou property of $X$}).
		\end{align*}
		The proof has been completed.
	\end{proof}

	The \enquote{integral} analog of Theorem~\ref{Thm: Tandori sequence representation} looks like this.
	
	\begin{theorem}[Block form representation - function spaces] \label{Thm: Tandori function representation}
		{\it Let $X$ be a Banach function space with the Fatou property. Suppose that either $X$ has non-trivial Boyd indices
		or $X$ is rearrangement invariant. Then, we have}
		\begin{equation*}
			\norm{f}_{\widetilde{X}} \approx \norm{ \left\{ \norm{f\chi_{\Delta_j}}_{L_{\infty}} \right\}_{j \in \mathbb{Z}}  }_{\mathbf{E}(X)}
		\end{equation*}
		{\it and}
		\begin{equation*}
			\norm{f}_{X^{\downarrow}} \approx \norm{ \left\{ 2^{-j} \norm{f\chi_{\Delta_j}}_{L_{1}} \right\}_{j \in \mathbb{Z}}  }_{\mathbf{E}(X)}.
		\end{equation*}
	\end{theorem}
	\begin{proof}
		In terms of ideas, this proof does not differ much from the proof of Theorem~\ref{Thm: Tandori sequence representation}.
		For this reason, we will only focus on showing how to modify the arguments used in the crucial second step.
		The rest is basically the same.
		
		We will show that
		\begin{equation} \tag{$\spadesuit$} \label{Step 1: d_BM <= 2}
			\norm{\text{id} \colon \left( \widetilde{X}, \norm{\cdot}_{\widetilde{X}} \right) \rightarrow \left( \widetilde{X}, \vertiii{\cdot}_{\widetilde{X}} \right)}
			\norm{\text{id} \colon \left( \widetilde{X}, \vertiii{\cdot}_{\widetilde{X}} \right) \rightarrow \left( \widetilde{X}, \norm{\cdot}_{\widetilde{X}} \right)}
			\leqslant 2,
		\end{equation}
		where
		\begin{equation*}
			\vertiii{f}_{\widetilde{X}} \coloneqq \norm{\sum_{j \in \mathbb{Z}} \left( \esssup\limits_{t \in \Delta_j} \abs{f(t)} \right) \chi_{\Delta_j}}_X
				= \norm{ \left\{ \norm{f\chi_{\Delta_j}}_{L_{\infty}} \right\}_{j \in \mathbb{Z}}  }_{\mathbf{E}(X)}.
		\end{equation*}
		The proof of the first inequality is simple,
		\begin{align*}
			\vertiii{f}_{\widetilde{X}}
			& = \norm{\sum_{j \in \mathbb{Z}} \left( \esssup\limits_{t \in \Delta_j} \abs{f(t)} \right) \chi_{\Delta_j}}_X \\
			& \leqslant \norm{D_{1/2}}_{X \rightarrow X}
					\norm{\sum_{j \in \mathbb{Z}} \left( \esssup\limits_{t \in \Delta_j} \abs{f(t)} \right) \chi_{\Delta_{j-1}}}_X \\
			& \leqslant \norm{D_{1/2}}_{X \rightarrow X}
					\norm{\sum_{j\in \mathbb{Z}} \left( \esssup\limits_{s \geqslant t} \abs{f(s)} \right) \chi_{\Delta_{j-1}}(t)}_X \\
			& = \norm{D_{1/2}}_{X \rightarrow X} \norm{\sum_{j \in \mathbb{Z}} \widetilde{f} \chi_{\Delta_{j-1}}}_{X} \\
			& \leqslant \norm{f}_{\widetilde{X}}.
		\end{align*}
		On the other hand,
		\begin{align*}
			\norm{f}_{\widetilde{X}}
				& = \norm{\sum_{j \in \mathbb{Z}} f\chi_{\Delta_j}}_{\widetilde{X}} \\
				& \leqslant \norm{\sum_{j \in \mathbb{Z}} \norm{f\chi_{\Delta_j}}_{L_\infty}\chi_{\Delta_j}}_{\widetilde{X}} \\
				& = \norm{\left\{ \norm{f\chi_{\Delta_j}}_{L_\infty} \right\}_{j \in \mathbb{Z}}}_{\widetilde{\mathbf{E}(X)}}.
		\end{align*}
		Thus, if we only knew that the sublinear operator
		\begin{equation*}
			\mathscr{T} \colon f \rightsquigarrow \left[ x \rightsquigarrow \esssup\limits_{t \geqslant x} \abs{f(t)} \right]
		\end{equation*}
		is bounded on $\mathbf{E}(X)$, we would get that
		\begin{align*}
			\norm{\left\{ \norm{f\chi_{\Delta_j}}_{L_\infty} \right\}_{j \in \mathbb{Z}}}_{\widetilde{\mathbf{E}(X)}}
				& \leqslant \norm{\mathscr{T}}_{\mathbf{E}(X) \rightarrow \mathbf{E}(X)} \norm{\left\{\norm{f\chi_{\Delta_j}}_{L_\infty}\right\}_{j \in \mathbb{Z}}}_{\mathbf{E}(X)} \\
				& = \max \left\{ \norm{\mathscr{T}}_{{\bf E}(\ell_1) \rightarrow {\bf E}(\ell_1)}, \norm{\mathscr{T}}_{{\bf E}(\ell_{\infty}) \rightarrow {\bf E}(\ell_{\infty})} \right\}
					\vertiii{f}_{\widetilde{X}} \\
				& \leqslant 2 \vertiii{f}_{\widetilde{X}},
		\end{align*}
		and this will end the proof of \eqref{Step 1: d_BM <= 2}. Therefore, the only thing left is to show
		the boundedness of $\mathscr{T}$ on $\mathbf{E}(X)$. To see this, note that the operator $\mathscr{T}$ is bounded on $\mathbf{E}(L_\infty)$, because
		$\mathbf{E}(L_\infty) \equiv \ell_\infty$ and $\widetilde{\ell_\infty} \equiv \ell_\infty$. Moreover, due to \cite[Proposition~2.1]{GHS96}
		(see also \cite[Theorem~2.3, p.~12]{GE98}), we have
		\begin{equation*}
			\sum_{j \in \mathbb{Z}} 2^j \left(\sup\limits_{k \geqslant j} \abs{x_k} \right) \leqslant 2 \sum_{j \in \mathbb{Z}} 2^j \abs{x_j}.
		\end{equation*}
		This clearly means that
		\begin{equation*}
			\mathbf{E}(L_1) \equiv \ell_1(w) \overset{2}{\hookrightarrow} \widetilde{\ell_1(w)} \equiv \widetilde{\mathbf{E}(L_1)},
		\end{equation*}
		where $w(j) = 2^{j}$ for $j \in \mathbb{Z}$. In consequence, the operator $\mathscr{T}$ is also bounded on $\mathbf{E}(L_1)$.
		Since $(L_1,L_{\infty})$ form the Calder{\' o}n--Mityagin couple, so $X = \left( L_1, L_{\infty} \right)^{\mathscr{K}}_{\mathscr{X}}$
		for some parameter space $\mathscr{X}$. Thus, in view of Theorem~\ref{Proposition: E komutuje z interpolacja}, we have
		\begin{equation*}
			\mathbf{E}(X)
				= {\bf E}\left[ (L_1,L_{\infty})^{\mathscr{K}}_{\mathscr{X}} \right]
				= \left( \mathbf{E}(L_1), \mathbf{E}(L_{\infty}) \right)^{\mathscr{K}}_{\mathscr{X}}
				= \left( \ell_1(w),\ell_{\infty} \right)^{\mathscr{K}}_{\mathscr{X}}.
		\end{equation*}
		All this together shows that the operator $\mathscr{T}$ is bounded on $\mathbf{E}(X)$ with
		$\norm{\mathscr{T}}_{\mathbf{E}(X) \rightarrow \mathbf{E}(X)} \leqslant 2$. This last observation ends the proof.
	\end{proof}

	\begin{remark}[About Theorems~\ref{Thm: Tandori sequence representation} and \ref{Thm: Tandori function representation}] \label{REMARK: EX izomorficzne z EaX}
		Careful examination of the proofs of Theorems~\ref{Thm: Tandori sequence representation} and \ref{Thm: Tandori function representation} shows
		that they remain true (with perhaps different constants involved) if we replace the dyadic intervals $[2^j,2^{j+1})$ by the intervals of the form
		$[a^j,a^{j+1})$ for some $a > 1$.
		This is not just a coincidence, and indeed the choice of dyadic blocks $\{\Delta_j\}_{j \in \mathbb{J}}$ for discretization construction $X \rightsquigarrow {\bf E}(X)$
		is of no importance. To make this statement precise, fix $a > 1$. Then, by analogy with Definition~\ref{Def: E(X) function} (and for the purposes of this note only),
		one can construct the {\bf $\square$-discretization} of $X$, denoted henceforth by ${\bf E}_a(X)$, by setting
		\begin{equation*}
			\mathbf{E}_a(X) \coloneqq \left\{ x = \{x_j\}_{j \in \mathbb{J}} \colon \sum_{j \in \mathbb{J}} x_j \chi_{\square_j} \in X \right\},
		\end{equation*}
		with the norm  $\norm{x}_{\mathbf{E}_a(X)} \coloneqq \norm{\sum_{j \in \mathbb{J}} x_j \chi_{\square_j}}_X$, where $\square_j \coloneqq [a^j,a^{j+1})$ for $j \in \mathbb{J}$.
		Now, let $X$ be a r.i. function space. We claim that for any $a > 1$ both spaces ${\bf E}(X)$ and ${\bf E}_a(X)$ are isometrically isomorphic.
		The proof goes like this. Let $W \colon (0,\infty) \rightarrow (0,\infty)$ be a unique continuous function that linearly maps $\Delta_j$ onto $\square_j$
		for each $j \in \mathbb{J}$. Let us define $T \colon {\bf E}(X) \rightarrow {\bf E}_a(X)$ as follows
		\begin{equation*}
			T \colon f \mapsto f \circ W^{-1}.
		\end{equation*}
		Then, for $f = \sum_{j \in \mathbb{J}} x_j \chi_{\Delta_j}$ from ${\bf E}(X)$, we have
	 	\begin{equation*}
	 		\norm{T(f)}_{{\bf E}(X)}
	 			= \norm{\sum_{j \in \mathbb{J}} x_j \chi_{\Delta_j} \circ W^{-1}}_X
	 			= \norm{\sum_{j \in \mathbb{J}} x_j \chi_{W(\Delta_j)}}_X
	 			= \norm{\sum_{j \in \mathbb{J}} x_j \chi_{\square_j}}_X
	 			= \norm{f}_{{\bf E}_a(X)}.
	 	\end{equation*}
 		This proves our claim.
 		The roots of the above observation go deep into \cite{GE98} (see, for example, \cite[p.~20]{GE98}).
 		\demo
	\end{remark}

	\begin{warning}
		In the limiting case, that is, taking $a = 1$ in the construction $X \rightsquigarrow {\bf E}_a(X)$ from the above Remark~\ref{REMARK: EX izomorficzne z EaX},
		we can obtain a space, say ${\bf E}_1(X)$, which is not isomorphic to ${\bf E}(X)$. Formally, the construction of the space ${\bf E}_1(X)$ should be understood as
		follows: ${\bf E}_1(X)$ is a vector space of all sequences $x = \{x_n\}_{n=1}^{\infty}$ such that $\sum_{n=1}^{\infty} x_n \chi_{\diamondsuit_n} \in X$, where $\diamondsuit_n \coloneqq [n,n+1)$ for $n \in \mathbb{N}$,
		equipped with the norm $\norm{x}_{{\bf E}_1(X)} \coloneqq \norm{\sum_{n=1}^{\infty} x_n \chi_{\diamondsuit_n}}_X$. For example, let $X$ be a Lorentz space $\Lambda_{\varphi}$
		with non-trivial Boyd indices. Then, as can easily be seen, ${\bf E}_1(\Lambda_{\varphi}) \equiv \lambda_{\varphi}$. On the other hand, due to Proposition~\ref{PROP: EX <-> carrier},
		${\bf E}(\Lambda_{\varphi}) \approx \ell_1$. Clearly\footnote{One can, for example, use the fact that $\ell_1$ has the Schur property,
		while $\{e_n\}_{n=1}^{\infty}$ converges weakly to zero in $\lambda_{\varphi}$.}, both spaces $\lambda_{\varphi}$ and $\ell_1$ are not isomorphic.
		However, it may also happen that the spaces ${\bf E}_1(X)$ and ${\bf E}(X)$ are isomorphic. For example,
		\begin{equation*}
			{\bf E}_1(L_p) \equiv \ell_p \approx \ell_p(w_p) \equiv {\bf E}(L_p),
		\end{equation*}
		where $w_p(j) = 2^{j/p}$ for $j \in \mathbb{Z}$ (cf. Example~\ref{EXAMPLE : ELp computations}).
	\end{warning}

	Recall that under some mild assumptions the space $\mathscr{C}X$ coincide (up to an equivalence of norms) with $X^{\downarrow}$ (see \cite{Si01}; cf. \cite{Si07}).
	This observation, together with Theorems~\ref{Thm: Tandori sequence representation} and \ref{Thm: Tandori function representation}, leads to the following
	
	\begin{corollary}[Block form representation of Ces{\' a}ro spaces] \label{COR: blocking technique CX}
		{\it Let $X$ be a Banach sequence or function space with the Fatou property. Suppose that either $X$ has non-trivial Boyd indices
		or $X$ is rearrangement invariant and $\alpha_X > 1$. Then, we have}
		\begin{equation*}
			\norm{f}_{\mathscr{C}X} \approx \norm{ \left\{ 2^{-j} \norm{f\chi_{\Delta_j}}_{L_{1}} \right\}_{j \in \mathbb{J}}  }_{\mathbf{E}(X)}.
		\end{equation*}
	\end{corollary}
	
	\subsection{Good old friends $ces_p$ and $Ces_p$}
	
	It may be instructive to illustrate Theorems~\ref{Thm: Tandori sequence representation} and \ref{Thm: Tandori function representation}
	(so also Corollary~\ref{COR: blocking technique CX}) using the example of $L_p$ spaces. As we already said, the block form representations
	of classical Ces{\' a}ro spaces were known to Grosse-Erdman (see \cite[Theorem~4.1, p.~22]{GE98}). Many authors have re-proved some
	of these results on various occasions (see, for example, \cite[Proposition~2]{Ast17}, \cite[Corollary~3.5]{ALM19}, \cite[Theorem~9]{AM09} and \cite[Corollary~3.4]{Wal20}).
	
	\begin{corollary}[K. G. Grosse-Erdmann, 1998] \label{COR: reprezentacje ces_p i Ces_p}
		{\it Let $1 \leqslant p \leqslant \infty$. Suppose\footnote{Note that both spaces $ces_1$ and $Ces_1$ are trivial while the spaces $\widetilde{\ell_{\infty}}$
		and $\widetilde{L_{\infty}}$ coincide (up to the equality of norms) with $\ell_{\infty}$ and $L_{\infty}$, respectively.} that $p \neq 1$ [resp. $p \neq \infty$]
		whenever we are talking about Ces{\' a}ro [resp. Tandori] spaces. We have}
		\begin{equation*}
			\norm{x}_{ces_p} \approx \left( \sum_{j=0}^{\infty} 2^{j(1-p)} \left[ \sum_{k = 2^j}^{2^{j+1}-1} \abs{x_k} \right]^p \right)^{1/p}, \quad
				\norm{f}_{Ces_p} \approx \left( \sum_{j \in \mathbb{Z}} 2^{j(1-p)} \left[ \int_{\Delta_j} \abs{f(t)} dt \right]^p \right)^{1/p},
		\end{equation*}
		\begin{equation*}
			\norm{x}_{{\widetilde{\ell_p}}} \approx \left( \sum_{j=0}^{\infty} 2^{-j} \left[ \sum_{k = 2^j}^{2^{j+1}-1} \abs{x_k} \right]^p \right)^{1/p}
				\quad \textit{ and } \quad
			\norm{f}_{\widetilde{L_p}} \approx \left( \sum_{j \in \mathbb{Z}} 2^{-j} \left[ \esssup_{t \in \Delta_j} \abs{f(t)} \right]^p \right)^{1/p}.
		\end{equation*}
		{\it Moreover,}
		\begin{equation*}
			ces_p \approx \biggl( \bigoplus_{n=1}^\infty \ell^{n}_1 \biggr)_{\ell_p},
				\quad Ces_p \approx \biggl( \bigoplus_{n=1}^\infty L_1(0,1) \biggr)_{\ell_p},
		\end{equation*}
		\begin{equation*}
			\widetilde{\ell_p} \approx \biggl( \bigoplus_{n=1}^\infty \ell^{n}_{\infty} \biggr)_{\ell_p}
				\quad \textit{ and } \quad
					\widetilde{L_p} \approx \biggl( \bigoplus_{n=1}^\infty L_{\infty}(0,1) \biggr)_{\ell_p}.
		\end{equation*}
	\end{corollary}
	\begin{proof}
		Fix $1 < p \leqslant \infty$. Since Hardy's operator $\mathscr{H}$ is bounded on $\ell_p$, so it follows
		from Corollary~\ref{COR: blocking technique CX} that
		\begin{equation*}
			\norm{x}_{ces_p} \approx \norm{ \left\{ 2^{-j} \sum_{k = 2^j}^{2^{j+1}-1} \abs{x_k} \right\}_{j=0}^{\infty} }_{\mathbf{E}(\ell_p)}
			= \norm{\sum_{j=0}^{\infty} 2^{-j} \biggl( \sum_{k = 2^j}^{2^{j+1}-1} \abs{x_k} \biggr) \chi_{\Delta_j}}_{\ell_p}.
		\end{equation*}
		Now, since $\#(\Delta_j) = 2^j$ for $j \in \mathbb{Z}_+$, so straightforward calculations shows that
		\begin{equation*}
			\norm{x}_{ces_p} \approx \left( \sum_{j=0}^{\infty} 2^{-jp} \biggl( \sum_{k = 2^j}^{2^{j+1}-1} \abs{x_k} \biggr)^p 2^j \right)^{1/p}
			= \left( \sum_{j=0}^{\infty} 2^{j(1-p)} \biggl( \sum_{k = 2^j}^{2^{j+1}-1} \abs{x_k} \biggr)^p \right)^{1/p}.
		\end{equation*}
		In consequence, we have
		\begin{equation*}
			ces_p
				\approx \Bigl( \bigoplus_{j=0}^\infty \ell^{2^j}_1 \Bigr)_{\ell_p(2^{j(1-p)})}
				\approx \Bigl( \bigoplus_{j=0}^\infty \ell^{2^j}_1 \Bigr)_{\ell_p}
				\approx \Bigl( \bigoplus_{n=1}^\infty \ell^{n}_1 \Bigr)_{\ell_p},
		\end{equation*}
		where the last isomorphism is due to Proposition~\ref{Prop: Podstawowe wlasnosci sum prostych}(a).
		
		The proof in the remaining cases is completely analogous, so let us stop there.
	\end{proof}

	\subsection{Lorentz and Marcinkiewicz's setting}
	
	Here we shift our interest to Ces{\' a}ro spaces built upon Lorentz and Marcinkiewicz spaces.
	Let us start by re-proving some results from \cite{ALM19}.
	
	\begin{corollary}[S. V. Astashkin, K. Le{\' s}nik and L. Maligranda, 2017] \label{PROPOSITION: ALM19}
		{\it Let $\varphi$ be an increasing concave function on $[0,\infty)$ with non-trivial dilation exponents. Then, we have}
		\begin{align*}
			\widetilde{\lambda_\varphi} \approx \Bigl( \bigoplus_{n=1}^{\infty} \ell_{\infty}^n \Bigr)_{\ell_1},
				\quad \mathscr{C} m_\varphi \approx \ell_\infty(\ell_1),
			\quad \quad \mathscr{C} \lambda_\varphi \approx \ell_1
				\quad \textit{ and } \quad \widetilde{m_\varphi} \approx \ell_\infty;
		\end{align*}
		\begin{equation*}
			\widetilde{\Lambda_\varphi} \approx \ell_1(\ell_\infty),
				\quad \mathscr{C} M_\varphi \approx \ell_\infty(\ell_1),
			\quad \mathscr{C} \Lambda_\varphi \approx L_1
				\quad \textit{ and } \quad \widetilde{M_\varphi} \approx \ell_\infty.
		\end{equation*}
	\end{corollary}
	\begin{proof}
		Again, this is essentially a direct consequence of Theorem~\ref{Thm: Tandori sequence representation}, Theorem~\ref{Thm: Tandori function representation}
		and Corollary~\ref{COR: blocking technique CX}. Only a few things require further explanation.
		
		Since, due to Proposition~\ref{PROP: EX <-> carrier}, ${\bf E}(\lambda_\varphi) \approx \ell_1$ and ${\bf E}(m_\varphi) \approx \ell_{\infty}$,
		so using Proposition~\ref{Prop: Podstawowe wlasnosci sum prostych}, we infer that
		\begin{equation*}
			\widetilde{\lambda_\varphi}
				= \Bigl( \bigoplus_{n=1}^{\infty} \ell_{\infty}^{2^n} \Bigr)_{{\bf E}(\lambda_\varphi)}
				\approx \Bigl( \bigoplus_{n=1}^{\infty} \ell_{\infty}^n \Bigr)_{\ell_1},
		\end{equation*}
		\begin{equation*}
			\mathscr{C} \lambda_\varphi
			= \Bigl( \bigoplus_{n=1}^{\infty} \ell_{1}^{2^n} \Bigr)_{{\bf E}(\lambda_\varphi)}
			\approx \Bigl( \bigoplus_{n=1}^{\infty} \ell_{1}^{n} \Bigr)_{\ell_{1}}
			\approx \ell_1
		\end{equation*}
		and
		\begin{equation*}
			\widetilde{m_\varphi}
			= \Bigl( \bigoplus_{n=1}^{\infty} \ell_{\infty}^{2^n} \Bigr)_{{\bf E}(m_\varphi)}
			\approx \Bigl( \bigoplus_{n=1}^{\infty} \ell_{\infty}^n \Bigr)_{\ell_{\infty}}
			\approx \ell_{\infty}.
		\end{equation*}
		Moreover,
		\begin{equation*}
			\mathscr{C} m_\varphi
				= \Bigl( \bigoplus_{n=1}^{\infty} \ell_{1}^{2^n} \Bigr)_{{\bf E}(m_\varphi)}
				\approx \Bigl( \bigoplus_{n=1}^{\infty} \ell_{1}^{n} \Bigr)_{\ell_{\infty}}
				\approx \Bigl( \bigoplus_{n=1}^{\infty} \ell_{1} \Bigr)_{\ell_{\infty}}
				\approx \ell_{\infty}(\ell_1),
		\end{equation*}
		where the second isomorphism follows from the fact that the canonical basis of $\ell_1$ is boundedly complete together with
		Proposition~4.2 from \cite{AA17}.
		
		The second part is analogous except for the following
		\begin{align*}
			\mathscr{C} M_\varphi
				& = \Bigl( \bigoplus_{j \in \mathbb{Z}} L_{1}(\Delta_j) \Bigr)_{{\bf E}(M_\varphi)} \\
				& \approx \Bigl( \bigoplus_{n=1}^{\infty} L_{1}(0,1) \Bigr)_{\ell_{\infty}} \quad (\text{in view of Proposition~\ref{Prop: Podstawowe wlasnosci sum prostych}}) \\
				& \approx \Bigl( \bigoplus_{n=1}^{\infty} \ell_{1}^{n} \Bigr)_{\ell_{\infty}} \quad (\text{by Theorem~5.1 in \cite{ALM19}}) \\
				& \approx \ell_{\infty}(\ell_1) \quad (\text{using Proposition~4.2 from \cite{AA17}})
		\end{align*}
	\end{proof}
	
	Let $w \colon (0,\infty) \rightarrow (0,\infty)$ be a non-increasing weight.
	Set $\omega(A) = \int_A wdm$ for any Lebesgue measurable set $A$.
	Further, let $X$ be a rearrangement invariant function space with the Fatou property. By the {\bf generalized weighted space $X_w$}
	induced by the space $X$ and the weight $w$ we understand a vector space
	\begin{equation*}
		X_w \coloneqq \left\{ f \in L_0(\omega) \colon f^{\star,w} \in X \right\}
	\end{equation*}
	with the functional $\norm{f}_{X_w} \coloneqq \norm{f^{\star,w}}_X$, where $f^{\star,w}$ is the non-increasing rearrangement of $f$
	with respect to the measure $\omega$. Suppose that $W(t) < \infty$ for all $t > 0$ (briefly, $W < \infty$), where
	\begin{equation*}
		W(x) = \int_0^x w(t)dt \quad \text{ and } \quad W(\infty) = \int_0^\infty w(t)dt.
	\end{equation*}
	Then the functional $f \rightsquigarrow \norm{f}_{X,w}$ is a norm and $\norm{f}_{X,w} = \norm{f \circ W^{-1}}_X$ (see \cite[Proposition~2.2]{KR19}).
	Following Kami{\' n}ska, Masty{\l}o and Raynaud (see \cite{KM07} and \cite{KR19}), we define the {\bf abstract Lorentz space $\Lambda_{X,w}$}
	as the symmetrization of the space $X_w$, that is,
	\begin{equation*}
		\Lambda_{X,w} \coloneqq X_w^{\bigstar} = \left\{ f \in L_0 \colon f^{\star} \in X_w \right\}
	\end{equation*}
	with the functional $\norm{f}_{\Lambda_{X,w}} \coloneqq \norm{f^{\star}}_{X,w}$. Again, if $W < \infty$, then $f \rightsquigarrow \norm{f}_{\Lambda_{X,w}}$
	is a norm and $\norm{f}_{\Lambda_{X,w}} = \norm{f^{\star} \circ W^{-1}}_X$ (see \cite[Proposition~2.3]{KR19}).
	
	\begin{example}
		Let $1 \leqslant p < \infty$. Suppose that $W < \infty$ and $W(\infty) = \infty$.
		Since
		\begin{equation*}
			\norm{f}_{(L_p)_w} = \left( \int_0^{\infty} (f \circ W^{-1})(t)dt \right)^{1/p} = \left( \int_0^{\infty} \abs{f(t)}^p w(t) dt \right)^{1/p},
		\end{equation*}
		so the space $(L_p)_w$ is isometrically isomorphic to the weighted Lebesgue space $L_p(w^{1/p})$.
		It is also straightforward to see that the corresponding abstract Lorentz space $\Lambda_{L_p,w}$ is nothing else but the Lorentz space $\Lambda_{p,w}$.
		Similarly, the space $(L_M)_w$ coincide with the Musielak--Orlicz space $L_{\Phi}$, where $\Phi(t,u) = M(t)w(u)$.
		This time the corresponding abstract Lorentz space $\Lambda_{L_M,w}$ coincide with the Orlicz--Lorentz space $\Lambda_{M,w}$.
		In other words, the generalized weighted spaces $X_w$ are carrier space for the corresponding abstract Lorentz spaces
		$\Lambda_{X,w}$ (cf. Remark~\ref{EXAMPLE : carrier spaces}).
	\end{example}
	
	It turns out that Corollary~\ref{PROPOSITION: ALM19} has a general counterpart for the Tandori spaces built upon Lorentz spaces.

	\begin{proposition}
		{\it Let $X$ be a rearrangement invariant function space with the Fatou property. Further, let $w$ be a non-increasing weight.
		Suppose that $W < \infty$ and $W(\infty) = \infty$. Then the space $\widetilde{\Lambda_{X,w}}$ is isometrically isomorphic to $\widetilde{X}$.}
	\end{proposition}
	\begin{proof}
		We claim that the mapping
		\begin{equation*}
			T \colon f \rightsquigarrow f \circ W^{-1},
		\end{equation*}
		gives an isomorphism between $\widetilde{X_w}$ and $\widetilde{X}$. Indeed, since the function $W^{-1}$ is increasing, so
		\begin{align*}
			\widetilde{(f \circ W^{-1})}(x)
				& = \esssup_{t \geqslant x} (f \circ W^{-1})(t) \\
				& = \esssup_{W^{-1}(t) \geqslant W^{-1}(x)} (f \circ W^{-1})(t) \\
				& = \esssup_{s \geqslant W^{-1}(x)} f(s) \\
				& = (\widetilde{f} \circ W^{-1})(x).
		\end{align*}
		Therefore,
		\begin{equation*}
			\norm{T(f)}_{\widetilde{X}}
				= \norm{\widetilde{f \circ W^{-1}}}_X = \norm{\widetilde{f} \circ W^{-1}}_X
				= \norm{\widetilde{f}}_{X_w} = \norm{f}_{\widetilde{X_w}}.
		\end{equation*}
		In consequence, we have
		\begin{equation*}
			\widetilde{\Lambda_{X,w}} = \widetilde{X_w^{\bigstar}} = \widetilde{X_w} \approx \widetilde{X}.
		\end{equation*}
	\end{proof}

	\subsection{Averaging the averages}

	As noted by Bennett, the family $\left\{ ces_p \right\}_{1 < p < \infty}$ has a unique feature that distinguishes it from classical sequence spaces, namely,
	a sequence of non-negative terms belongs to the space $ces_p$ precisely when its sequence of averages does (see \cite[Theorem~20.31, p.~121]{Be96};
	see also \cite{CR13} and \cite[Theorem~1]{LM16} for some generalizations). To the best of our knowledge, Bennett's result\footnote{A short reading
	of \cite{Be96} is enough to notice that Bennett's proof is not at all elementary. Indeed, the proof of sufficiency uses Hardy's inequality,
	while the proof of necessity lies on Knopp, Schnee and Hausdorff's theorem on the equivalence of the H{\" o}lder and Ces{\' a}ro summability
	methods of the same order. It should be also noted that a wonderfully simple and direct proof of the necessity
	part was given by Curbera and Ricker in \cite{CR13}.} has not yet been proven using the blocking technique.
	We would now like to show how to do this.
	
	\begin{theorem} \label{PROP: CCX = CX}
		{\it Let $X$ be a Banach sequence or function space with the Fatou property. Suppose that the space $X$ has non-trivial Boyd indices. Then}
			\begin{equation*}
				\mathscr{C}\mathscr{C}X = \mathscr{C}X.
			\end{equation*}
	\end{theorem}
	\begin{proof}
		Since we assumed that the space $X$ has non-trivial Boyd indices, so the maximal operator
		$f \rightsquigarrow [x \rightsquigarrow \frac{1}{x}\int_0^x f^{\star}(t)dt]$ is bounded on $X$ (see \cite[pp.~122--126]{BS88}).
		This means that
		\begin{equation*}
			\left[ \mathscr{C}X \right]^{\bigstar} = X.
		\end{equation*}
		In other words, the carrier space for $X$ is just $\mathscr{C}X$.
		Next, we claim that the space $\mathscr{C}X$ has non-trivial Boyd indices. To see this, observe that the dilation operators $D_s$
		for $s > 0$ commute with Hardy's operator $\mathscr{H}$, that is,
		\begin{equation*}
			\mathscr{H} \abs{D_s(f)} = \frac{1}{x}\int_0^x \abs{f\left( \frac{t}{s} \right)}dt = \frac{1}{x/s} \int_0^{x/s} \abs{f(t)}dt = D_s\mathscr{H}(\abs{f}).
		\end{equation*}
		Therefore,
		\begin{equation*}
			\norm{D_s}_{\mathscr{C}X \rightarrow \mathscr{C}X}
				\leqslant \norm{D_s}_{X \rightarrow X} \norm{\mathscr{H}}_{X \rightarrow X}
				\leqslant \max\{1,s\} \norm{\mathscr{H}}_{X \rightarrow X}.
		\end{equation*}
		Consequently,
		\begin{equation*}
			1 < \alpha_X
				= \sup\limits_{s > 1} \frac{\log s}{\log \norm{D_s}_{X \rightarrow X}}
				\leqslant \sup\limits_{s > 1} \frac{\log s}{\log \norm{D_s}_{\mathscr{C}X \rightarrow \mathscr{C}X}}
				= \alpha_{\mathscr{C}X}
		\end{equation*}
		and, similarly,
		\begin{equation*}
			\beta_{\mathscr{C}X} \leqslant \beta_X < \infty.
		\end{equation*}
		Thus, due to Proposition~\ref{PROP: EX <-> carrier}, we have
		\begin{equation} \label{EQ: EX = ECXs = ECX}
			{\bf E}(X) = {\bf E}(\left[ \mathscr{C}X \right]^{\bigstar}) = {\bf E}(\mathscr{C}X).
		\end{equation}
		In consequence, we have
		\begin{align*}
			\mathscr{C}\mathscr{C}X
				& = \Bigl( \bigoplus_{j \in J} L_1(\Delta_j) \Bigr)_{\mathbf{E}(\mathscr{C}X)} \quad (\text{using Theorems~\ref{Thm: Tandori sequence representation} and \ref{Thm: Tandori function representation}}) \\
				& = \Bigl( \bigoplus_{j \in J} L_1(\Delta_j) \Bigr)_{\mathbf{E}(X)} \quad (\text{by \eqref{EQ: EX = ECXs = ECX}}) \\
				& = \mathscr{C}X \quad (\text{again, using Theorems~\ref{Thm: Tandori sequence representation} and \ref{Thm: Tandori function representation}}).
		\end{align*}
	\end{proof}

	\begin{corollary}[G.~Bennett, 1996]
		{\it Let $1 < p < \infty$. Then}
		\begin{equation*}
			x = \{x_n\}_{n=1}^{\infty} \in ces_p \quad \textit{ if, and only if, } \quad \left\{ \frac{1}{n}\sum_{k=1}^n x_k \right\}_{n=1}^{\infty} \in ces_p.
		\end{equation*}
	\end{corollary}
	
	\section{{\bf Interpolation structure}} \label{SECTION: Interpolation structure}
	
	\subsection{Revisiting Bukhvalov's interpolation formula} \label{SUBSECTION: revisiting Bukhvalov}
	Fix a Banach space $X$ and let $(E,F)$ be a couple of Banach sequence spaces.
	Observation that the mixed-norm construction $E \rightsquigarrow E[X]$ (note that the \enquote{inside}
	space stays fixed) commutes with any interpolation functor $\mathbf{F}$, that is,
	\begin{equation} \label{EQ: Bukhvalov's observation}
		\mathbf{F}(E[X],F[X]) = \mathbf{F}(E,F)[X],
	\end{equation}
	seems to be part of the folklore, at least among interpolation theory specialists.
	Indeed, the above result was already noted in the early 1980's by Alexander V. Bukhvalov (see \cite[p.~95]{Buk84}).
	
	In our opinion, however, Bukhvalov's proof is tainted with some drawbacks, namely,
	$\bullet$ it is quite difficult to access the paper \cite{Buk84} (in fact, our copy comes from Lech Maligranda's private collection);
	$\bullet$ it is written in Russian (and we know no version published in another language);
	$\bullet$ does not abound in details (it is more like a sketch of the proof).
	
	Since Bukhvalov's observation \eqref{EQ: Bukhvalov's observation} is an important step towards what we plan to do here, all this has prompted
	us to write in full detail even a more general\footnote{It should be noted that $\bigl( \bigoplus_{j \in J} X \bigr)_{E} = E(X) = E[X]$, where
	while the first equality is obvious (cf. Remark~\ref{REMARK: Bochner construction}), the second one is not and essentially follows from \cite{Buk79}.}
	version of the aforementioned result.
	
	\begin{theorem}[Interpolation of amalgams] \label{THM: a'la Bukhvalov}
		{\it Let $(E,F)$ be a couple of Banach sequence spaces defined on $J$. Further, let $\{X_j\}_{j \in J}$ be a family of Banach spaces.
			Then, for any interpolation functor $\mathbf{F}$, we have}
		\begin{equation*}
			\mathbf{F}\biggl( \Bigl( \bigoplus_{j \in J} X_j \Bigr)_{E}, \Bigl( \bigoplus_{j \in J} X_j \Bigr)_{F} \biggr)
			= \Bigl( \bigoplus_{j \in J} X_j \Bigr)_{\mathbf{F}(E,F)}.
		\end{equation*}
	\end{theorem}
	\begin{proof}
		Since the category of Banach sequence spaces is interpolation stable, so $\mathbf{F}(E,F)$ is a Banach sequence space (see \cite[Example~2.6.12, p.~252]{BK91}).
		
		Once we know that $\mathbf{F}(E,F)$ is a Banach sequence space, it is time to show the first embedding
		\begin{equation} \label{first emb}
			\norm{\text{id} \colon \mathbf{F}\biggl( \Bigl( \bigoplus_{j \in J} X_j \Bigr)_{E}, \Bigl( \bigoplus_{j \in J} X_j \Bigr)_{F} \biggr)
				\rightarrow \Bigl( \bigoplus_{j \in J} X_j \Bigr)_{\mathbf{F}(E,F)} } < \infty.
		\end{equation}
		Along the way, define the mapping $S \colon \bigl( \bigoplus_{j \in J} X_j \bigr)_{E} + \bigl( \bigoplus_{j \in J} X_j \bigr)_{F} \rightarrow E + F$
		as follows
		\begin{equation*}
			S \colon \{f_j\}_{j \in J} \rightsquigarrow \left\{ \norm{f_j}_{X_j} \right\}_{j \in J}.
		\end{equation*}
		Then $S$ is a sublinear operator that is bounded when considered to act from $\bigl( \bigoplus_{j \in J} X_j \bigr)_{E}$
		into $E$ and from $\bigl( \bigoplus_{j \in J} X_j \bigr)_{F}$ into $F$. However, $\mathbf{F}$ is an interpolation functor, so
		\begin{equation*}
			S \colon \mathbf{F}\biggl( \Bigl( \bigoplus_{j \in J} X_j \Bigr)_{E}, \Bigl( \bigoplus_{j \in J} X_j \Bigr)_{F} \biggr) \rightarrow \mathbf{F}(E,F)
		\end{equation*}
		is bounded. In consequence, if we take $f = \{f_j\}_{j \in J}$ from
		$\mathbf{F}\Bigl( \bigl( \bigoplus_{j \in J} X_j \bigr)_{E}, \bigl( \bigoplus_{j \in J} X_j \bigr)_{F} \Bigr)$, then $S(f)$
		belongs to $\mathbf{F}(E,F)$ and that, in turn, means exactly that $f \in \bigl( \bigoplus_{j \in J} X_j \bigr)_{\mathbf{F}(E,F)}$.
		Therefore, (\ref{first emb}) follows.
		
		It remains to prove the reverse embedding
		\begin{equation} \label{second emb}
			\norm{\text{id} \colon \Bigl( \bigoplus_{j \in J} X_j \Bigr)_{\mathbf{F}(E,F)} \rightarrow
				\mathbf{F}\biggl( \Bigl( \bigoplus_{j \in J} X_j \Bigr)_{E}, \Bigl( \bigoplus_{j \in J} X_j \Bigr)_{F} \biggr) } < \infty.
		\end{equation}
		To do this, take $f = \{f_j\}_{j \in J}$ from $\bigl( \bigoplus_{j \in J} X_j \bigr)_{\mathbf{F}(E,F)}$ and consider the mapping
		$T_f \colon E + F \rightarrow \bigl( \bigoplus_{j \in J} X_j \bigr)_{E} + \bigl( \bigoplus_{j \in J} X_j \bigr)_{F}$ defined via
		\begin{equation*}
			T_f \colon \{a_j\}_{j \in J} \rightsquigarrow \left\{ a_j \frac{f_j}{\norm{f_j}_{X_j}} \right\}_{j \in J}.
		\end{equation*}
		Since for $a = \{a_j\}_{j \in J} \in E$ we have
		\begin{equation*}
			\norm{T_f(a)}_{\left( \bigoplus_{j \in J} X_j \right)_{E}}
			= \norm{ \left\{ \norm{T_f(a)}_{X_j} \right\}_{j \in J}}_E = \norm{\{a_j\}_{j \in J}}_E = \norm{a}_E,
		\end{equation*}
		so $T_f$ is bounded when acting from $E$ into $\bigl( \bigoplus_{j \in J} X_j \bigr)_{E}$. Similarly, $T_f$ is bounded when acting
		from $F$ into $\bigl( \bigoplus_{j \in J} X_j \bigr)_{F}$. But this means that
		\begin{equation*}
			T_f \colon \mathbf{F}(E,F) \rightarrow \mathbf{F}\biggl( \Bigl( \bigoplus_{j \in J} X_j \Bigr)_{E}, \Bigl( \bigoplus_{j \in J} X_j \Bigr)_{F} \biggr)
		\end{equation*}
		is also bounded. Consequently, taking $a = \left\{ \norm{f_j}_{X_j} \right\}_{j \in J}$ and noting that $a \in \mathbf{F}(E,F)$,
		we get
		\begin{equation*}
			f = T_f(a) \in \mathbf{F}\biggl( \Bigl( \bigoplus_{j \in J} X_j \Bigr)_{E}, \Bigl( \bigoplus_{j \in J} X_j \Bigr)_{F} \biggr).
		\end{equation*}
		Due to $f$'s arbitrariness the proof has been completed.
	\end{proof}

	From our perspective, the most important result is the following
	
	\begin{proposition}[Interpolation] \label{PROP: interpolation down, tandori and cesaro}
		{\it Let $X$ and $Y$ be two Banach ideal spaces with the Fatou property both defined on the same measure space.
		Suppose that either both spaces have non-trivial Boyd indices or are rearrangement invariant.
		Then, for any interpolation functor $\mathbf{F}$, we have}
		\begin{equation*}
			\mathbf{F} (\widetilde{X}, \widetilde{Y}) = \widetilde{\mathbf{F}(X,Y)}
				\quad \textit{ and } \quad \mathbf{F} (X^{\downarrow}, Y^{\downarrow}) = \mathbf{F}(X,Y)^{\downarrow}.
		\end{equation*}
	\end{proposition}
	\begin{proof}
		The argument boils down to juggling Theorems~\ref{Proposition: E komutuje z interpolacja}, \ref{Thm: Tandori sequence representation},
		\ref{Thm: Tandori function representation} and \ref{THM: a'la Bukhvalov}. For example,
		\begin{align*}
			\mathbf{F} (\widetilde{X}, \widetilde{Y})
				& = \mathbf{F}\biggl( \Bigl( \bigoplus_{j \in \mathbb{J}} L_{\infty}(\Delta_j) \Bigr)_{{\bf E}(X)}, \Bigl( \bigoplus_{j \in \mathbb{J}} L_{\infty}(\Delta_j) \Bigr)_{{\bf E}(Y)} \biggr)
					\quad (\text{by Theorems~\ref{Thm: Tandori sequence representation} and \ref{Thm: Tandori function representation}})\\
				& = \Bigl( \bigoplus_{j \in \mathbb{J}} L_{\infty}(\Delta_j) \Bigr)_{\mathbf{F}({\bf E}(X),{\bf E}(Y))} \quad (\text{using Theorem~\ref{THM: a'la Bukhvalov}}) \\
				& = \Bigl( \bigoplus_{j \in \mathbb{J}} L_{\infty}(\Delta_j) \Bigr)_{{\bf E}(\mathbf{F}(X,Y))} \quad (\text{in view of Theorem~\ref{Proposition: E komutuje z interpolacja}}) \\
				& = \widetilde{\mathbf{F}(X,Y)} \quad (\text{again, by Theorems~\ref{Thm: Tandori sequence representation} and \ref{Thm: Tandori function representation}}).
		\end{align*}
		However, in the case of down spaces, one more hint seems to be useful. Since, due to Theorems~\ref{Thm: Tandori sequence representation}
		and \ref{Thm: Tandori function representation}, $X^{\downarrow} = \bigl( \bigoplus_{j \in \mathbb{J}} L_{1}(\Delta_j) \bigr)_{{\bf E}(X)(w)}$,
		where $w(j) = 2^{-j}$ for $j \in \mathbb{J}$, so one can easily avoid the annoying problem with the homogeneity of the functor ${\bf F}$, that is,
		the additional assumption that ${\bf F}({\bf E}(X)(w),{\bf E}(Y)(w)) = {\bf F}({\bf E}(X),{\bf E}(Y))(w)$, by simply noting that
		\begin{equation} \label{EQ: zauwazka}
			X^{\downarrow} = \Bigl( \bigoplus_{j \in \mathbb{J}} L_{1}(\Delta_j) \Bigr)_{{\bf E}(X)(w)} \equiv \Bigl( \bigoplus_{j \in \mathbb{J}} L_{1}(\Delta_j, W_j) \Bigr)_{{\bf E}(X)},
		\end{equation}
		where $W_j(t) = 2^{-j}$ for $t \in \Delta_j$ and $j \in \mathbb{J}$, while $L_{1}(\Delta_j, W_j)$ is the weighted $L_1(\Delta_j)$ space
		defined by the norm $\norm{f}_{L_{1}(\Delta_j, W_j)} \coloneqq \norm{2^{-j}f}_{L_1(\Delta_j)}$.
		After this observation, the proof can be completed as above.
	\end{proof}

	\begin{corollary} \label{COR: interpolation F(CX,CY)}
		{\it Let $X$ and $Y$ be two Banach ideal spaces with the Fatou property both defined on the same measure space.
		Suppose that either both spaces $X$ and $Y$ have non-trivial Boyd indices or are rearrangement invariant
		and Hardy's operator $\mathscr{H}$ is bounded on both of them. Then}
		\begin{equation*}
			\mathbf{F} \left( \mathscr{C}X,\mathscr{C}Y \right) = \mathscr{C} \mathbf{F}(X,Y).
		\end{equation*}
	\end{corollary}
	
	\begin{remark}[About Corollary~\ref{COR: interpolation F(CX,CY)}]
		The above result unifies and generalizes the main interpolation results obtained by Astashkin and Maligranda \cite{AM13}, and by Le{\' s}nik and Maligranda \cite{LM16}
		(precisely, see \cite[Proposition~3.1]{AM13} and \cite[Theorem~6]{LM16}). In the first case, the authors assume that the both spaces $X$ and $Y$
		are rearrangement invariant function spaces (and, moreover, rely heavily on earlier Masty{\l}o and Sinnamon's work \cite{MS06}), while in the second they
		assume that the functor ${\bf F}$ is homogeneous, that is, ${\bf F}(X(w),Y(w)) = {\bf F}(X,Y)(w)$ for any weight $w$ (we easily avoided the same problem; see \eqref{EQ: zauwazka}).
		\demo
	\end{remark}

	Let us go back to Theorem~\ref{THM: a'la Bukhvalov} for a moment. It is known that in general one should not expect the equality
	\begin{equation*}
		\mathbf{F}(E[X],F[Y]) = \mathbf{F}(E,F)[\mathbf{F}(X,Y)]
	\end{equation*}
	to hold (see \cite{Buk87}; see also \cite{Mal04} and references therein). Thus, in this sense, Theorem~\ref{THM: a'la Bukhvalov} is optimal.
	However, if we limit ourselves to considering the Calder{\' o}n product (instead of any interpolation functor), then we know that
	\begin{equation} \label{EQ: Calderon product wszystko sie zmienia}
		E[X]^{1-\theta}F[Y]^{\theta} = E^{1-\theta}F^{\theta}[X^{1-\theta}Y^{\theta}]
	\end{equation}
	(see \cite[Theorem~3]{Buk87}). Just like Theorem~\ref{THM: a'la Bukhvalov} extends \eqref{EQ: Bukhvalov's observation}, we now want to show
	that the same can be done with \eqref{EQ: Calderon product wszystko sie zmienia}.
	
	\begin{theorem}[Calder{\' o}n product of amalgams] \label{THM: Calderon product of amalgams} \label{THM: Calderon product amalgams}
		{\it Let $E$ and $F$ be two Banach sequence spaces both defined on $J$. Further, let $\{X_j\}_{j \in J}$ and $\{Y_j\}_{j \in J}$
			be two families of Banach ideal spaces. Suppose that for any $j \in J$ both spaces $X_j$ and $Y_j$ are defined on the same measure space
			and have the Fatou property. Then, for $0 < \theta < 1$, we have the following identification}
		\begin{equation*}
			\left[ \Bigl( \bigoplus_{j \in J} X_j \Bigr)_E \right]^{1-\theta} \left[ \Bigl( \bigoplus_{j \in J} Y_j \Bigr)_F \right]^{\theta}
				= \Bigl( \bigoplus_{j \in J} X_j^{1-\theta} Y_j^{\theta} \Bigr)_{E^{1-\theta}F^{\theta}}.
		\end{equation*}
	\end{theorem}
	\begin{proof}
		It is enough to slightly modify Bukhvalov's proof of Theorem~3 from \cite{Buk87} (see also \cite[Theorem~15.12(a), p.~186]{Ma89}
		together with the proofs of Corollaries~2 and 3 in \cite{Mal04}).
	\end{proof}
	
	\begin{remark}[About Theorem~\ref{THM: Calderon product of amalgams}] \label{REMARK: Maligranda on CL-construction}
		Let $X$ and $Y$ be two Banach ideal spaces with the Fatou property.
		Note that the above result is optimal in the following sense: Lech Maligranda (see \cite[Corollary~4]{Mal04}) was able to show that
		the Calder{\' o}n--Lozanovski{\u \i} construction $\varrho(\cdot,\cdot)$ commutes with the mixed norm spaces, that is,
		\begin{equation*}
			\varrho(E[X],F[Y]) = \varrho(E,F)[\varrho(X,Y)],
		\end{equation*}
		if, and only if, $\varrho$ is equivalent to the power function $\varrho_{\theta}(s,t) = s^{1-\theta}t^{\theta}$ for some $0 < \theta < 1$.
		In this situation, of course, the space $\varrho_{\theta}(X,Y)$ coincides with the Calder{\' o}n product $X^{1-\theta}Y^{\theta}$.
		\demo
	\end{remark}
	
	\subsection{On Calder{\' o}n--Mityagin couples}
	Let $\mathscr{I}$ be a family of interpolation functors. Recall, following Brudny{\u \i} and Krugljak (see \cite[Definition~4.4.1, p.~578]{BK91}),
	that the couple $\vv{X}$ is said to be {\bf $\mathscr{I}$-adequate to a couple} $\vv{Y}$ if for any interpolation functor $\mathbf{F}$
	there is a functor $\mathbf{J}$ from the family $\mathscr{I}$ such that
	$\mathbf{F}(\vv{X}) \hookrightarrow \mathbf{J}(\vv{X})$ and $\mathbf{J}(\vv{Y}) \hookrightarrow \mathbf{F}(\vv{Y})$.
	In particular, if both couples $\vv{X}$ and $\vv{Y}$ coincide, then the couple $\vv{X}$ is referred to as {\bf $\mathscr{I}$-adequate}.
	
	Since, due to Aronszajn and Gagliardo's theorem (see \cite[Theorem~2.5.1, p.~29]{BL76}), for any interpolation space $X$ with respect to
	the couple $\vv{X}$ there is an interpolation functor $\mathbf{F}$ such that $\mathbf{F}(\vv{X}) = X$, so the above definition roughly
	means that the couple $\vv{X}$ is $\mathscr{I}$-adequate if the functor $\mathbf{F}$ can be chosen from the family $\mathscr{I}$.
	At the same time, one can generate all interpolation spaces of $\vv{X}$ using only the functors from the family $\mathscr{I}$.
	For quite obvious reasons, the problem of describing all interpolation spaces of a given couple is called the {\it \enquote{basic problem}}
	(for the interpolation theory) and the concept of the $\mathscr{I}$-adequate couple sort of \enquote{solves} it.
	However, even assuming that we know that a given couple is $\mathscr{I}$-adequate, without a reasonable choice of the family
	$\mathscr{I}$ this \enquote{solution} has no practical meaning.
	
	According to Brudny{\u \i} and Krugljak (see \cite[p.~251]{BK91}), the only reasonable (they prefer the word \enquote{prolific})
	choice is the family of the real method functors $\mathscr{K} = \left\{ (\cdot,\cdot)^{\mathscr{K}}_{\mathscr{X}} \right\}$.
	Even then, however, showing that a given couple is $\mathscr{K}$-adequate is in general a very difficult problem.
	Priority in this respect goes to Alberto Calder{\' o}n \cite{Cal66} and, independently, Boris Mityagin \cite{Mit65}.
	They were able to show that the couple $(L_1,L_{\infty})$ is $\mathscr{K}$-adequate, that is, any interpolation space with respect
	to $(L_1,L_{\infty})$ can be described by the real $\mathscr{K}$-method of interpolation (cf. \cite[Theorem~2.6.9, p.~251]{BK91}).
	In retrospect, there is no doubt that this discovery has a major impact on the study of interpolation spaces, especially in a more
	abstract setting.
	
	Henceforth, the $\mathscr{K}$-adequate couples are called the {\bf Calder{\' o}n--Mityagin couples}\footnote{Many authors use different nomenclature
	to name these, for example, {\it Calder{\' o}n couples}, {\it $\mathscr{K}$-adequate couples}, {\it $\mathscr{K}$-monotone couples}, {\it $\mathscr{C}$-couples}
	or {\it $\mathscr{C}\mathscr{M}$-couples}.}.
	Further, we will say that the couple $\vv{X}$ is the {\bf relative Calder{\' o}n--Mityagin couple} with respect to $\vv{Y}$, provided
	$\vv{X}$ is $\mathscr{K}$-adequate to a couple $\vv{Y}$.
	
	\begin{remark}[On Calder{\' o}n--Mityagin couples]
		Following the pioneering work of Calder{\' o}n \cite{Cal66} and Mityagin \cite{Mit65}, the plethora of examples of Calder{\' o}n--Mityagin
		couples have been detected among classical Banach function spaces including pairs of
		$\bullet$ (weighted) Lebesgue spaces $L_p$;
		$\bullet$ Orlicz spaces $L_M$;
		$\bullet$ Lorentz spaces $L_{p,q}$ and $\Lambda_{\varphi}$;
		$\bullet$ Marcinkiewicz spaces $L_{p,\infty}$ and $M_{\varphi}$,
		and so on... (see Arazy and Cwikel \cite{AC84}, Cerd{\` a} and Mart{\' i}n \cite{CM01}, Cwikel \cite{Cwi84}, Kalton \cite{Kal92},
		Lorentz and Shimogaki \cite{LS71}, Nilsson \cite{Nil85} and Spaar \cite{Spa78};
		see also Cwikel, Nilsson and Schechtman's memoir \cite{CNS03} for more information).
		\demo
	\end{remark}
	
	We will say, following Kalton \cite{Ka93a}, that a r.i. space $X$ is {\bf stretchable} if the space ${\bf E}(X)$ has the {\bf right-shift property}
	meaning that there is a constant $C > 0$ such that for any pair of finite normalized sequences from ${\bf E}(X)$, say $\{x_j\}_{n=1}^N$ and $\{y_j\}_{n=1}^N$, with
	\begin{equation*}
		\supp x_1 < \supp y_1 < \supp x_2 < ... < \supp x_N < \supp y_N
	\end{equation*}
	and any sequence of scalars, say $\{a_n\}_{n=1}^{N}$, we have
	\begin{equation*}
		\norm{\sum_{n=1}^N a_j y_j}_{{\bf E}(X)} \leqslant C \norm{\sum_{n=1}^N a_j x_j}_{{\bf E}(X)}.
	\end{equation*}
	In other words, the space ${\bf E}(X)$ is stretchable provided the right-shift operator is uniformly bounded on the closed linear span of every block basic sequence
	with respect to the canonical basis. Similarly, we will say that a r.i. space $X$ is {\bf compressible} if ${\bf E}(X)$ has the corresponding {\bf left-shift property}.
	Finally, if a r.i. space $X$ is both stretchable and compressible, we will say that $X$ is {\bf elastic}.
	
	\begin{example}[Elastic spaces] \label{EXAMPLE: elastic spaces}
		Let $1 < p,q < \infty$.
		Then the spaces $L_p$, $L_{p,q}$, $L_p \cap L_q$ and $L_p + L_q$ are elastic.
		Moreover, one can give rather complicated conditions guaranteeing that the Orlicz space $L_M$ is elastic (see \cite[Theorem~6.8]{Kal92} for the details).
	\end{example}
	
	The next result states, roughly speaking, that both interpolation scales, the one induced by the construction $X \rightsquigarrow \widetilde{X}$
	and the dual one obtained by the construction $X \rightsquigarrow X^{\downarrow}$, are parallel to the scale of rearrangement invariant spaces.
	Pragmatically, this result can be seen as a qualitative version of Masty{\l}o and Sinnamon's results from \cite{MS17}
	(precisely, see Theorem~3.5 and Corollary~3.10 ibidem; see also \cite{HS23} for some generalities).
	
	\begin{theorem}[Calder{\' o}n--Mityagin couples] \label{THM: CM-couples}
		{\it Let $X$ and $Y$ be two Banach ideal spaces with the Fatou property both defined on the same measure space.
		Suppose that either both spaces have non-trivial Boyd indices or are rearrangement invariant.
		Then the couples $(X,Y)$, $(\widetilde{X},\widetilde{Y})$ and $(X^{\downarrow},Y^{\downarrow})$ form a Calder{\' o}n--Mityagin
		couples if, and only if, one of them form a Calder{\' o}n--Mityagin couple.
		In particular, the couples $(\widetilde{X}, L_{\infty})$ and $(L_1,Y^{\downarrow})$ form a Calder{\' o}n--Mityagin couple
		if, and only if, $X$ is stretchable and, respectively, $Y$ is compressible.}
	\end{theorem}
	\begin{proof}
		Let $\mathbf{F}$ be any interpolation functor. Suppose that $(X,Y)$ is a Calder{\' o}n--Mityagin couple. Then
		\begin{align*}
			\mathbf{F} \left( \widetilde{X}, \widetilde{Y} \right)
				& = \widetilde{\mathbf{F} \left( X, Y \right)} \quad (\text{by Proposition~\ref{PROP: interpolation down, tandori and cesaro}}) \\
				& = \widetilde{\left( X, Y \right)^{\mathscr{K}}_{\mathscr{X}}} \quad (\text{since $(X,Y)$ is a Calder{\' o}n--Mityagin couple})\\
				& = ( \widetilde{X}, \widetilde{Y} )^{\mathscr{K}}_{\mathscr{X}} \quad (\text{again, by Proposition~\ref{PROP: interpolation down, tandori and cesaro}}),
		\end{align*}
		which means that $(\widetilde{X}, \widetilde{Y})$ is a Calder{\' o}n--Mityagin couple as well.
		
		Next, suppose that $(\widetilde{X}, \widetilde{Y})$ is a Calder{\' o}n--Mityagin couple. Then, referring again to
		Proposition~\ref{PROP: interpolation down, tandori and cesaro}, it is enough to repeat the above argument to show that
		\begin{equation} \label{EQ: functor tandori = functor K}
			\widetilde{{\bf F}(X,Y)} = \widetilde{\left( X, Y \right)^{\mathscr{K}}_{\mathscr{X}}}.
		\end{equation}
		Further, we require the following simple
		
		\begin{claim} \label{LEMMA: X=Y<=>TandoriX=TandoriY}
			{\it Let $Z$ and $W$ be two r.i. spaces. Then $Z = W$ if, and only if, $\widetilde{Z} = \widetilde{W}$.}
		\end{claim}
		\begin{proof}[Proof of Clam~\ref{LEMMA: X=Y<=>TandoriX=TandoriY}]
			As for the proof of the first implication, because $Z \hookrightarrow W$ undoubtedly guarantees that $\widetilde{Z} \hookrightarrow \widetilde{W}$,
			so clearly $Z = W$ implies that $\widetilde{Z} = \widetilde{W}$. On the other hand, if $Z \neq W$, then we can find a function, say $f$, such that
			$f \in Z \setminus W$ (if it happens to be the other way around, just swap the roles of $Z$ and $W$). Since both spaces are rearrangement invariant,
			so $f^{\star} \in Z \setminus W$. However, clearly $\widetilde{f^{\star}} = f^{\star}$, so $f \in \widetilde{Z} \setminus \widetilde{W}$. But this
			means that $\widetilde{Z} \neq \widetilde{W}$.
		\end{proof}
		
		We can now complete the proof. Remembering about equality \eqref{EQ: functor tandori = functor K} and referring to Claim~\ref{LEMMA: X=Y<=>TandoriX=TandoriY},
		we see that
		\begin{equation*}
			\widetilde{{\bf F}(X,Y)} = \widetilde{\left( X, Y \right)^{\mathscr{K}}_{\mathscr{X}}}
				\quad \text{ if, and only if, } \quad {\bf F}(X,Y) = (X,Y)^{\mathscr{K}}_{\mathscr{X}}.
		\end{equation*}
		This, however, means that $(X,Y)$ is a Calder{\' o}n--Mityagin couple.
		
		As for the remaining part, since $(\widetilde{X},L_{\infty}) = (\widetilde{X},\widetilde{L_{\infty}})$ and $(L_1,Y^{\downarrow}) = (L_1^{\downarrow},Y^{\downarrow})$,
		so the couples $(\widetilde{X},L_{\infty})$ and $(L_1,Y^{\downarrow})$ form Calder{\' o}n--Mityagin couples if, and only if, $X$ is stretchable and, respectively,
		$Y$ is compressible (see \cite[Theorem~5.4]{Kal92} and, respectively, \cite[Theorem~5.5]{Kal92}).
	\end{proof}

	The relativized version of Theorem~\ref{THM: CM-couples} is as follows (cf. \cite[Theorem~3.9]{MS17}).

	\begin{theorem}[Relative Calder{\' o}n--Mityagin couples] \label{THM: relative CM-couples}
		{\it Let $X$, $Y$, $Z$ and $W$ be four Banach ideal spaces with the Fatou property.
			Suppose that either all spaces have non-trivial Boyd indices or are rearrangement invariant.
			Then the couples $(X,Y)$, $(\widetilde{X},\widetilde{Y})$ and $(X^{\downarrow},Y^{\downarrow})$ form a relative Calder{\' o}n--Mityagin couple
			with respect to $(Z,W)$, $(\widetilde{Z},\widetilde{W})$ and, respectively, $(Z^{\downarrow},W^{\downarrow})$ if, and only if, one of them form
			a respective relative Calder{\' o}n--Mityagin couple.
			In particular, the couples $(\widetilde{\Lambda_{\varphi}},L_{\infty})$ and $(L_1,M_{\varphi}^{\downarrow})$ are a relative Calder{\' o}n--Mityagin
			couples with respect to $(\widetilde{L_1},L_{\infty})$ and, respectively, $(L_1,L_{\infty}^{\downarrow})$.}
	\end{theorem}
	\begin{proof}
		The key observation is due to Dmitriev \cite{Dmi74}, who showed that the couple $(\Lambda_{\varphi},L_{\infty})$
		is a relative Calder{\' o}n--Mityagin couple with respect to $(L_1,L_{\infty})$. After this remark, the proof is not much
		different from the proof of Theorem~\ref{THM: CM-couples}. We leave the easy-to-complete details to the interested reader.
	\end{proof}
	
	Several important conclusions can be derived from what we have done so far.
	For example, since $L_p$'s are elastic (see Example~\ref{EXAMPLE: elastic spaces}), so the following result is an immediate conclusion from
	Theorem~\ref{THM: CM-couples} (see \cite[Theorem~3.8]{MS06}).
	
	\begin{corollary}[M.~Masty{\l}o and G.~Sinnamon, 2006] \label{COR: MS06}
		{\it The couple $(L_1,L_p^{\downarrow}) = (L_1,Ces_p)$, where $1 < p \leqslant \infty$, is a Calder{\' o}n--Mityagin couple.}
	\end{corollary}

	By finding a monotone variant of the classical Hardy, Littlewood and P{\' o}lya submajorization theorem and adopting some of Calder{\' o}n's ideas
	from \cite{Cal66} to this new situation, Le{\' s}nik completed the picture outlined in \cite{MS06} showing that also the couple $(\widetilde{L_1},L_\infty)$
	is a Calder{\' o}n--Mityagin couple (see \cite[Theorem~4.8]{Les15}; cf. \cite[Theorem~3.1]{MS17}).
	Now, this result can be seen as a straightforward consequence of Theorem~\ref{THM: CM-couples}.

	\begin{corollary}[K.~Le{\' s}nik, 2015] \label{COR: Lesnik}
		{\it The couple $(\widetilde{L_p},L_\infty)$, where $1 \leqslant p < \infty$, is a Calder{\' o}n--Mityagin couple.}
	\end{corollary}
	
	Although from the point of view of interpolation theory, it is Calder{\' o}n--Mityagin couples that lead to a clear structure of the interpolation scale,
	no less important are examples of concrete couples that fail to form a Calder{\' o}n--Mityagin couples. For example, it is known that the couples as simple
	as $(\ell_1 \oplus \ell_2, \ell_{\infty} \oplus \ell_{\infty})$, $(\ell_1(L_1),\ell_1(L_{\infty}))$ and $(L_1 \cap L_{\infty}, L_1 + L_{\infty})$
	do not form a Calder{\' o}n--Mityagin couples (see \cite{CNS03}, \cite{Kal92} and \cite{MO92} for more details and references).
	We would like to add a few more rather natural examples to this list.
	
	\begin{example} \label{THM: NOT CM-couples}
		{\it Let $1 < p < q < \infty$. Then the couples $(L_p,\widetilde{L_q})$, $(Ces_p,L_q) = (L_p^{\downarrow},L_q)$
		and $(Ces_p,\widetilde{L_q}) = (L_p^{\downarrow},\widetilde{L_q})$ are not Calder{\' o}n--Mityagin couples.}
	\end{example}
	\begin{proof}
		Let $1 < p < q < \infty$ and $0 < \theta < 1$. We will provide all the details of the proof only in the case of the couple $(Ces_p,L_q)$.
		The rest is completely analogous.
		
		We need a little preparation.
		Due to Shestakov's representation result (see \cite[Theorem~1]{She74}), the complex method of interpolation for any couple
		of Banach ideal spaces\footnote{Formally, the complex method of interpolation $[\cdot,\cdot]_{\theta}$ is defined only for couples
		of Banach ideal spaces defined over the field of complex numbers. Thus, if it happens that the pair $(X,Y)$ is not like this,
		we must consider the couple $(X(\mathbb{C}),Y(\mathbb{C}))$ instead of $(X,Y)$, where $X(\mathbb{C})$ and $Y(\mathbb{C})$ are
		the {\it complexificiations} of $X$ and, respectively, $Y$, that is to say, $X(\mathbb{C}) \coloneqq X \otimes_{\mathbb{R}} \mathbb{C}$
		and $Y \coloneqq Y \otimes_{\mathbb{R}} \mathbb{C}$ (see \cite[Example~2.6.6, p.~249]{BK91}). In such a situation, by $[X,Y]_{\theta}$
		we simply understand the subspace of all real-valued functions living inside $[X(\mathbb{C}),Y(\mathbb{C})]_{\theta}$.},
		say $(X,Y)$, can be described as follows
		\begin{equation*}
			[X,Y]_{\theta} \equiv \overline{X \cap Y}^{X^{1-\theta}Y^{\theta}},
		\end{equation*}
		that is, as the closure of $X \cap Y$ in the space $X^{1-\theta}Y^{\theta}$.
		In our situation, both spaces $Ces_p$ and $L_q$ are order continuous (see \cite[Theorem~3]{KT17}; cf. \cite[Lemma~1]{LM16}), so the Calder{\' o}n product
		$Ces_p^{1-\theta}L_q^{\theta}$ is order continuous as well (see \cite[Lemma~20, p.~428]{Loz69} and \cite[Theorem~15.10]{Ma89}; in fact, it is enough to know
		that one of the spaces $Ces_p$ and $L_q$ is order continuous, which is completely obvious for $L_q$). In consequence, since simple functions are dense in
		the space $Ces_p^{1-\theta}L_q^{\theta}$, so
		\begin{equation*}
			[Ces_p,L_q]_{\theta} \equiv \overline{Ces_p \cap L_q}^{Ces_p^{1-\theta}L_q^{\theta}} = Ces_p^{1-\theta}L_q^{\theta}.
		\end{equation*}
		Next, thanks to Corollary~\ref{COR: blocking technique CX}, we have
		\begin{equation*}
			Ces_p = \Bigl( \bigoplus_{j \in \mathbb{Z}} L_{1}(\Delta_j) \Bigr)_{\ell_p(w)},
		\end{equation*}
		where $w(j) = 2^{j(1/p-1)}$ for $j \in \mathbb{Z}$. Moreover, it is clear that
		\begin{equation*}
			L_q = \bigl( \bigoplus_{j \in \mathbb{Z}} L_{q}(\Delta_j) \bigr)_{\ell_q}.
		\end{equation*}
		Thus, using Theorem~\ref{THM: Calderon product amalgams}, we have
		\begin{align*}
			[Ces_p,L_q]_{\theta}
				& = Ces_p^{1-\theta}L_q^{\theta} \\
				& = \left[ \Bigl( \bigoplus_{j \in \mathbb{Z}} L_{1}(\Delta_j) \Bigr)_{\ell_p(w)} \right]^{1-\theta}
					\left[ \Bigl( \bigoplus_{j \in \mathbb{Z}} L_{q}(\Delta_j) \Bigr)_{\ell_q} \right]^{\theta} \\
				& = \Bigl( \bigoplus_{j \in \mathbb{Z}} L_1^{1-\theta}L_q^{\theta}(\Delta_j) \Bigr)_{\ell_p^{1-\theta}(w^{1-\theta})\ell_q^{\theta}} \\
				& = \Bigl( \bigoplus_{j \in \mathbb{Z}} L_{r_{\theta}}(\Delta_j) \Bigr)_{\ell_{s_{\theta}}(w^{1-\theta})},
		\end{align*}
		where $1/r_{\theta} = 1 - \theta + \theta/q$ and $1/s_{\theta} = (1-\theta)/p + \theta/q$.
		
		Let us get to the point.
		Suppose that the space $[Ces_p,L_q]_{\theta}$ can be described by the real $\mathscr{K}$-method of interpolation, that is,
		$[Ces_p,L_q]_{\theta} = (Ces_p,L_q)^{\mathscr{K}}_{\mathscr{X}}$ for some parameter $\mathscr{X}$. To say something more about $\mathscr{X}$,
		let us consider the projection
		\begin{equation*}
			P \colon f \rightsquigarrow \left\{ 2^{-j} \int_{\Delta_j} f(t)dt \right\}_{j \in \mathbb{Z}}.
		\end{equation*}
		Then, as a routine calculations show,
		the couple $(\ell_{p}(u),\ell_q(v))$ is a complemented subcouple of $(Ces_p,L_q)$ and the space $\ell_{s_{\theta}}(w^{1-\theta}2^{j/r_{\theta}})$
		is a complemented subspace of $[Ces_p,L_q]_{\theta}$. Here, $u(j) = 2^{j/p}$ and $v(j) = 2^{j/q}$ for $j \in \mathbb{Z}$.
		Furthermore, relying on \cite[Theorem~5.5.1, p.~119]{BL76}, which provides a precise description of the Lions--Petree real interpolation method
		for couples of weighted $L_p$-spaces, we infer that
		\begin{equation*}
			(P(Ces_p), P(L_q))_{\theta,s_{\theta}}^{\mathscr{K}}
				= (\ell_p(u),\ell_q(v))_{\theta,s_{\theta}}^{\mathscr{K}}
				= \ell_{s_{\theta}}(u^{1-\theta}v^{\theta}).
		\end{equation*}
		However, since
		\begin{equation*}
			u^{1-\theta}v^{\theta}
				= 2^{j\left( \frac{1-\theta}{p} + \frac{\theta}{q} \right)}
				= \left[ 2^{j\left( \frac{1}{p}-1 \right)} \right]^{1-\theta}2^{1-\theta}2^{j\frac{\theta}{q}}
				= w^{1-\theta}2^{j/r_{\theta}},
		\end{equation*}
		so
		\begin{equation*}
			P((Ces_p,L_q)^{\mathscr{K}}_{\mathscr{X}})
				= P([Ces_p,L_q]_{\theta})
				= \ell_{s_{\theta}}(w^{1-\theta}2^{j/r_{\theta}})
				= (P(Ces_p), P(L_q))_{\theta,s_{\theta}}^{\mathscr{K}}
		\end{equation*}		
		and is clear that the parameter $\mathscr{X}$ must coincide with the Lorentz space $L_{1/\theta,s_{\theta}}$.
		Now, let us consider the projection
		\begin{equation*}
			Q \colon f \rightsquigarrow f\chi_{\Delta_1}.
		\end{equation*}
		Then, we have
		\begin{equation*}
			(Q(Ces_p), Q(L_q))_{\theta,s_{\theta}}^{\mathscr{K}}
				= (L_1(\Delta_1),L_q(\Delta_1))_{\theta,s_{\theta}}^{\mathscr{K}}
				= L_{r_{\theta},s_{\theta}}(\Delta_1)
				\neq L_{s_{\theta}}(\Delta_1)
				= Q([Ces_p,L_q]_{\theta}).
		\end{equation*}
		All this means that the space $[Ces_p,L_q]_{\theta}$ cannot be described by the real $\mathscr{K}$-method of interpolation.
		In other words, $(Ces_p,L_q)$ is not a Calder{\' o}n--Mityagin couple.
		
		It remains to explain how to modify the above argument for the couple $(ces_p,\ell_q)$. For this purpose, it is not difficult
		to check that everything works {\it mutatis mutandis} until the very end when we consider the projection $Q$. Instead of doing this,
		let us consider the projection
		\begin{equation*}
			Q_n \colon f \rightsquigarrow f\chi_{\Delta_n}
		\end{equation*}
		for $n \in \mathbb{N}$. This time, we have
		\begin{equation*}
			\left( Q_n(ces_p), Q_n(\ell_q) \right)_{\theta,s_{\theta}}^{\mathscr{K}}
				= (\ell_1^{2^n},\ell_q^{2^n})_{\theta,s_{\theta}}^{\mathscr{K}}
				= \ell_{r_{\theta},s_{\theta}}^{2^n}
		\end{equation*}
		and
		\begin{equation*}
			Q_n([ces_p,\ell_q]_{\theta})
				= \ell_{r_{\theta}}^{2^n}.
		\end{equation*}
		Now, in order to see that $[ces_p,\ell_q]_{\theta} \neq (ces_p,\ell_q)_{\theta,s_{\theta}}^{\mathscr{K}}$, it is enough to take
		$g \in \ell_{r_{\theta},s_{\theta}} \setminus \ell_{r_{\theta}}$ (note that $r_{\theta} < s_{\theta}$) and consider the sequence
		$g_n = g\chi_{\Delta_n}$ for $n \in \mathbb{N}$. Then, clearly, we have
		\begin{equation} \label{EQ: normy interpolacja nierownowazne}
			\frac{\norm{g_n}_{[ces_p,\ell_q]_{\theta}}}{\norm{g_n}_{(ces_p,\ell_q)_{\theta,s_{\theta}}^{\mathscr{K}}}}
				= \frac{\norm{Q_n(g)}_{[ces_p,\ell_q]_{\theta}}}{\norm{Q_n(g)}_{(ces_p,\ell_q)_{\theta,s_{\theta}}^{\mathscr{K}}}}
				= \frac{\norm{g_n}_{\ell_{r_{\theta}}}}{\norm{g_n}_{\ell_{r_{\theta},s_{\theta}}}}.
		\end{equation}
		Since the right-hand side of \eqref{EQ: normy interpolacja nierownowazne} is unbounded (as $n$ goes to infinity), so the proof is thus completed.
	\end{proof}

	\begin{remark}[About Example~\ref{THM: NOT CM-couples}]
		It must be admitted that the couples mentioned in Example~\ref{THM: NOT CM-couples} do not form Calder{\' o}n--Mityagin couples
		for a fairly standard reason, namely, as is usually the case, the corresponding complex interpolation spaces cannot be described
		by the real $\mathscr{K}$-method of interpolation.
		However, as Masty{\l}o and Ovchinnikov showed, it is possible to give an example of a couple $(X,Y)$ which is not a Calder{\' o}n--Mityagin
		couple, but for which all the spaces $[X,Y]_{\theta}$ are Calder{\' o}n--Mityagin couples (see \cite[Theorem~3]{MO97} for the details of this construction).
		\demo
	\end{remark}
	
	\section{{\bf Products and factors: Grosse-Erdmann's style}} \label{SECTION: Products and factors: Grosse-Erdmann's style}
	
	\subsection{Pointwise multipliers} \label{SUBSECTION: pointwise multipliers}
	
	We will start with the following result, which can be understood as a far-reaching generalization of the duality between
	$\bigl( \bigoplus_{j \in J} X_j \bigr)_{E}$ and $\bigl( \bigoplus_{j \in J} X_j^{\times} \bigr)_{E^{\times}}$
	(see Proposition~\ref{Prop: Podstawowe wlasnosci sum prostych}(c)).
	
	\begin{theorem}[Pointwise multipliers between amalgams] \label{Thm: komutowanie M z cdot}
		{\it Let $E$ and $F$ be two Banach sequence spaces both defined on $J$. Further, let $\{X_j\}_{j \in J}$ and $\{Y_j\}_{j \in J}$
			be two families of Banach ideal spaces. Then}
			\begin{equation*}
				M\biggl( \Bigl( \bigoplus_{j \in J} X_j \Bigr)_{E}, \Bigl( \bigoplus_{j \in J} Y_j \Bigr)_{F} \biggr)
					\equiv \Bigl( \bigoplus_{j \in J} M \left( X_j,Y_j \right) \Bigr)_{M(E,F)}.
			\end{equation*}
	\end{theorem}
	\begin{proof}
		Let's begin by showing the following embedding
		\begin{equation} \label{INEQ: ->} \tag{$\lozenge$}
			\norm{ \text{id} \colon \Bigl( \bigoplus_{j \in J} M(X_j,Y_j) \Bigr)_{M(E,F)}
					\rightarrow M\biggl( \Bigl( \bigoplus_{j \in J} X_j \Bigr)_{E}, \Bigl( \bigoplus_{j \in J} Y_j \Bigr)_{F} \biggr) } \leqslant 1.
		\end{equation}
		For this, take $f = \{f_j\}_{j \in J}$ from $\bigl( \bigoplus_{j \in J} M(X_j,Y_j) \bigr)_{M(E,F)}$ and $g = \{g_j\}_{j \in J}$
		from $\bigl( \bigoplus_{j \in J} X_j \bigr)_{E}$ with $\norm{g}_{\left( \bigoplus_{j \in J} X_j \right)_{E}}\leqslant 1$. Then,
		applying H{\" o}lder--Rogers's inequality twice, we get that
		\begin{align*}
			\norm{fg}_{\left( \bigoplus_{j \in J} Y_j \right)_{F}} & = \norm{\sum\limits_{j\in J} \norm{f_j g_j}_{Y_j} e_j}_{F} \\
				& \leqslant \norm{\sum\limits_{j\in J} \norm{f_j}_{M(X_j,Y_j)} \norm{g_j}_{X_j} e_j}_{F} \\
			& \leqslant \norm{\sum\limits_{j\in J}\norm{f_j}_{M(X_j,Y_j)} e_j}_{M(E,F)} \norm{\sum\limits_{j\in J} \norm{g_j}_{X_j} e_j}_{E} \\
			& \leqslant \norm{\sum\limits_{j\in J}\norm{f_j}_{M(X_j,Y_j)} e_j}_{M(E,F)}
				= \norm{f}_{\left( \bigoplus_{j \in J} M(X_j,Y_j) \right)_{M(E,F)}}.
		\end{align*}
		Thus $f$ belongs to the space $M\left( \left( \bigoplus_{j \in J} X_j \right)_{E}, \left( \bigoplus_{j \in J} Y_j \right)_{F} \right)$
		and, moreover,
		\begin{align*}
			\norm{f}_{M \left( \left( \bigoplus_{j \in J} X_j \right)_{E}, \left( \bigoplus_{j \in J} Y_j \right)_{F} \right)}
				& = \sup\limits_{\norm{g}_{\left( \bigoplus_{j \in J} X_j \right)_{E}} \leqslant 1} \norm{fg}_{\left( \bigoplus_{j \in J} Y_j \right)_{F}} \\
				& \leqslant \norm{f}_{\left( \bigoplus_{j \in J} M(X_j,Y_j) \right)_{M(E,F)}}.
		\end{align*}
		This gives (\ref{INEQ: ->}).
		
		It remains to show that
		\begin{equation} \label{INEQ: <-} \tag{$\blacklozenge$}
			\norm{ \text{id} \colon M\biggl( \Bigl( \bigoplus_{j \in J} X_j \Bigr)_{E}, \Bigl( \bigoplus_{j \in J} Y_j \Bigr)_{F} \biggr)
					\rightarrow \Bigl( \bigoplus_{j \in J} M(X_j,Y_j) \Bigr)_{M(E,F)} } \leqslant 1.
		\end{equation}
		Take $f = \{f_j\}_{j \in J}$ from $\left( \bigoplus_{j \in J} M(X_j,Y_j) \right)_{M(E,F)}$ and observe that
		\begin{align*}
			\norm{f}_{\left( \bigoplus_{j \in J} M(X_j,Y_j) \right)_{M(E,F)}}
				& = \sup\limits_{\substack{a = \{a_j\}_{j \in J} \\ \norm{ a }_{E} \leqslant 1}} \norm{\sum\limits_{j \in J} \norm{f_j}_{M(X_j,Y_j)} \abs{a_j} e_j }_{F} \\
			& = \sup\limits_{\substack{a = \{a_j\}_{j \in J} \\ \norm{ a }_{E} \leqslant 1 \\ \norm{g_j}_{X_j} \leqslant 1 \text{ for } j \in J}}
				\norm{\sum\limits_{j \in J} \norm{f_j g_j}_{Y_j} \abs{a_j} e_j }_{F} \\
			& = \sup\limits_{\substack{a = \{a_j\}_{j \in J} \\ \norm{ a }_{E} \leqslant 1 \\ \norm{g_j}_{X_j}\leqslant 1 \text{ for } j \in J}}
				\norm{\sum\limits_{j \in J}\norm{f_j g_j a_j}_{Y_j} e_j }_{F}.
		\end{align*}
		Now, since the sequence $a = \{a_j\}_{j \in J}$ belongs to $E$ and $\norm{g_j}_{X_j} \leqslant 1$ for all $j \in J$,
		so we claim the function $h$ defined as
		\begin{equation*}
			h \coloneqq \sum_{j \in J} g_j a_j e_j
		\end{equation*}
		belongs to the space $(\bigoplus_{j \in J} X_j)_{E}$. In fact, we have
		\begin{equation*}
			\norm{h}_{ \left( \bigoplus_{j \in J} X_j \right)_{E}} = \norm{\sum_{j \in J} \norm{g_j a_j}_{X_j} e_j}_E
				= \norm{\sum_{j \in J} \norm{g_j}_{X_j} \abs{a_j} e_j}_E \leqslant \norm{\sum_{j \in J} a_j e_j}_E = \norm{a}_E,
		\end{equation*}
		where the above inequality is due to the ideal property of the space $E$. This proves our claim. Consequently,
		\begin{align*}
			\sup\limits_{\substack{a = \{a_j\}_{j \in J} \\ \norm{ a }_{E} \leqslant 1 \\ \norm{g_j}_{X_j}\leqslant 1 \text{ for } j \in J}}
				\norm{\sum\limits_{j \in J}\norm{f_j g_j a_j}_{Y_j} e_j }_{F}
			& \leqslant \sup\limits_{\substack{\norm{h}_{\left( \bigoplus_{j \in J} X_j \right)_{E}} \leqslant 1}} \norm{\sum\limits_{j\in J} \norm{f_j h_j}_{Y_j} e_j}_{F} \\
			& = \norm{f}_{M \left( \left( \bigoplus_{j \in J} X_j \right)_{E}, \left( \bigoplus_{j \in J} Y_j \right)_{F} \right)},
		\end{align*}
		which means that the proof of (\ref{INEQ: <-}) has just been completed. That's all.
	\end{proof}

	In particular, from Theorem~\ref{Thm: komutowanie M z cdot} we easily obtain the following result of Kellogg (see \cite[Theorem~1]{Kel71}).

	\begin{corollary}[C. N. Kellogg, 1971]
		{\it Let $1 \leqslant p,q,r,s \leqslant \infty$. Further, let $\{d_n\}_{n=1}^{\infty}$ be a sequence of positive integers.
		Suppose that $r > p$ and $s > q$. Then, we have}
		\begin{equation*}
			M\biggl( \Bigl( \bigoplus_{n=1}^{\infty} \ell_{p}^{d_n} \Bigr)_{\ell_{q}}, \Bigl( \bigoplus_{n=1}^{\infty} \ell_{r}^{d_n} \Bigr)_{\ell_{s}} \biggr)
				\equiv \Bigl( \bigoplus_{n=1}^{\infty} \ell_{1/r - 1/p}^{d_n} \Bigr)_{\ell_{1/s - 1/q}}.
		\end{equation*}
	\end{corollary}

	We are ready to execute the main goal of this section.
	
	\begin{theorem}[Pointwise multipliers] \label{Cor: multipliers between Cesaro and Tandori}
		{\it Let $X$ and $Y$ be two r.i. spaces with the Fatou property both defined on the same measure space.
			Further, suppose that}
		\begin{itemize}
			\item[$\bullet$] {\it either  the space $\mathbf{E}(Y)$ factorizes through $\mathbf{E}(X)$, that is, $\mathbf{E}(Y) = \mathbf{E}(X) \odot M(\mathbf{E}(X), \mathbf{E}(Y))$,}
			\item[$\bullet$] {\it or both spaces $X$ and $Y$ are strongly separated, that is, $\alpha_X > \beta_Y$.}
		\end{itemize}
			{\it Then}
		\begin{equation} \label{EQ: MTandori = TandoriM}
			M(\widetilde{X},\widetilde{Y}) = \widetilde{M(X,Y)}.
		\end{equation}
		{\it Moreover, if Hardy's operator $\mathscr{H}$ is bounded on both spaces $X$ and on $Y$, then also}
		\begin{equation} \label{M(CX,CY) = CM(X,Y)}
			M(\mathscr{C}X,\mathscr{C}Y) = \widetilde{M(X,Y)}.
		\end{equation}
	\end{theorem}
	\begin{proof}
		We will only show the equality \eqref{EQ: MTandori = TandoriM} in all details.
		The proof of equality \eqref{M(CX,CY) = CM(X,Y)} proceeds in completely analogous manner.
		
		Using Theorems~\ref{Thm: Tandori sequence representation} and \ref{Thm: Tandori function representation}
		along with Theorem~\ref{Thm: komutowanie M z cdot}, we have
		\begin{equation*}
			M(\widetilde{X},\widetilde{Y})
				= M\left( \Bigl( \bigoplus_{j \in \mathbb{J}} L_\infty(\Delta_j) \Bigr)_{{\bf E}(X)}, \Bigl( \bigoplus_{j \in \mathbb{J}} L_\infty(\Delta_j) \Bigr)_{{\bf E}(Y)} \right)
				= \Bigl( \bigoplus_{j \in \mathbb{J}} L_\infty(\Delta_j) \Bigr)_{M({\bf E}(X),{\bf E}(Y))}.			   
		\end{equation*}
		Assume for a moment that
		\begin{equation} \label{EQ: komutowanie E z M dowod} \tag{$\diamondsuit$}
			M({\bf E}(X),{\bf E}(Y)) = \mathbf{E}(M(X,Y)).
		\end{equation}
		Then, applying Theorems~\ref{Thm: Tandori sequence representation} and \ref{Thm: Tandori function representation} once again to what we did above, we have
		\begin{equation*}
			M(\widetilde{X},\widetilde{Y})
				= \Bigl( \bigoplus_{j \in \mathbb{J}} L_\infty(\Delta_j) \Bigr)_{M({\bf E}(X),{\bf E}(Y))}
			 	\overset{(\diamondsuit)}{=} \Bigl( \bigoplus_{j \in \mathbb{J}} L_\infty(\Delta_j) \Bigr)_{\mathbf{E}(M(X,Y))}
			 	= \widetilde{M(X,Y)}.
		\end{equation*}
		Thus, the equality \eqref{EQ: MTandori = TandoriM} follows and the only thing left to do is justify \eqref{EQ: komutowanie E z M dowod}.
		
		To do this, let us first assume that $\mathbf{E}(Y) = \mathbf{E}(X) \odot M(\mathbf{E}(X), \mathbf{E}(Y))$. Then, using the so-called
		{\it cancelation property} (see Theorem~\ref{THM: cancellation property for multipliers}), we have
		\begin{align*}
			M({\bf E}(X),{\bf E}(Y))
				& = M({\bf E}(X) \odot \ell_\infty(\mathbb{J}), {\bf E}(X) \odot {\bf E}(M(X,Y)) \\
				& = M(\ell_\infty(\mathbb{J}), {\bf E}(M(X,Y))) \\
				& = {\bf E}(M(X,Y)).
		\end{align*}
		This means that \eqref{EQ: komutowanie E z M dowod} follows.
		
		Next, let us assume that $\alpha_X > \beta_Y$. It is known that the product space $X \odot Y^{\times}$ can be identified with the
		$\frac{1}{2}$-concavification of the Calder{\' o}n product $X^{1/2}Y^{1/2}$ (see \eqref{EQ: product as 1/2-concavification of Calderon product};
		cf. \cite[Theorem~2.1]{Sch10} and \cite[Theorem~1(iv)]{KLM14}).
		For notational convenience, let us put
		\begin{equation*}
			Z = \left( X^{1/2}{Y^{\times}}^{1/2} \right)^{(1/2)} \quad \text{ and } \quad W = X^{1/2}{Y^{\times}}^{1/2}.
		\end{equation*}
		It is straightforward to see that the operation of $\frac{1}{2}$-concavification applied to a Banach ideal space divides
		its Boyd indices in half. In consequence, we have
		\begin{equation*}
			\frac{1}{\alpha_{X \odot Y^{\times}}}
				= \frac{1}{\alpha_{Z}}
				= \frac{2}{\alpha_{W}}
				\leqslant 2\left( \frac{\frac{1}{2}}{\alpha_X} + \frac{\frac{1}{2}}{\alpha_{Y^{\times}}} \right)
				= 1 + \frac{1}{\alpha_{X}} - \frac{1}{\beta_{Y}},
		\end{equation*}
		where the last equality follows from \eqref{EQ: Boydy X i dualu X}.
		However, since we assumed that $\alpha_X > \beta_Y$, so
		\begin{equation*}
			\alpha_{X \odot Y^{\times}} > 1.
		\end{equation*}
		For this reason, we are free to use Proposition~\ref{Prop: E komutuje z Kothe dualem}. This, in tandem with the following \enquote{tensor--hom} duality
		\begin{equation*}
			(X \odot Y^{\times})^{\times} \equiv M(X,Y^{\times\times}) \equiv M(X,Y),
		\end{equation*}
		where the last equality is due to the Fatou property of the space $Y$ (see, for example, \cite[Corollary~3]{KLM14}), gives us
		\begin{equation} \label{EQ: rachuneczek}
			\left[ {\bf E}(X \odot Y^{\times}) \right]^{\times} \equiv {\bf E}(M(X,Y))(w),
		\end{equation}
		where $w(j) = 2^{-j}$ for $j \in J$. On the other hand,
		\begin{align*}
			\left[ {\bf E}(X \odot Y^{\times}) \right]^{\times}
				& = \left[ {\bf E}(X) \odot {\bf E}(Y^{\times}) \right]^{\times} \quad (\text{by Proposition~\ref{PROP: E komutuje z produktem}}) \\
				& = \left[ {\bf E}(X) \odot {\bf E}(Y)^{\times}(1/w) \right]^{\times} \quad (\text{using Proposition~\ref{Prop: E komutuje z Kothe dualem}}) \\
				& = \left[ {\bf E}(X) \odot {\bf E}(Y)^{\times} \right]^{\times} (w) \\
				& = {\bf E}(M(X,Y))(w).
		\end{align*}
		Combining the above with \eqref{EQ: rachuneczek} we get \eqref{EQ: komutowanie E z M dowod}. The proof has been completed.
	\end{proof}

	\begin{remark}[About Theorem~\ref{Cor: multipliers between Cesaro and Tandori}] \label{REMARK: M komutuje z E}
		It is not just that \eqref{EQ: komutowanie E z M dowod} implies \eqref{EQ: MTandori = TandoriM}, but actually both conditions are equivalent, namely,
		\begin{equation*}
			M(\widetilde{X},\widetilde{Y}) = \widetilde{M(X,Y)} \quad \text{ if, and only if, } \quad M({\bf E}(X),{\bf E}(Y)) = \mathbf{E}(M(X,Y)).
		\end{equation*}
		In fact, suppose that $M(\widetilde{X},\widetilde{Y}) = \widetilde{M(X,Y)}$. Then, due to
		Theorems~\ref{Thm: Tandori sequence representation}, \ref{Thm: Tandori function representation} and \ref{Thm: komutowanie M z cdot}, we have
		\begin{equation*}
			\Bigl( \bigoplus_{j \in \mathbb{J}} L_\infty(\Delta_j) \Bigr)_{M({\bf E}(X),{\bf E}(Y))} = \Bigl( \bigoplus_{j \in \mathbb{J}} L_\infty(\Delta_j) \Bigr)_{{\bf E}(M(X,Y))}.
		\end{equation*}
		Then, clearly, $M({\bf E}(X),{\bf E}(Y)) = \mathbf{E}(M(X,Y))$. This is what we wanted.
		\demo
	\end{remark}

	\begin{remark}[About the assumptions in Theorem~\ref{Cor: multipliers between Cesaro and Tandori}]
		A weaker version of Theorem~\ref{Cor: multipliers between Cesaro and Tandori}, that is, with the assumption
		that the space $Y$ factorizes through $X$, was proved by Kolwicz, Le{\' s}nik and Maligranda \cite{KLM19} (see \cite[Theorem~6]{KLM19}).
		For an explicit example showing that the space ${\bf E}(Y)$ can factorize through ${\bf E}(X)$, but the space $Y$ does not
		factorize through $X$ we refer to Section~\ref{REMARK: faktoryzacja z E lepsza niz bez E}.
		\demo
	\end{remark}

	Since the space $\ell_p$ factorizes through $\ell_q$ provided $p \leqslant q$, so Theorem~~\ref{Cor: multipliers between Cesaro and Tandori}
	immediately leads us to the following conclusion (see \cite[pp.~69 and~76]{Be96} and \cite[Chapter~7]{GE98}).

	\begin{corollary}[G. Bennett, 1996/ K. G. Grosse-Erdmann, 1998]
		{\it Let $1 < q \leqslant p \leqslant \infty$ with $1/r = 1/q - 1/p$. Then}
		\begin{equation*}
			M(\widetilde{\ell_p},\widetilde{\ell_q}) = \widetilde{\ell_r} \quad \textit{ and } \quad M(ces_p,ces_q) = \widetilde{\ell_r}.
		\end{equation*}
	\end{corollary}

	\subsection{Some examples in the spirit of Bennett} \label{SECTION: Examples a'la Bennett}
	Let $1 \leqslant p, q \leqslant \infty$. Some elementary considerations based on H{\" o}lder--Rogers's inequality
	(see, for example, \cite[Theorem~2]{MP89}) inevitably lead to the conclusion that
	\begin{equation} \label{EQ: lp -> lq <=> M = linfty}
		\ell_p \hookrightarrow \ell_q \quad \text{ if, and only if, } \quad M(\ell_p,\ell_q) = \ell_{\infty}.
	\end{equation}
	Even more generally, but without surprises, for any two rearrangement invariant sequence spaces $E$ and $F$, we have
	\begin{equation} \label{EQ: E -> F <=> M = linfty}
		E \hookrightarrow F \quad \text{ if, and only if, } \quad M(E,F) = \ell_{\infty}.
	\end{equation}
	Asking the question we have to ask here: {\it What is the analogue of \eqref{EQ: lp -> lq <=> M = linfty} for the family
	of $ces_p$-spaces?} Bennett was able to show that
	\begin{equation} \label{EQ: cesp -> cesq <=> M = linfty(wpq)}
		ces_p \hookrightarrow ces_q \quad \text{ if, and only if, } \quad M(ces_p,ces_q) = \ell_{\infty}(w_{p \to q}),
	\end{equation}
	where $w_{p \to q}(n) \coloneqq n^{1/q-1/p}$ for $n \in \mathbb{N}$ and $1 < p,q \leqslant \infty$ (see \cite[p.~69]{Be96}; cf. \cite[Chapter II.7]{GE98}).
	This may come as a slight\footnote{After all, the space $ces_p$ with $1 < p \leqslant \infty$ is neither isomorphic to $\ell_q$
	for any $1 \leqslant q \leqslant \infty$ nor is rearrangement invariant (see \cite[Proposition~15.13, p.~82]{Be96} and, respectively, \cite[Theorem~3.1 (a)]{AM14}).}
	surprise, because the space $\ell_{\infty}(w_{p \to q})$ is definitely larger than $\ell_{\infty}$.
	Anyway, since \eqref{EQ: lp -> lq <=> M = linfty} can be easily generalized to \eqref{EQ: E -> F <=> M = linfty},
	so the question about the general variant of \eqref{EQ: cesp -> cesq <=> M = linfty(wpq)}, seems to be intriguing
	(also in the context of Theorem~\ref{Cor: multipliers between Cesaro and Tandori}).
	
	Below, we will answer\footnote{By the way, this partially answers the question posed in \cite{KLM19} (see paragraph preceding Corollary~4, p.~1164).}
	this question for the most important classes of rearrangement invariant sequence spaces, such as Orlicz, Lorentz and Marcinkiewicz sequence spaces.
	The general case, unfortunately, remains open (see Section~\ref{QUESTION: M(CX,CY) if M(X,Y) = Linfty}).
	
	Let us start with a direct generalization of Bennett's result \eqref{EQ: cesp -> cesq <=> M = linfty(wpq)} to the case of Orlicz sequence spaces.
	
	\begin{example} \label{przyklad Orlicze infty}
		{\it Let $M$ and $N$ be two Young functions with $\alpha_M, \alpha_N > 1$. Suppose that $\ell_M \hookrightarrow \ell_N$. Then}
			\begin{equation*}
				M(ces_M, ces_N) = \ell_\infty\left( w_{M \to N} \right),
			\end{equation*}
		{\it where $w_{M \to N}(n) \coloneqq \norm{\sum_{k=1}^n e_k}_{\ell_N} / \norm{\sum_{k=1}^n e_k}_{\ell_M}$ for $n \in \mathbb{N}$.}
	\end{example}
	\begin{proof}
		Our ultimate goal is to show that
		\begin{equation} \label{EOrliczeMnoznik} \tag{$\clubsuit$}
			M(\E(\ell_M), \E(\ell_N)) = \ell_\infty\left( \frac{M^{-1}(2^{-n})}{N^{-1}(2^{-n})} \right).
		\end{equation}
		Suppose for a moment that (\ref{EOrliczeMnoznik}) has been proved.
		Then, due to Corollary~\ref{COR: blocking technique CX} and \ref{Thm: komutowanie M z cdot}, we have
		\begin{align*}
			M(ces_M, ces_N)
				= \biggl( \bigoplus_{j=0}^{\infty} \ell^{2^j}_\infty \biggr)_{M(\E(X),\E(Y))}
				\overset{(\clubsuit)}{=} \biggl( \bigoplus_{j=0}^{\infty} \ell^{2^j}_\infty \biggr)_{\ell_\infty\left( \frac{M^{-1}(2^{-n})}{N^{-1}(2^{-n})} \right)}
				= \ell_\infty\left( W_{M \to N} \right),
		\end{align*}
		where
		\begin{equation*}
			W_{M \to N}(n) \coloneqq \sum_{k=0}^{\infty} \frac{M^{-1}(2^{-k})}{N^{-1}(2^{-k})} \chi_{\Delta_k}(n) \quad \text{ for } \quad n \in \mathbb{N}.
		\end{equation*}
		Moreover, it is quite easy to see that $W_{M \to N} \approx \left\{ \frac{M^{-1}(1/n)}{N^{-1}(1/n)} \right\}_{n=1}^{\infty}$, that is,
		\begin{equation*}
			\ell_{\infty}(W_{M \to N}) = \ell_{\infty}(w_{M \to N}).
		\end{equation*}
		Indeed, take any $n \in \mathbb{N}$ and choose $k \in \mathbb{N}_0$ with $2^k \leqslant n < 2^{k+1}$. We have
		\begin{align*}
			\frac{M^{-1}(1/n)}{N^{-1}(1/n)}
				\leqslant \frac{M^{-1}(2^{-k})}{N^{-1}(2^{-k-1})}
				\leqslant \frac{M^{-1}(2^{-k})}{2^{-1}N^{-1}(2^{-k})} = 2W_{M \to N}(n).
		\end{align*}
		On the other hand, keeping in mind that the function $N^{-1}$ is concave,
		\begin{align*}
			\frac{M^{-1}(1/n)}{N^{-1}(1/n)}
				\geqslant\frac{M^{-1}(2^{-k-1})}{N^{-1}(2^{-k})}
				\geqslant\frac{2^{-1}M^{-1}(2^{-k})}{N^{-1}(2^{-k})} = \frac{1}{2}W_{M \to N}(n).
		\end{align*}
		Consequently, since $\norm{\sum_{k=1}^n e_k}_{\ell_M} = 1/M^{-1}(1/n)$ and $\norm{\sum_{k=1}^n e_k}_{\ell_N} = 1/N^{-1}(1/n)$
		for $n \in \mathbb{N}$ (see \cite[Corollary~4, p.~58]{Ma89}), so
		\begin{equation*}
			M(ces_M, ces_N) = \ell_\infty\left( W_{M \to N} \right)
				= \ell_{\infty}\left( \frac{M^{-1}(1/n)}{N^{-1}(1/n)} \right)
				= \ell_{\infty}(w_{M \to N}).
		\end{equation*}
		All this means that to complete the proof it is enough to show the equality \eqref{EOrliczeMnoznik}.
		
		Let us focus on that. That said, we need three simple observations. First of all, without loss of generality, we can assume that $a_N = 0$,
		or, which is one thing, that $\ell_N \neq \ell_\infty$. Suppose this is not the case. Now, if $a_N \neq 0$ and $a_M = 0$, then
		\begin{equation*}
			M(\ell_M, \ell_N) = M(\ell_\infty, \ell_N ) = \ell_N \neq \ell_\infty,
		\end{equation*}
		which contradicts our assumption about $M(\ell_M, \ell_N)$. On the other hand, if $a_N \neq 0$, but $a_M \neq 0$, then
		\begin{equation*}
			M(\E(\ell_M), \E(\ell_N)) = M(\ell_\infty, \ell_\infty) = \ell_\infty
		\end{equation*}
		and there is nothing to prove. Secondly, since $M(\ell_M, \ell_N) = \ell_{N\ominus M}$, so our assumption about $M(\ell_M, \ell_N)$ implies that
		$a_{N\ominus M} > 0$, that is,
		\begin{equation} \label{OrliczeDopelniajacaZero}
			N(us) - M(s) \leqslant 0 \quad \text{ for all } \quad u < a_{N\ominus M} \quad \text{ and } \quad 0 \leqslant s \leqslant 1.
		\end{equation}
		Thirdly, and lastly, note that the space $\E(\ell_O)$, where $O$ is a Young function, can be described as an appropriate Musielak--Orlicz space.
		In fact, take $f = \sum_{n=0}^{\infty} x_n \chi_{\Delta_n}$ from $\E(\ell_O)$. We have
		\begin{align*}
			m_O(f)
				& = \sum_{k=1}^{\infty} O \left( \sum_{n=0}^{\infty} x_n \chi_{\Delta_n}(k) \right) \\
				& = \sum_{k=1}^{\infty} \sum_{n=0}^{\infty} O(x_n)\chi_{\Delta_n}(k) \\
				& = \sum_{n=0}^{\infty} 2^n O(x_n) \\
				& = \sum_{n=0}^{\infty} \Phi_O(n,x_n) \\
				& = m_{\Phi_O} \left( \sum_{n=0}^{\infty} x_n e_n \right),
		\end{align*}
		where the Musielak--Orlicz function $\Phi_O$ is defined as
		\begin{equation*}
			\Phi_O(n,u) \coloneqq 2^n O(u) \quad \text{ for } \quad n \in \mathbb{Z}_+ \quad \text{ and } \quad u \geqslant 0.
		\end{equation*}
		Therefore, the equality between modulars $m_O(f) = m_{\Phi_O}(x)$ for all possible choices of $x = \{x_n\}_{n=0}^{\infty}$ shows that $\E(\ell_O) \equiv \ell_{\Phi_O}$.
		
		We are now ready to show \eqref{EOrliczeMnoznik}. Using the above representation of $\E(\ell_O)$ and \cite[Corollary~13]{LT21} (cf. Theorem~\ref{THM: Faktoryzacja MO}),
		we have
		\begin{equation*}
			M(\E(\ell_M), \E(\ell_{N}))=M(\ell_{\Phi_M}, \ell_{\Phi_N})=\ell_{\Phi_N\ominus\Phi_M},
		\end{equation*}
		where  $\Phi_N\ominus\Phi_M$ is a Musielak--Orlicz function given by
		\begin{equation*}
			\Phi_N\ominus\Phi_M(k,u) \coloneqq \sup\limits_{0\leqslant s\leqslant \Phi_M^{-1}(k,1)} \left\{ \Phi_N(k,us)-\Phi_M(k,s) \right\}
				\quad \text{ for } \quad k \in \mathbb{N}_0 \quad \text{ and } \quad u \geqslant 0.
		\end{equation*}
		Note that for a given Musielak--Orlicz function $\Psi$, if $a_{\Psi}(k) \coloneqq \sup \left\{u \geqslant 0 \colon \Psi(u) = 0 \right\} > 0$
		for every $k \in \mathbb{N}_0$, then $\ell_\Psi = \ell_\infty(1/a_{\Psi}(k))$. Thus, in order to prove \eqref{EOrliczeMnoznik}, we need to show that
		$$
			a_{\Phi_N\ominus\Phi_M}(k) = \frac{N^{-1}(2^{-k})}{M^{-1}(2^{-k})} \quad \text{ for } \quad k \in \mathbb{N}_0.
		$$
		Noting that $\Phi_M^{-1}(k,1) = M^{-1}(2^{-k})$, we have
		\begin{align*}
			\Phi_N\ominus\Phi_M(k,u)
				& = \sup\limits_{0\leqslant s\leqslant \Phi_M^{-1}(k,1)} \left\{ \Phi_N(k,us)-\Phi_M(k,s) \right\} \\
				& = \sup\limits_{0\leqslant s\leqslant M^{-1}(2^{-k})} \left\{ 2^kN(us)-2^kM(s) \right\} \\
				& = 2^k \sup\limits_{0 \leqslant s \leqslant M^{-1}(2^{-k})} \left\{ N(us)-M(s) \right\}.
		\end{align*}
		Since $a_N=0$, so we can find $K \in \mathbb{N}$ such that for $k \geqslant K$ we have $N^{-1}(2^{-k})\leqslant a_{N\ominus M}$.
		Then, using the inequality \eqref{OrliczeDopelniajacaZero}, for  $k \geqslant K$ and $0\leqslant s\leqslant M^{-1}(2^{-k})$, we get
		$$
		N\left(\frac{N^{-1}(2^{-k})}{M^{-1}(2^{-k})}s\right) - M(s) \leqslant 0.
		$$
		Denoting $u_k \coloneqq N^{-1}(2^{-k})/M^{-1}(2^{-k})$, we see that $u_k \leqslant a_{\Phi_N\ominus\Phi_M}(k)$ for $k \geqslant K$.
		To show the reverse inequality, take $u > u_k$. We can find $\varepsilon > 0$ for which $u=N^{-1}(2^{-k}+\varepsilon)/M^{-1}(2^{-k})$.
		In consequence,
		\begin{align*}
			\Phi_N\ominus\Phi_M(k,u)
				& = 2^k\sup\limits_{0\leqslant s\leqslant M^{-1}(2^{-k})} \left\lbrace N\left(\frac{N^{-1}(2^{-k}+\varepsilon)}{M^{-1}(2^{-k})}s\right)-M(s) \right\rbrace \\ 
				& \geqslant 2^k \left[N\left(\frac{N^{-1}(2^{-k}+\varepsilon)}{M^{-1}(2^{-k})}M^{-1}(2^{-k})\right)-M(M^{-1}(2^{-k}) \right]\\
				& = 2^k \left( 2^{-k}+\varepsilon-2^{-k} \right)
					= 2^{k}\varepsilon > 0.
		\end{align*}
		This, however, means that $u_k = a_{\Phi_N\ominus\Phi_M}(k)$. For $k < K$, using \eqref{OrliczeDopelniajacaZero}, we can only infer that $a_{\Phi_N\ominus\Phi_M}(k)>0$.
		Nonetheless, up to an equivalent norm, we have  
		$$
			M\left( \E(\ell_M), \E(\ell_{N}) \right)
				= \ell_{\Phi_N\ominus\Phi_M}
				= \ell_\infty\left( \frac{M^{-1}(2^{-k})}{N^{-1}(2^{-k})} \right).
		$$
		In light of what we said before, this completes the proof.
	\end{proof}

	The complementary\footnote{Note that the inclusion $\ell_M \hookrightarrow \ell_N$ means exactly that $M(\ell_M,\ell_N) = \ell_{\infty}$.} case is as follows.

	\begin{example} \label{przyklad Orlicze nieinfty}
		{\it Let $M$ and $N$ be two Young functions with $\alpha_M, \alpha_N > 1$. Suppose that $M(\ell_M, \ell_N) \neq \ell_{\infty}$. Then}
		\begin{equation*}
			M(ces_M, ces_N) = \widetilde{\ell_{N \ominus M}}.
		\end{equation*}
	\end{example}
	\begin{proof}
		To prove the above equality it is enough to show that
		\begin{equation} \label{EQ: EMOrlicze = MEOrlicze} \tag{$\spadesuit$}
			\E \left( M(\ell_M,\ell_N) \right) = M\left( \mathbf{E}(\ell_M), \mathbf{E}(\ell_N) \right)
		\end{equation}
		and invoke the Theorem~\ref{Cor: multipliers between Cesaro and Tandori}. Using the notation form Example~\ref{przyklad Orlicze infty}, we have
		$$
			\E \left( M(\ell_M,\ell_N)=\E (\ell_{N\ominus M}) \right) = \ell_{\Phi_{N\ominus M}}
		$$
		and
		$$
			M(\mathbf{E}(\ell_M), \mathbf{E}(\ell_N))= M(\ell_{\Phi_{M}},\ell_{\Phi_{M}})=\ell_{\Phi_{N}\ominus {\Phi_{M}}}.
		$$
		Therefore, in order to prove \eqref{EQ: EMOrlicze = MEOrlicze}, we need to show
		\begin{equation} \label{EQ: dozakonczeniadowodu}
			\ell_{\Phi_{N\ominus M}} = \ell_{\Phi_{N}\ominus {\Phi_{M}}}.
		\end{equation}
		We claim that
		\begin{equation} \label{Orlicze przyklad 2 rownosc}
			\Phi_{N}\ominus {\Phi_{M}}(k,u) = \Phi_{N\ominus M}(k,u) \quad \text{ for } \quad k \in \mathbb{Z}_+ \quad \text{ and } \quad 0 \leqslant u \leqslant N^{-1}(2^{-k}).
		\end{equation}
		Of course, the above equality means that $\ell_{\Phi_{N\ominus M}}=\ell_{\Phi_{N}\ominus {\Phi_{M}}}$. We have
		$$
			\Phi_N\ominus\Phi_M(k,u) = 2^k \sup\limits_{0\leqslant s\leqslant M^{-1}(2^{-k})} \left\{ N(us)-M(s) \right\}
		$$
		and
		$$
			\Phi_{N\ominus M}(k,u) = 2^k \sup\limits_{0 \leqslant s \leqslant 1} \left\{ N(us)-M(s) \right\}.
		$$
		Thus, to show \eqref{Orlicze przyklad 2 rownosc}, we need to prove that
		\begin{equation} \label{INEQ: N - M}
			N(us) - M(s) < 0 \quad \text{ for } \quad s > M^{-1}(2^{-k}) \quad \text{ and } \quad u \geqslant 0 \quad \text{ small enough}.
		\end{equation}
		To do this, fix $k \in \mathbb{N}_0$, take $u \leqslant N^{-1}(2^{-k})$ and $M^{-1}(2^{-k}) \leqslant s \leqslant 1$. Then
		\begin{align*}
			N(us) - M(s) \leqslant N(u) - M(M^{-1}(2^{-k})) = N(u) - 2^{-k} \leqslant 2^{-k} - 2^{-k} = 0
		\end{align*}
		and \eqref{INEQ: N - M} follows. In consequence, also \eqref{EQ: dozakonczeniadowodu} holds and the proof is finished.
	\end{proof}

	\begin{remark}[About Example~\ref{przyklad Orlicze nieinfty}]
		It should be noted that Example~\ref{przyklad Orlicze nieinfty} above strengthens Theorem~\ref{Cor: multipliers between Cesaro and Tandori}
		in the case of Orlicz sequence spaces. This is clear, because the assumption that $M(\ell_M, \ell_N) \neq \ell_{\infty}$ does not in general
		imply the factorization $\ell_N = \ell_M \odot M(\ell_M, \ell_N)$. Nevertheless, the construction of an appropriate example is not the most
		obvious one (we recommend taking a look at \cite[Proposition~A.1]{KT23}).
		\demo
	\end{remark}

	By specifying Example~\ref{przyklad Orlicze infty} to power functions, we can come full circle and return to Bennett's
	result \eqref{EQ: cesp -> cesq <=> M = linfty(wpq)}.
	However, to make our lives easier, we would like to offer an alternative and somehow more transparent proof of
	\eqref{EQ: cesp -> cesq <=> M = linfty(wpq)} that bypasses the modular specificity of Orlicz spaces and does not directly
	use the technology invented in \cite{DR00} or \cite{LT21}.
	
	\begin{corollary}[G. Bennett, 1996] \label{COR: Bennett M(ces_p,ces_q)}
		{\it Let $1 < p \leqslant q < \infty$ with $1/r = 1/q - 1/p$. Then}
		\begin{equation} \label{NFAK}
			M(ces_p,ces_q) = \ell_\infty(n^{1/r}),
		\end{equation}
		{\it where $\ell_\infty( n^{1/r})
			\coloneqq \left\{ x = \{x_n\}_{n=1}^\infty \colon \norm{x}_{\ell_\infty(n^{1/r})} = \sup\limits_{n \in \mathbb{N}} n^{1/r} \abs{x_n} < \infty \right\}$.
		Moreover, we have}
		\begin{equation} \label{FAK}
			M(ces_q,ces_p) = \widetilde{\ell_r},
		\end{equation}
		{\it where $\widetilde{\ell_r}
			\coloneqq \left\{ x = \{x_n\}_{n=1}^\infty \colon \norm{x}_{\widetilde{\ell_r}} = \left( \sum_{n=1}^\infty \sup\limits_{k \geqslant n} \abs{x_k}^r \right)^{1/r} < \infty \right\}$.}
	\end{corollary}
	\begin{proof}
		Let us start from the beginning, that is, from the proof of \eqref{NFAK}. We have
		\begin{align*}
			M({\bf E}(\ell_p),{\bf E}(\ell_q))
				& \equiv M\left( \ell_p(2^{n/p}), \ell_q(2^{n/q}) \right) \quad (\text{since ${\bf E}(\ell_r) \equiv \ell_r(2^{n/r})$ for $1 \leqslant r < \infty$}) \\
				& \equiv M( \ell_p,\ell_q)\left( 2^{n(1/q - 1/p)} \right) \quad (\text{because $M(\ell_p(w),\ell_q(\upsilon)) \equiv M(\ell_p,\ell_q)(\upsilon/w)$}) \\
				& \equiv \ell_\infty\left( 2^{n(1/q - 1/p)} \right) \quad (\text{since $\ell_p \hookrightarrow \ell_q$, so $M(\ell_p, \ell_q) \equiv \ell_\infty$})
		\end{align*}
		Due to Theorem~\ref{Thm: Tandori sequence representation}, one can identify the space $ces_p$ with $\bigl( \bigoplus_{n=0}^{\infty} \ell_1^{2^n} \bigr)_{\ell_p(2^{n/p})}$.
		This, together with the fact that $M(\ell_1^{2^n}, \ell_1^{2^n}) \equiv \ell_\infty^{2^n}$ and Theorem~\ref{Thm: komutowanie M z cdot}
		(cf. \cite[Theorem~1]{Kel71}), gives us that
		\begin{equation*}
			M(ces_p,ces_q)
				= M\left( \biggl( \bigoplus_{n=0}^{\infty} \ell_1^{2^n} \biggr)_{\ell_p(2^{n/p})} , \biggl( \bigoplus_{n=0}^{\infty} \ell_1^{2^n} \biggr)_{\ell_q(2^{n/q})} \right)
				\equiv \biggl( \bigoplus_{n=0}^\infty \ell_\infty^{2^n} \biggr)_{\ell_{\infty}\left( 2^{n\left( 1/q - 1/p \right)} \right)}.
		\end{equation*}
		However, since
		\begin{equation*}
			\biggl( \bigoplus_{n=0}^\infty \ell_\infty^{2^n} \biggr)_{\ell_{\infty}\left( 2^{n\left( 1/q - 1/p \right)} \right)} 
				= \ell_\infty(n^{1/q - 1/p})
				\equiv \ell_\infty(n^{1/r}),
		\end{equation*}
		so \eqref{NFAK} follows (in case, the proof of the less obvious equality is included in the first part of the proof of Example~\ref{przyklad Orlicze infty}).
		
		Now, let us move on to the proof of \eqref{FAK}. We have
		\begin{align*}
			M(ces_q,ces_p)
				& = M\left( \biggl( \bigoplus_{n=0}^{\infty} \ell_1^{2^n} \biggr)_{\ell_q(2^{n/q})} , \biggl( \bigoplus_{n=0}^{\infty} \ell_1^{2^n} \biggr)_{\ell_p(2^{n/p})} \right)
						\quad (\text{just as above}) \\
				& \equiv \biggl( \bigoplus_{n=0}^\infty \ell_\infty^{2^n} \biggr)_{M(\ell_q,\ell_p)\left( 2^{n\left( 1/q - 1/p \right)} \right)}
						\quad (\text{by Theorem~\ref{Thm: komutowanie M z cdot}}) \\
				& \equiv \biggl( \bigoplus_{n=0}^\infty \ell_\infty^{2^n} \biggr)_{\ell_{r}(2^{n/r})} \quad (\text{since $M(\ell_q,\ell_p) \equiv \ell_r$ with $1/r = 1/q - 1/p$}) \\
				& \equiv	\biggl( \bigoplus_{n=0}^\infty \ell_\infty^{2^n} \biggr)_{{\bf E}(\ell_r)}
						\quad (\text{because ${\bf E}(\ell_r) \equiv \ell_r(2^{n/r})$ for $1 \leqslant r < \infty$}) \\
				& = \widetilde{\ell_r} \quad (\text{due to Theorem~\ref{Thm: Tandori sequence representation}}).
		\end{align*}
		The proof has been completed.
	\end{proof}

	Let us now focus on Lorentz and Marcinkiewicz sequence spaces.

	\begin{lemma} \label{EXAMPLE: CLorentz i CMarcinkiewicz a'la Bennett}
		{\it Let $X$ and $Y$ be two r.i. sequence space with $\alpha_X, \alpha_Y > 1$. Further, let $\varphi$ and $\psi$ be two increasing concave function with
		$1 < \delta_{\varphi}, \delta_{\psi} < \infty$. Then}
		\begin{equation*}
			M(\mathscr{C}m_{\varphi},\mathscr{C}Y) = \widetilde{Y(w_{\varphi})} \quad \textit{ and } \quad
				M(\mathscr{C}X,\mathscr{C}\lambda_{\psi}) = \widetilde{X^{\times}(W_{\psi})},
		\end{equation*}
		{\it where $w_{\varphi}(n) \coloneqq 1/\varphi(n)$ and, respectively, $W_{\psi}(n) \coloneqq \psi(n)/n$ for $n \in \mathbb{N}$. In particular,}
		\begin{equation*}
			M(\mathscr{C}m_{\varphi},\mathscr{C}m_{\psi}) = \widetilde{\ell_{\infty}(\psi/\varphi)},
				\quad M(\mathscr{C}m_{\varphi},\mathscr{C}\lambda_{\psi}) = \widetilde{\ell_1(\Delta\psi/\varphi)}
				\quad \textit{ and } \quad M(\mathscr{C}\lambda_{\varphi},\mathscr{C} \lambda_{\psi}) = \widetilde{\ell_{\infty}(\psi/\varphi)},
		\end{equation*}
		{\it where $\Delta\psi(n) \coloneqq \psi(n+1) - \psi(n)$ for $n \in \mathbb{N}$.}
	\end{lemma}
	\begin{proof}
		The argument basically rests on the following observation: the spaces ${\bf E}(Y)$ and ${\bf E}(\lambda_{\psi})$ can be factorized through ${\bf E}(m_{\varphi})$
		and, respectively, ${\bf E}(X)$. Having realized this, in view of Theorem~\ref{Cor: multipliers between Cesaro and Tandori}, we have
		\begin{equation} \label{EQ: MCmfi,CY}
			M(\mathscr{C}m_{\varphi},\mathscr{C}Y) = \widetilde{M(m_{\varphi},Y)} = \widetilde{\left[ Y(w_{\varphi}) \right]^{\bigstar}} \equiv \widetilde{Y(w_{\varphi})},
		\end{equation}
		and, respectively,
		\begin{equation*}
			M(\mathscr{C}X,\mathscr{C}\lambda_{\psi}) = \widetilde{M(X,\lambda_{\psi})} = \widetilde{\left[ X^{\times}(W_{\psi}) \right]^{\bigstar}} \equiv \widetilde{X^{\times}(W_{\psi})},
		\end{equation*}
		where every second equality follows from \cite[Lemma~4.3]{KT23}. Therefore, it remains to justify the observation with which we started.
		This is, however, straightforward. We have
		\begin{align*}
			{\bf E}(m_{\varphi}) \odot M({\bf E}(m_{\varphi}),{\bf E}(Y))
				& = \ell_{\infty}(\varphi(2^n)) \odot M(\ell_{\infty}(\varphi(2^n)),{\bf E}(Y)) \quad (\text{by Proposition~\ref{PROP: EX <-> carrier}})\\
				& \equiv \ell_{\infty}(\varphi(2^n)) \odot M(\ell_{\infty},{\bf E}(Y))(1/\varphi(2^n)) \\
				& \equiv \ell_{\infty}(\varphi(2^n)) \odot {\bf E}(Y)(1/\varphi(2^n)) \quad (\text{since $M(\ell_{\infty},E) \equiv E$}) \\
				& \equiv \ell_{\infty} \odot {\bf E}(Y) \\
				& \equiv {\bf E}(Y).
		\end{align*}
		In a completely similar way
		$${\bf E}(\lambda_{\varphi}) = {\bf E}(X) \odot M({\bf E}(X),{\bf E}(\lambda_{\varphi})).$$
		
		As for the last part, we will now only show that
		$$M(\mathscr{C}m_{\varphi},\mathscr{C}m_{\psi}) = \widetilde{\ell_{\infty}(\psi/\varphi)}.$$
		The rest is essentially the same. We have
		\begin{align*}
			M(\mathscr{C}m_{\varphi},\mathscr{C}m_{\psi})
				& = \widetilde{m_{\psi}(w_{\varphi})} \quad (\text{using \eqref{EQ: MCmfi,CY}}) \\
				& = \widetilde{\left[ \ell_{\infty}(\psi) \right]^{\bigstar}(w_{\varphi})} \quad (\text{because $m_{\psi} = \left[ \ell_{\infty}(\psi) \right]^{\bigstar}$}) \\
				& \equiv \widetilde{\left[ \ell_{\infty}(\psi) \right](w_{\varphi})} \quad (\text{since the function $n \rightsquigarrow \widetilde{x}(n)/\varphi(n)$ is non-increasing}) \\
				& \equiv \widetilde{\ell_{\infty}(\psi/\varphi)}.
		\end{align*}
	\end{proof}

	\begin{remark}[About Lemma~\ref{EXAMPLE: CLorentz i CMarcinkiewicz a'la Bennett}] \label{REMARK: warunki rownowazne M = linfty}
		Let us keep the setting of Lemma~\ref{EXAMPLE: CLorentz i CMarcinkiewicz a'la Bennett}. Then the following four conditions
		\begin{enumerate}
			\item[(a)] $M(m_{\varphi},Y) = \ell_{\infty}$,
			\item[(b)] $[Y(w_{\varphi})]^{\bigstar} = \ell_{\infty}$,
			\item[(c)] $m_{\varphi} \hookrightarrow Y$,
			\item[(d)] $w_{\varphi} \in Y$,
		\end{enumerate}
		are equivalent. A short explanation seems in order. To start with, $(a)$ and $(b)$ are equivalent,
		because $M(m_{\varphi},Y) = [Y(w_{\varphi})]^{\bigstar}$ (see \cite[Lemma~4.3]{KT23}).
		Also, it is clear that $(a)$ implies $(c)$ and {\it vice versa}.
		Next, we claim that $(c)$ and $(d)$ are equivalent. Indeed, if $w_{\varphi} \in Y$, then $\ell_{\infty}(\varphi) \hookrightarrow Y$ and,
		consequently,
		\begin{equation*}
			m_{\varphi} = \left[ \ell_{\infty}(\varphi) \right]^{\bigstar} \hookrightarrow Y^{\bigstar} \equiv Y,
		\end{equation*}
		where the first equality is due to our assumption that $\delta_{\varphi} > 1$.
		But this means that $(c)$ holds. On the other hand, assuming that $m_{\varphi} \hookrightarrow Y$, we have
		\begin{equation*}
			w_{\varphi} = w_{\varphi}^{\star} \in [\ell_{\infty}(\varphi)]^{\bigstar} = m_{\varphi} \hookrightarrow Y. 
		\end{equation*}
		Therefore, $(c)$ implies $(d)$ and this proves our claim.
		\demo
	\end{remark}

	Therefore, in the light of what we have noted above, the analogue of Bennett's result for Lorentz and Marcinkiewicz sequence spaces is as follows.
	
	\begin{example} \label{EXAMPLE: Lorentz i Marcinkiewicz Bennett <=>}
		{\it Let $X$ and $Y$ be a r.i. sequence space with $\alpha_X, \alpha_Y > 1$. Let, moreover, $\varphi$ and $\psi$ be two increasing concave function with
			$1 < \delta_{\varphi}, \delta_{\psi} < \infty$. Suppose that $m_{\varphi} \hookrightarrow Y$ and $X \hookrightarrow \lambda_{\psi}$. Then}
		\begin{equation*}
			M(\mathscr{C}m_{\varphi},\mathscr{C}Y) = \ell_{\infty}
				\quad \textit{ and, respectively, } \quad M(\mathscr{C}X, \mathscr{C}\lambda_{\psi}) = \ell_{\infty}.
		\end{equation*}
		{\it In particular, for any increasing concave function $\phi$ with $1 < \delta_{\phi} < \infty$, we have}
		\begin{equation*}
			M(\mathscr{C}m_{\varphi},\mathscr{C}m_{\phi}) = \ell_{\infty} \quad [\textit{resp. } M(\mathscr{C}\lambda_{\phi},\mathscr{C}\lambda_{\psi}) = \ell_{\infty}]
				\quad \textit{ if, and only if, } \quad \phi \preccurlyeq \varphi \quad [\textit{resp. } \psi \preccurlyeq \phi].
		\end{equation*}
		{\it In addition, assuming that $\varphi(n) \leqslant \phi(n)$ for all $n \in \mathbb{N}$, we also have}
		\begin{equation*}
			M(\mathscr{C}m_{\varphi},\mathscr{C}\lambda_{\phi}) = \ell_{\infty} \quad \textit{ if, and only if, } \quad
				\sum_{n=1}^{\infty} \frac{\phi(n+1) - \phi(n)}{\varphi(n)} < \infty.
		\end{equation*}
	\end{example}
	\begin{proof}
		The first part that $M(\mathscr{C}m_{\varphi},\mathscr{C}Y) = \ell_{\infty}$ and $M(\mathscr{C}X, \mathscr{C}\lambda_{\psi}) = \ell_{\infty}$
		is clear and basically follows from Lemma~\ref{EXAMPLE: CLorentz i CMarcinkiewicz a'la Bennett} combined with Remark~\ref{REMARK: warunki rownowazne M = linfty}.
		
		Thus, let us move on to the second part.
		We want to show that $M(\mathscr{C}m_{\varphi},\mathscr{C}m_{\phi}) = \ell_{\infty}$ if, and only if, $\phi \preccurlyeq \varphi$.
		Suppose that $\phi \preccurlyeq \varphi$.
		Then $m_{\varphi} \hookrightarrow m_{\phi}$ and, in consequence, $\mathscr{C}m_{\varphi} \hookrightarrow \mathscr{C}m_{\phi}$.
		Therefore, a constant sequence belongs to $M(\mathscr{C}m_{\varphi},\mathscr{C}m_{\phi})$. Moreover, since the space $M(m_{\varphi},m_{\phi})$
		is rearrangement invariant, so in view of Lemma~\ref{EXAMPLE: CLorentz i CMarcinkiewicz a'la Bennett}, we have
		\begin{equation*}
			\ell_{\infty} \hookrightarrow M(\mathscr{C}m_{\varphi},\mathscr{C}m_{\phi}) = \widetilde{M(m_{\varphi},m_{\phi})} \hookrightarrow \widetilde{\ell_{\infty}} \equiv \ell_{\infty}.
		\end{equation*}
		To show the second implication, assume that $M(\mathscr{C}m_{\varphi},\mathscr{C}m_{\phi}) = \ell_{\infty}$. Then, of course, $\mathscr{C}m_{\varphi} \hookrightarrow \mathscr{C}m_{\phi}$.
		However, mimicking the argument used in Claim~\ref{LEMMA: X=Y<=>TandoriX=TandoriY}, it is straightforward to see that then $m_{\varphi} \hookrightarrow m_{\phi}$.
		But this, in turn, means that $\phi \preccurlyeq \varphi$.
		The second statement that $M(\mathscr{C}\lambda_{\phi},\mathscr{C}\lambda_{\psi}) = \ell_{\infty}$ if, and only if $\phi \preccurlyeq \psi$ has virtually the same proof.
		
		Finally, let us move on to the last part. First, we claim that $m_{\varphi} \hookrightarrow \lambda_{\phi}$ if, and only if,
		\begin{equation} \label{EQ: marcinkiewicz w lorentza}
			\sum_{n=1}^{\infty} \frac{\phi(n+1) - \phi(n)}{\varphi(n)} < \infty.
		\end{equation}
		To see this, suppose that $m_{\varphi} \hookrightarrow \lambda_{\phi}$. Since $1/\varphi \in m_{\varphi}$, so also $1/\varphi \in \lambda_{\phi}$. But the latter means that
		$\norm{1/\varphi}_{\lambda_{\phi}} = \sum_{n=1}^{\infty} \frac{1}{\varphi(n)}\left[ \phi(n+1) - \phi(n) \right] < \infty$. Conversely, suppose that \eqref{EQ: marcinkiewicz w lorentza} holds.
		Take $f \in m_{\varphi}$. Since $m_{\varphi} = [\ell_{\infty}(\varphi)]^{\bigstar}$, so it is clear that $f \in m_{\varphi}$ if, and only if, there is $C = C(f) > 0$
		such that $f^{\star}(n) \leqslant C/\varphi(n)$ for $n \in \mathbb{N}$. Thus, we have
		\begin{equation*}
			\norm{f}_{\lambda_{\phi}} = \sum_{n=1}^{\infty} f^{\star}(n) \left[ \phi(n+1) - \phi(n) \right] \leqslant C \sum_{n=1}^{\infty} \frac{\phi(n+1) - \phi(n)}{\varphi(n)} < \infty.
		\end{equation*}
		In other words, $f \in \lambda_{\phi}$ and our claim follows.
		We are now ready to complete the proof. Suppose that \eqref{EQ: marcinkiewicz w lorentza} holds. Then $m_{\varphi} \hookrightarrow \lambda_{\phi}$ and, consequently,
		$M(m_{\varphi},\lambda_{\phi}) = \ell_{\infty}$. Moreover, due to Theorem~\ref{Cor: multipliers between Cesaro and Tandori}, we have
		\begin{equation} \label{EQ: ===}
			M(\mathscr{C}m_{\varphi},\mathscr{C}\lambda_{\phi}) = \widetilde{M(m_{\varphi},\lambda_{\phi})} = \widetilde{\ell_{\infty}} \equiv \ell_{\infty}.
		\end{equation}
		Conversely, assume that $M(\mathscr{C}m_{\varphi},\mathscr{C}\lambda_{\phi}) = \ell_{\infty}$. Then, in view of \eqref{EQ: ===}, $\widetilde{M(m_{\varphi},\lambda_{\phi})} = \ell_{\infty}$.
		However, since the space $M(m_{\varphi},\lambda_{\phi})$ is rearrangement invariant and $\widetilde{\ell_{\infty}} \equiv \ell_{\infty}$, so it follows from Lemma~\ref{LEMMA: X=Y<=>TandoriX=TandoriY}
		that $M(m_{\varphi},\lambda_{\phi}) = \ell_{\infty}$. But this means that $m_{\varphi} \hookrightarrow \lambda_{\phi}$, which in the context of what we said above is
		equivalent to the fact that $\sum_{n=1}^{\infty} \frac{\phi(n+1) - \phi(n)}{\varphi(n)} < \infty$.
	\end{proof}
	
	\subsection{Pointwise products} \label{SUBSECTION: pointwise products}
	
	The overall result about the pointwise product looks like this.
	
	\begin{theorem}[Pointwise product of amalgams] \label{PROP: Pointwise multipliers of amalgams}
		{\it Let $E$ and $F$ be two Banach sequence spaces both defined on $J$. Further, let $\{X_j\}_{j \in J}$ and $\{Y_j\}_{j \in J}$
			be two families of Banach ideal spaces. Suppose that for any $j \in J$ both spaces $X_j$ and $Y_j$ are defined on the same measure
			space and have the Fatou property. Then}
		\begin{equation*}
			\Bigl( \bigoplus_{j \in J} X_j \Bigr)_E \odot \Bigl( \bigoplus_{j \in J} Y_j \Bigr)_F
				\equiv \Bigl( \bigoplus_{j \in J} X_j \odot Y_j \Bigr)_{E \odot F}.
		\end{equation*}
	\end{theorem}
	\begin{proof}
		We know from Schep \cite{Sch10} that we can represent the pointwise product space $X \odot Y$ as the $\frac{1}{2}$-concavification
		of the Calder{\' o}n space $X^{1/2}Y^{1/2}$, that is to say,
		\begin{equation} \label{EQ: X o Y as Calderon product}
			X \odot Y \equiv \left( X^{1/2}Y^{1/2} \right)^{(1/2)}
		\end{equation}
		(see the proof of Theorem~2.1 in \cite{Sch10}; cf. \cite[Theorem~1(iv)]{KLM14} where this fact is stated explicitly).
		Moreover, it is not hard to see that the $\frac{1}{2}$-concavification of $\bigl( \bigoplus_{j \in J} Z_j \bigr)_G$,
		where $\{ Z_j \}_{j \in J}$ is the family of Banach ideal spaces and $G$ is a Banach sequence space modeled on $J$, can be identified
		with the space $\bigl( \bigoplus_{j \in J} Z_j^{(1/2)} \bigr)_{G^{(1/2)}}$, that is,
		\begin{equation} \label{EQ: convexifications}
			\left[ \Bigl( \bigoplus_{j \in J} Z_j \Bigr)_G \right]^{(1/2)} \equiv \Bigl( \bigoplus_{j \in J} Z_j^{(1/2)} \Bigr)_{G^{(1/2)}}.
		\end{equation}
		Even more generally, the space $\left[ \bigl( \bigoplus_{j \in J} Z_j \bigr)_G \right]^{(p)}$, where $0 < p < \infty$, can be identified
		with $\bigl( \bigoplus_{j \in J} Z_j^{(p)} \bigr)_{G^{(p)}}$.
		To see this, take $f = \{f_j\}_{j \in J}$ from $\left[ \bigl( \bigoplus_{j \in J} Z_j \bigr)_G \right]^{(p)}$. This clearly means that
		$\abs{f}^p = \left\{ \abs{f_j}^p \right\}_{j \in J}$ belongs to $\bigl( \bigoplus_{j \in J} Z_j \bigr)_G$. Moreover, we have
		\begin{equation*}
			\norm{\abs{f}^p}^{(1/p)}_{\bigl( \bigoplus_{j \in J} Z_j \bigr)_G}
				= \norm{\sum_{j \in J} \norm{\abs{f_j}^p}_{Z_j} e_j}^{1/p}_G
				= \norm{ \left( \sum_{j \in J} \norm{\abs{f_j}^p}^{1/p}_{Z_j} e_j \right)^{p} }^{1/p}_G
				= \norm{f}_{\left( \bigoplus_{j \in J} Z_j^{(p)} \right)_{G^{(p)}}}
		\end{equation*}
		and \eqref{EQ: convexifications} follows. Consequently, we have
		\begin{align*}
			\Bigl( \bigoplus_{j \in J} X_j \Bigr)_E \odot \Bigl( \bigoplus_{j \in J} Y_j \Bigr)_F
				& \equiv \left( \left[ \Bigl( \bigoplus_{j \in J} X_j \Bigr)_E \right]^{1/2} \left[ \Bigl( \bigoplus_{j \in J} Y_j \Bigr)_F \right]^{1/2} \right)^{(1/2)}
					\quad (\text{by \eqref{EQ: X o Y as Calderon product}}) \\
				& \equiv \left[ \Bigl( \bigoplus_{j \in J} X_j^{1/2}Y_j^{1/2} \Bigr)_{E^{1/2}F^{1/2}} \right]^{(1/2)} \quad (\text{using Theorem~\ref{THM: Calderon product of amalgams}}) \\
				& \equiv \left[ \bigoplus_{j \in J} \left( X_j^{1/2}Y_j^{1/2} \right)^{(1/2)} \right]_{\left( E^{1/2}F^{1/2} \right)^{(1/2)}} \quad (\text{by \eqref{EQ: convexifications}}) \\
				& \equiv \Bigl( \bigoplus_{j \in J} X_j \odot Y_j \Bigr)_{E \odot F} \quad (\text{again, by \eqref{EQ: X o Y as Calderon product}}).
		\end{align*}
		This ends the proof.
	\end{proof}

	This way, in particular, we recover Buntinas's result from \cite{Bun87} (see \cite[Theorem~9]{Bun87}; cf. \cite[Theorem~8.1, p.~42]{GE98}).

	\begin{corollary}[M. Buntinas, 1987]
		{\it Let $1 \leqslant p,q,r,s \leqslant \infty$. Further, let $\{d_n\}_{n=1}^{\infty}$ be a sequence of positive integers. Then}
		\begin{equation*}
			\Bigl( \bigoplus_{n=1}^{\infty} \ell_p^{d_n} \Bigr)_{\ell_q} \odot \Bigl( \bigoplus_{n=1}^{\infty} \ell_r^{d_n} \Bigr)_{\ell_s}
				\equiv \Bigl( \bigoplus_{n=1}^{\infty} \ell_{1/p+1/r}^{d_n} \Bigr)_{\ell_{1/q+1/r}}.
		\end{equation*}
	\end{corollary}
	
	The following lemma will prove useful many times.
	
	\begin{lemma}[Discretization of pointwise products] \label{PROP: E komutuje z produktem}
		{\it Let $X$ and $Y$ be two Banach ideal spaces defined on the same measure space. Suppose that the averaging operator}
			\begin{equation*}
				\text{Ave} \colon f \rightsquigarrow \left[ t \rightsquigarrow \sum_{j \in \mathbb{J}} \left( 2^{-j} \int_{\Delta_j} h(t)dt \right)\chi_{\Delta_j}(t) \right]
			\end{equation*}
			{\it is bounded on both spaces $X$ and $Y$. Then}
			\begin{equation*}
				{\mathbf E}(X) \odot {\mathbf E}(Y) \overset{1}{\hookrightarrow} {\mathbf E}(X \odot Y)
					\overset{C}{\hookrightarrow} {\mathbf E}(X) \odot {\mathbf E}(Y),
			\end{equation*}
			{\it where} $C = \norm{\text{Ave}}_{X \rightarrow X} \norm{\text{Ave}}_{Y \rightarrow Y}$.
	\end{lemma}
	\begin{proof}
		Evidently,
		\begin{equation*}
			\norm{\text{id} \colon \mathbf{E}(X) \odot \mathbf{E}(Y) \rightarrow \mathbf{E}(X \odot Y)} \leqslant 1,
		\end{equation*}
		so it is enough to show the reverse embedding.
		
		To do this, take $z = \{z_j\}_{j \in \mathbb{J}}$ from $\mathbf{E}(X \odot Y)$. It follows from the very definition of the product space $X \odot Y$ that
		\begin{equation*}
			h \coloneqq \sum_{j \in \mathbb{J}} z_j \chi_{\Delta_j} = fg \quad \text{ for some } \quad f \in X \quad \text{ and } \quad g \in Y.
		\end{equation*}
		Therefore, using the inverse Chebyshev's inequality, we get
		\begin{equation*}
			h = \text{Ave}(h) = \text{Ave}(fg) \leqslant \text{Ave}(f) \cdot \text{Ave}(g).
		\end{equation*}
		Since, due to our assumptions, the averaging operator $\text{Ave}$ is bounded on $X$ and $Y$, so it is straightforward to see that
		\begin{equation*}
			\left\{ 2^{-j} \int_{\Delta_j} f(t)dt \right\}_{j \in \mathbb{J}} \in \mathbf{E}(X) \quad \text { and } \quad
				\left\{ 2^{-j} \int_{\Delta_j} g(t)dt \right\}_{j \in \mathbb{J}} \in \mathbf{E}(Y).
		\end{equation*}
		In consequence, for $C = \norm{\text{Ave}}_{X \rightarrow X} \norm{\text{Ave}}_{Y \rightarrow Y}$, we have
		\begin{align*}
			\norm{z}_{{\mathbf E}(X \odot Y)} = \norm{h}_{X \odot Y}
				& = \inf\limits_{\substack{h = fg \\ f \in X,\, g \in Y}} \norm{f}_X \norm{g}_Y \\
				& \geqslant \inf\limits_{\substack{z \leqslant xy \\ x \in \mathbf{E}(X),\, y \in \mathbf{E}(Y)}}
					\frac{\norm{x}_{\mathbf{E}(X)}}{\norm{\text{Ave}}_{X \rightarrow X}} \frac{\norm{y}_{\mathbf{E}(Y)}}{\norm{\text{Ave}}_{Y \rightarrow Y}} \\
				& = C^{-1} \inf\limits_{\substack{z = xy \\ x \in \mathbf{E}(X),\, y \in \mathbf{E}(Y)}} \norm{x}_{\mathbf{E}(X)} \norm{y}_{\mathbf{E}(Y)} \\
				& = C^{-1} \norm{z}_{{\mathbf E}(X) \odot {\mathbf E}(Y)},
		\end{align*}
		where the third equality follows from \cite[Proposition~1]{KLM14}.
		But this means that
		\begin{equation*}
			\norm{\text{id} \colon \mathbf{E}(X \odot Y) \rightarrow \mathbf{E}(X) \odot \mathbf{E}(Y)} \leqslant C.
		\end{equation*}
	\end{proof}

	Recall that a quasi-Banach ideal space $X$ is said to be {\bf $L$-convex} whenever $X$ is $p$-convex for some $p > 0$ (another,
	more geometric, definition can be found in Kalton's paper \cite{Kal84}).
	
	\begin{remark}[About normability of quasi-Banach spaces]
		Let $X = (X,\norm{\cdot}_X)$ be a quasi-Banach ideal space. We say that the space $X$ is {\bf normable} if one can equip $X$
		with a norm equivalent to the initial quasi-norm. It is known that $X$ is normable if, and only if, it is
		$1$-convex (see, for example, \cite[p.~102]{Mal04b}). In this case, an equivalent norm on $X$ has the following description
		\begin{equation*}
			\norm{f} \coloneqq \inf \left\{ \sum_{n=1}^N \norm{f_n}_X \colon f = \sum_{n=1}^N f_n \text{ with } \{ f_n \}_{n=1}^N \subset X \right\}.
		\end{equation*}
		\demo
	\end{remark}
	
	\begin{theorem}[Pointwise products of Tandori spaces] \label{PROP: Tandori komutuje z produktem}
		{\it Let $X$ and $Y$ be two r.i. quasi-Banach ideal spaces defined on the same measure space. Suppose that both spaces $X$ and $Y$
		are $L$-convex and have the Fatou property. Then}
		\begin{equation*}
			\widetilde{X} \odot \widetilde{Y} = \widetilde{X \odot Y}.
		\end{equation*}
	\end{theorem}
	\begin{proof}
		Before moving on to the main part, let us first explain how to use the block form representations proposed in Theorems~\ref{Thm: Tandori sequence representation}
		and \ref{Thm: Tandori function representation} in the case when we are in the Kalton's zone.
		
		Let $Z$ be a quasi-Banach ideal space. Then, as essentially follows from the definition, the family of spaces $\widetilde{Z}$
		enjoy a very useful homogeneity property (cf. \cite[(d) on p.~72]{Be96}). By this we mean that for any $0 < p < \infty$, we have
		\begin{equation} \label{EQ: Tandori komutuje z (p)}
			\widetilde{Z}^{(p)} \equiv \widetilde{Z^{(p)}}.
		\end{equation}
		Now, due to our assumptions, $X$ is $p$-convex and $Y$ is $q$-convex for some $p,q > 0$. Actually, without loss of generality we can assume that
		$0 < p,q < 1$ (otherwise, either both spaces are normable or the described procedure only needs to be applied to one of them).
		Set $r = 2\max\{1/p,1/q\}$. Then, since the space $\widetilde{X}^{(r)}$ is normable, so using Theorems~\ref{Thm: Tandori sequence representation}
		and \ref{Thm: Tandori function representation}, we have
		\begin{equation} \label{EQ: tandori X (r)}
			\widetilde{X}^{(r)}
				\equiv \widetilde{X^{(r)}}
				= \Bigl( \bigoplus_{j \in \mathbb{Z}} L_{\infty}(\Delta_j) \Bigr)_{\mathbf{E}\left( X^{(r)} \right)}.
		\end{equation}
		Similarly,
		\begin{equation} \label{EQ: tandori Y (r)}
			\widetilde{Y}^{(r)} \equiv \widetilde{Y^{(r)}} = \Bigl( \bigoplus_{j \in \mathbb{Z}} L_{\infty}(\Delta_j) \Bigr)_{\mathbf{E}\left( Y^{(r)} \right)}.
		\end{equation}
		
		Time to get to the point. We have
		\begin{align*}
			\left( \widetilde{X} \odot \widetilde{Y} \right)^{(r)}
				& \equiv \widetilde{X}^{(r)} \odot \widetilde{Y}^{(r)} \\
				& \equiv \widetilde{X^{(r)}} \odot \widetilde{Y^{(r)}} \quad (\text{due to \eqref{EQ: Tandori komutuje z (p)}}) \\
				& = \Bigl( \bigoplus_{j \in \mathbb{Z}} L_{\infty}(\Delta_j) \Bigr)_{\mathbf{E}\left( X^{(r)} \right)}
					\odot \Bigl( \bigoplus_{j \in \mathbb{Z}} L_{\infty}(\Delta_j) \Bigr)_{\mathbf{E}\left( Y^{(r)} \right)}
						\quad (\text{using \eqref{EQ: tandori X (r)} and \eqref{EQ: tandori Y (r)}}) \\
				& \equiv \Bigl( \bigoplus_{j \in \mathbb{Z}} L_{\infty}(\Delta_j) \odot L_{\infty}(\Delta_j) \Bigr)_{\mathbf{E}\left( X^{(r)} \right) \odot \mathbf{E}\left( Y^{(r)} \right)}
					\quad (\text{by Proposition~\ref{PROP: Pointwise multipliers of amalgams}}) \\
				& \equiv \Bigl( \bigoplus_{j \in \mathbb{Z}} L_{\infty}(\Delta_j) \Bigr)_{\mathbf{E}\left[ (X \odot Y)^{(r)} \right]}
					\quad (\text{using Lemma~\ref{PROP: E komutuje z produktem}}) \\
				& = \widetilde{(X \odot Y)^{(r)}} \quad (\text{by Theorems~\ref{Thm: Tandori sequence representation} and \ref{Thm: Tandori function representation}}) \\
				& \equiv \left( \widetilde{X \odot Y} \right)^{(r)} \quad (\text{again, due to \eqref{EQ: Tandori komutuje z (p)}})
		\end{align*}
		Thus, $$\widetilde{X} \odot \widetilde{Y} = \widetilde{X \odot Y}$$ and the proof is complete.
	\end{proof}

	\begin{remark}[About Theorem~\ref{PROP: Tandori komutuje z produktem}]
		The assumption that a quasi-Banach space is $L$-convex is very mild.
		Anyway, Kalton was able to give an example of a quasi-Banach ideal space $X$ which is not $L$-convex (see \cite[Example~2.4]{Kal84}).
		This example, however, is quite artificial and is based on the existence of a certain (so to speak, pathological) sub-measure
		that is non-trivial, but dominates no non-trivial measure. In practice, most naturally occurring spaces are $L$-convex.
		\demo
	\end{remark}
	
	Since $\ell_p$ spaces with $0 < p \leqslant \infty$ are clearly $L$-convex, so with the help of Theorem~\ref{PROP: Tandori komutuje z produktem}
	we easily get the following Bennett's result from \cite{Be96} (see \cite[Theorem~3.15, p.~12]{Be96}).

	\begin{corollary}[G. Bennett, 1996]
		{\it Let $0 < p,q \leqslant \infty$ with $1/r = 1/p + 1/q$. Then}
		\begin{equation*}
			\widetilde{\ell_p} \odot \widetilde{\ell_q} = \widetilde{\ell_r}.
		\end{equation*}
	\end{corollary}
	
	\section{{\bf Factorization}} \label{SECTION: Factorization}
	
	The last section is entirely devoted to different types of factorization. Here, all previously developed techniques and results come together.
	
	\subsection{General scheme after Lozanovski{\u \i}}
	Let us first consider the general factorization scheme in the class of direct sum of Banach ideal spaces.
	The essence of the next result is that in order to conclude about factorization of $\bigl( \bigoplus_{j \in J} Y_j \bigr)_F$ through
	$\bigl( \bigoplus_{j \in J} X_j \bigr)_E$, it is not enough to know that we have the factorization of corresponding components,
	but we must to know that these components can be factorized in a \enquote{uniform} way.
	
	\begin{proposition}[Factorization of amalgams]
		{\it Let $E$ and $F$ be two Banach sequence spaces both defined on $J$. Further, let $\{X_j\}_{j \in J}$ and $\{Y_j\}_{j \in J}$
			be two families of Banach ideal spaces. Suppose that for any $j \in J$ both spaces $X_j$ and $Y_j$ are defined on the same
			measure space $(\Omega_j,\Sigma_j,\mu_j)$ and have the Fatou property. Furthermore, suppose that}
		\begin{itemize}
			\item[$\bullet$] {\it the space $F$ factorizes through $E$, that is, $F = E \odot M(E,F)$,}
			\item[$\bullet$] {\it the family $\{Y_j\}_{j \in J}$ uniformly factorizes through the family $\{X_j\}_{j \in J}$, that is,}
				\begin{equation*}
					Y_j \overset{c_j}{\hookrightarrow} X_j \odot M(X_j,Y_j) \overset{C_j}{\hookrightarrow} Y_j \quad \textit{ for all } \quad j \in J
				\end{equation*}
				{\it and $0 < \inf_{j \in J} C_j \leqslant \sup_{j \in J} c_j < \infty$.}
		\end{itemize}
		{\it Then we have the following factorization}
		\begin{equation*}
			\Bigl( \bigoplus_{j \in J} Y_j \Bigr)_F \overset{c}{\hookrightarrow}
				\Bigl( \bigoplus_{j \in J} X_j \Bigr)_E \odot M\biggl( \Bigl( \bigoplus_{j \in J} X_j \Bigr)_E, \Bigl( \bigoplus_{j \in J} Y_j \Bigr)_F \biggr)
					\overset{C}{\hookrightarrow} \Bigl( \bigoplus_{j \in J} Y_j \Bigr)_F,
		\end{equation*}
		{\it where $c = (\sup_{j \in J} c_j)\norm{F \hookrightarrow E \odot M(E,F)}$ and $C = (\inf_{j \in J} C_j)\norm{E \odot M(E,F) \hookrightarrow F}$.
			In particular, the following Lozanovski{\u \i}'s type factorization holds, namely, for any $\varepsilon > 0$, we have}
		\begin{equation*}
			L_1\Bigl( \bigsqcup_{j \in J} \Omega_j \Bigr) \overset{1+\varepsilon}{\hookrightarrow}
				\Bigl( \bigoplus_{j \in J} X_j \Bigr)_E \odot \left[ \Bigl( \bigoplus_{j \in J} X_j \Bigr)_{E} \right]^{\times}
					\overset{1}{\hookrightarrow} L_1\Bigl( \bigsqcup_{j \in J} \Omega_j \Bigr).
		\end{equation*}
		{\it In the above we can put $\varepsilon = 0$ provided the space $E$ has the Fatou property.}
	\end{proposition}
	\begin{proof}
		The first part is a straightforward consequence of our assumptions in tandem with Theorem~\ref{Thm: komutowanie M z cdot}
		and Theorem~\ref{PROP: Pointwise multipliers of amalgams}. The second part results from the first one after applying 
		Lozanovski{\u \i}'s factorization \cite{Loz69} together with Proposition~\ref{Prop: Podstawowe wlasnosci sum prostych}(c)
		(see also Remark~\ref{REMARK: Direct sum = BFS}). 
	\end{proof}
	
	\subsection{Factorization through $L_p$'s}
	
	Below we propose the general variant of Bennett's factorization of Hardy's inequality from \cite{Be96}
	(see \cite[Theorem~4.5, p.~13]{Be96} and \cite[Chapter~8]{GE98}; cf. \cite[Theorem~7.7, p.~36]{GE98}).
	
	\begin{theorem}[Factorization of Ces{\` a}ro spaces through Lebesgue spaces] \label{THM: factorization CX with Lp}
		{\it Let $1 < p,q < \infty$ with $1/p + 1/q = 1$. Further, let $X$ be a r.i. space with the Fatou property such that Hardy's operator
			$\mathscr{H}$ is bounded on $X$. Suppose that $X$ factorizes through $L_p$, that is, $X = L_p \odot M(L_p,X)$. Then}
			\begin{equation*}
				M(L_p,\mathscr{C}X) = \Bigl( \bigoplus_{j \in \mathbb{J}} L_q(\Delta_j) \Bigr)_{\mathbf{E}(M(L_p,X))(w^{1/q})},
			\end{equation*}
			{\it where $w^{1/q}(j) \coloneqq 2^{-j/q}$ for $j \in \mathbb{J}$. Moreover, we have the following factorization}
			\begin{equation*}
				\mathscr{C}X = L_p \odot M(L_p, \mathscr{C}X).
			\end{equation*}
	\end{theorem}
	\begin{proof}
		Fix $1 < p < \infty$. Since, for $f \in L_p$, we have
		\begin{equation*}
			\norm{f}_{L_p}
				= \norm{ \left\{ \left( \int_{\Delta_j} \abs{f(t)}^p dt \right)^{1/p} \right\}_{j \in \mathbb{J}} }_{\ell_p(\mathbb{J})},
		\end{equation*}
		so
		\begin{equation} \label{EQ : L_p = suma prosta}
			L_p \equiv \Bigl( \bigoplus_{j \in \mathbb{J}} L_p(\Delta_j) \Bigr)_{\ell_p(\mathbb{J})}.
		\end{equation}
		Therefore, we get
		\begin{align*}
			M(L_p,\mathscr{C}X)
				& = M\biggl( \Bigl( \bigoplus_{j \in \mathbb{J}} L_p(\Delta_j) \Bigr)_{\ell_p(\mathbb{J})}, \Bigl( \bigoplus_{j \in \mathbb{J}} L_1(\Delta_j) \Bigr)_{\mathbf{E}(X)(w)} \biggr)
						\quad (\text{by \eqref{EQ : L_p = suma prosta} and Theorem \ref{Thm: Tandori function representation}}) \\
				& = \Bigl( \bigoplus_{j \in \mathbb{J}} M(L_p(\Delta_j),L_1(\Delta_j)) \Bigr)_{M\left( \ell_p(\mathbb{J}),\mathbf{E}(X)(w) \right)}
						\quad (\text{using Theorem~\ref{Thm: komutowanie M z cdot}}) \\
				& = \Bigl( \bigoplus_{j \in \mathbb{J}} L_q(\Delta_j) \Bigr)_{M\left( \ell_p(\mathbb{J}),\mathbf{E}(X)(w) \right)} \quad (\text{since $M(L_p,L_1) \equiv L_p^{\times} \equiv L_q$}),
		\end{align*}
		where $w(j) = 2^{-j}$ for $j \in \mathbb{J}$ and $1/p+1/q = 1$.
		Now, let us focus our attention on the space $M\left( \ell_p(\mathbb{J}),\mathbf{E}(X)(w) \right)$. We claim that
		\begin{equation} \tag{$\clubsuit$} \label{ROW}
			M\left( \ell_p(\mathbb{J}),\mathbf{E}(X)(w) \right) = \mathbf{E}(M(L_p,\mathscr{C}X))(w^{1/q}),
		\end{equation}
		where $w^{1/q}(j) \coloneqq 2^{-j/q}$ for $j \in \mathbb{J}$. Indeed, we have
		\begin{align*}
			M(\ell_p(\mathbb{J}),\mathbf{E}(X)(w))
				& = M\left( (\mathbf{E}(X)(w))^{\times}, \ell_q(\mathbb{J}) \right) \quad (\text{since $M(E,F) \equiv M(F^{\times},E^{\times})$}) \\
				& = M\left( \mathbf{E}(X^{\times}), \ell_q(\mathbb{J}) \right) \quad (\text{applying Proposition~\ref{Prop: E komutuje z Kothe dualem}}) \\
				& = M\left( \left( \mathbf{E}(X^{\times}) \right)^{(1/q)}, \ell_1(\mathbb{J}) \right)^{(q)} \quad (\text{because $M(E^{(r)},F^{(r)}) \equiv M(E,F)^{(r)}$}) \\
				& = \left( \left[ \left( \mathbf{E}(X^{\times}) \right)^{(1/q)} \right]^{\times} \right)^{(q)} \quad (\text{since $M(E,\ell_1) \equiv E^{\times}$}).
		\end{align*}
		Moreover,
		\begin{align*}
			\left( \left[ \left( \mathbf{E}(X^{\times}) \right)^{(1/q)} \right]^{\times} \right)^{(q)}
				& = \left( \left[ \mathbf{E}\left( \left( X^{\times} \right)^{(1/q)} \right) \right]^{\times} \right)^{(q)} \quad (\text{since ${\bf E}(X)^{(r)} \equiv {\bf E}(X^{(r)})$}) \\
				& = \left( \mathbf{E}\left[ \left( \left( X^{\times} \right)^{(1/q)} \right)^{\times} \right](w) \right)^{(q)} \quad (\text{by Proposition~\ref{Prop: E komutuje z Kothe dualem}}) \\
				& = \mathbf{E} \left( \left[ \left( \left( X^{\times} \right)^{(1/q)} \right)^{\times} \right]^{(q)} \right)(w^{1/q}) \\
				& = \mathbf{E}(M(L_p,X))(w^{1/q}).
		\end{align*}
		Thus the claim (\ref{ROW}) follows. In consequence, we have the following identification
		\begin{equation} \label{EQ: M(Lp,CX) =}
			M(L_p,\mathscr{C}X) = \Bigl( \bigoplus_{j \in \mathbb{J}} L_q(\Delta_j) \Bigr)_{\mathbf{E}(M(L_p,X))(w^{1/q})}.
		\end{equation}
		Going further, observe also that
		\begin{equation} \label{EQ : ELpw = lp}
			\mathbf{E}(L_p)(w^{1/p}) \equiv \ell_p(\mathbb{J}),
		\end{equation}
		where $w^{1/p}(j) \coloneqq 2^{-j/p}$ for $j \in \mathbb{J}$.
		In consequence, remembering about Lozanovski{\u \i}'s factorization result in the form $L_p \odot L_q \equiv L_1$, we have
		\begin{align*}
			L_p \odot M(L_p,\mathscr{C}X)
				& = \Bigl( \bigoplus_{j \in \mathbb{J}} L_p(\Delta_j) \Bigr)_{\ell_p(\mathbb{J})} \odot \Bigl( \bigoplus_{j \in \mathbb{J}} L_q(\Delta_j) \Bigr)_{\mathbf{E}(M(L_p,X))(w^{1/q})}
						\quad (\text{due to \eqref{EQ : L_p = suma prosta} and \eqref{EQ: M(Lp,CX) =}}) \\
				& = \Bigl( \bigoplus_{j \in \mathbb{J}} L_p(\Delta_j) \Bigr)_{\mathbf{E}(L_p)(w^{1/p})} \odot \Bigl( \bigoplus_{j \in \mathbb{J}} L_q(\Delta_j) \Bigr)_{\mathbf{E}(M(L_p,X))(w^{1/q})}
						\quad (\text{using \eqref{EQ : ELpw = lp}}) \\
				& = \Bigl( \bigoplus_{j \in \mathbb{J}} L_p(\Delta_j) \odot L_q(\Delta_j) \Bigr)_{\mathbf{E}(L_p)(w^{1/p}) \odot \mathbf{E}(M(L_p,X))(w^{1/q})}
						\quad (\text{by Proposition~\ref{PROP: Pointwise multipliers of amalgams}}) \\
				& = \Bigl( \bigoplus_{j \in \mathbb{J}} L_1(\Delta_j) \Bigr)_{\mathbf{E}(L_p \odot M(L_p,X))(w^{1/p + 1/q})}
						\quad (\text{by Lemma~\ref{PROP: E komutuje z produktem}}) \\
				& = \Bigl( \bigoplus_{j \in \mathbb{J}} L_1(\Delta_j) \Bigr)_{\mathbf{E}(X)(w)}
						\quad (\text{because we assumed that $L_p \odot M(L_p,X) = X$}) \\
				& = \mathscr{C}X \quad (\text{in view of Theorem~\ref{Thm: Tandori function representation}}).
		\end{align*}
		The proof follows.
	\end{proof}
	
	\begin{remark}[About Theorem~\ref{THM: factorization CX with Lp} with $p=1$]
		Let $X$ be a r.i. space with the Fatou property such that Hardy's operator $\mathscr{H}$ is bounded on $X$. Suppose that we have the factorization $L_1 \odot M(L_1,X) = X$.
		Then $L_1 \odot M(L_1,\mathscr{C}X) = \mathscr{C}X$. The proof goes as follows. Since
		\begin{equation*}
			L_1 \equiv \Bigl( \bigoplus_{j \in \mathbb{J}} L_1(\Delta_j) \Bigr)_{\ell_1(\mathbb{J})},
		\end{equation*}
		so for $w(j) \coloneqq 2^{-j}$ for $j \in \mathbb{J}$ we have
		\begin{align*}
			M(L_1,\mathscr{C}X)
			& = M\biggl( \Bigl( \bigoplus_{j \in \mathbb{J}} L_1(\Delta_j) \Bigr)_{\ell_1(\mathbb{J})}, \Bigl( \bigoplus_{j \in \mathbb{J}} L_1(\Delta_j) \Bigr)_{\mathbf{E}(X)(w)} \biggr)
					\quad (\text{by Theorem~\ref{Thm: Tandori function representation}}) \\
			& \equiv \Bigl( \bigoplus_{j \in \mathbb{J}} M(L_1(\Delta_j),L_1(\Delta_j)) \Bigr)_{M\left( \ell_1(\mathbb{J}),\mathbf{E}(X)(w) \right)}
					\quad (\text{using Theorem~\ref{Thm: komutowanie M z cdot}}) \\
			& \equiv \Bigl( \bigoplus_{j \in \mathbb{J}} L_{\infty}(\Delta_j) \Bigr)_{M\left( \ell_1(\mathbb{J}),\mathbf{E}(X)(w) \right)}
					\quad (\text{since $M(L_1,L_1) \equiv L_{\infty}$}) \\
			& \equiv \Bigl( \bigoplus_{j \in \mathbb{J}} L_{\infty}(\Delta_j) \Bigr)_{M\left( {\bf E}(L_1),\mathbf{E}(X) \right)}
					\quad (\text{in view of ${\bf E}(L_1) \equiv \ell_1(1/w)$}) \\
			& \equiv \Bigl( \bigoplus_{j \in \mathbb{J}} L_{\infty}(\Delta_j) \Bigr)_{{\bf E}(M(L_1,X))}.
		\end{align*}
		Then, using exactly the same ingredients as in the last part of the proof of Theorem~\ref{THM: factorization CX with Lp}, one can easily show that
		\begin{equation*}
			L_1 \odot M(L_1,\mathscr{C}X) = \mathscr{C}X.
		\end{equation*}
		\demo
	\end{remark}

	\begin{remark}[About Theorem~\ref{THM: factorization CX with Lp} with $p=\infty$]
		Here things become almost trivial, because as it turns out, every Banach ideal space with the Fatou property can be factorized through $L_{\infty}$
		(even up to the equality of norms). Indeed, it is straightforward to see that
		\begin{equation*}
			L_{\infty} \odot M(L_{\infty},X)
				\equiv L_{\infty} \odot M(X^{\times},L_1)
				\equiv L_{\infty} \odot X^{\times \times}
				\equiv L_{\infty} \odot X
				\equiv X.
		\end{equation*}
		Thus, in particular, $L_{\infty} \odot M(L_{\infty},\mathscr{C}X) \equiv \mathscr{C}X$.
		\demo
	\end{remark}

	\begin{remark}[Factorization through $p$-concavity]
		Fix $1 < p < \infty$. Let $X$ be a Banach ideal space with the Fatou property. As Schep proved (see \cite[Theorem~3.9]{Sch10}; cf. \cite[Lemma~2.5]{Nil85}),
		the factorization of the form
		\begin{equation*}
			L_p \odot M(L_p,X) = X
		\end{equation*}
		is {\it equivalent} to the fact that the space $X$ is {\bf $p$-concave with concavity constant equal one} (briefly, {\bf strictly $p$-concave});
		that is, for any finite family $\{f_j\}_{n=1}^N$ of functions from $X$, we have
		\begin{equation*}
			\left( \sum_{n=1}^N \norm{f_j}_X^p \right)^{1/p} \leqslant \norm{\left( \sum_{n=1}^{N} \abs{f_j}^p \right)^{1/p}}_X.
		\end{equation*}
		Now, remembering that the space $L_1$ is strictly $p$-concave, it is routine to verify that also the space $\mathscr{C}X$ is strictly $p$-concave provided $X$
		is strictly $p$-concave and Hardy's operator $\mathscr{H}$ is bounded on $X$ (see, for example, \cite[Example~6.5(c)]{KKM21} for the details). Thus, under these
		assumptions, we have
		\begin{equation*}
			L_p \odot M(L_p,\mathscr{C}X) = \mathscr{C}X.
		\end{equation*}
		Therefore, the real strength of Theorem~\ref{THM: factorization CX with Lp} is not to be found just in the factorization conclusion (which basically
		can be deduced form \cite{Sch10}), but rather in the possibility of finding an explicit description of the spaces $M(L_p,\mathscr{C}X)$ and $L_p \odot M(L_p,\mathscr{C}X)$.
		With this approach, many factorization results that do not result from Schep's paper \cite{Sch10} can be drawn
		(we will say more about this later in Section~\ref{SUBSECTION: Faktoryzacje CX przez CX i inne}).
		\demo
	\end{remark}
	
	Let us see how both Bennett's and Astashkin and Maligranda's factorizations follow from Theorem~\ref{THM: factorization CX with Lp}
	(comp. with \cite[Theorem~4.5, p.~13]{Be96} and \cite[Proposition~1]{AM09}).
	
	\begin{corollary}[G. Bennett, 1996/S. V. Astashkin and L. Maligranda, 2009] \label{COR: B & AM factorization}
		{\it Let $1 < p,q < \infty$ with $1/p + 1/q = 1$. Then}
		\begin{equation} \label{EQ : factorization HLR seq}
			ces_p = \ell_p \odot g_q,
		\end{equation}
		{\it where
		$g_q \coloneqq \left\{ x = \{x_n\}_{n=1}^{\infty} \colon \norm{x}_{g_q} \coloneqq \sup_{n \in \mathbb{N}} \left( \frac{1}{n}\sum_{j=1}^n \abs{x_j}^q \right)^{1/q} < \infty \right\}$.
		Moreover, we have}
		\begin{equation} \label{EQ : factorization HLR fun}
			Ces_p = L_p \odot G_q,
		\end{equation}
		{\it where
		$G_q \coloneqq \left\{ f \in L_0 \colon \norm{f}_{G_q} \coloneqq \sup_{x > 0} \left( \frac{1}{x}\int_0^x \abs{f(t)}^q dt \right)^{1/q} < \infty \right\}$.}
	\end{corollary}
	\begin{proof}
		Since Hardy's operator $\mathscr{H}$ is bounded on $L_p$ and $L_p$ factorizes through itself,
		so due to Theorem~\ref{THM: factorization CX with Lp}, we have
		\begin{equation*}
			Ces_p = L_p \odot M(L_p,Ces_p).
		\end{equation*}
		Therefore, the only thing we have to do is to identify $M(L_p,Ces_p)$ with the space $G_q$.
		To do this, recall that it follows from (\ref{EQ: M(Lp,CX) =}) that
		\begin{align} \label{EQ: MLpCesp = blabla bla}
			M(L_p,Ces_p) = \Bigl( \bigoplus_{j \in \mathbb{Z}} L_q(\Delta_j) \Bigr)_{\mathbf{E}(M(L_p,L_p))(w^{1/q})},
		\end{align}
		where $w^{1/q}(j) \coloneqq 2^{-j/q}$ for $j \in \mathbb{Z}$. In consequence, we have
		\begin{align*}
			M(L_p,Ces_p)
				& = \Bigl( \bigoplus_{j \in \mathbb{Z}} L_q(\Delta_j) \Bigr)_{{\bf E}(L_{\infty})(w^{1/q})} \quad (\text{by \eqref{EQ: MLpCesp = blabla bla} and $M(L_p,L_p) \equiv L_{\infty}$}) \\
				& \equiv \left[ \Bigl( \bigoplus_{j \in \mathbb{Z}} L_1(\Delta_j) \Bigr)_{\ell_{\infty}(w)} \right]^{(q)}
					\quad (\text{since ${\bf E}(L_{\infty}) \equiv \ell_{\infty}$ and $\ell_{\infty}(w^{1/q}) \equiv \left[ \ell_{\infty}(w) \right]^{(q)}$}) \\
				& = \left[ Ces_{\infty} \right]^{(q)} \quad (\text{because of Theorem~\ref{Thm: Tandori function representation}}) \\
				& \equiv G_q \quad (\text{by the very definition}).
		\end{align*}
	\end{proof}

	\begin{remark}[General analogues of the spaces $g_p$ and $G_p$] \label{REMARK: CqX construction}
		Fix $1 \leqslant r < \infty$. Let us consider the operator $\mathscr{H}_{(r)}$ defined in the following way
		\begin{equation*}
			\mathscr{H}^{(r)} \colon f \rightsquigarrow \mathscr{H}^{(r)}(f) \coloneqq \bigl[ \mathscr{H}\left( \abs{f}^r \right) \bigr]^{1/r},
		\end{equation*}
		where $\mathscr{H}$, as always, stands for the Hardy operator. We will call $\mathscr{H}^{(r)}$ the {\bf $r^{\text{th}}$-convexification of Hardy's operator}
		$\mathscr{H}$. Although $\mathscr{H}^{(r)}$ is not linear, it is sublinear and positive. It is also straightforward to see that $\mathscr{H}^{(r)}$ is bounded
		on a Banach ideal space $X$ if, and only if, $\mathscr{H}$ is bounded on the $\frac{1}{r}$-concavification of $X$, that is, on $X^{(1/r)}$.
		Routine calculations show that
		\begin{equation*}
			\norm{\mathscr{H}^{(r)} \colon X \rightarrow X} = \norm{\mathscr{H} \colon X^{(1/r)} \rightarrow X^{(1/r)}}^{1/r}.
		\end{equation*}
	
		With such a notation in hand, we can define the following variant of the Ces{\` a}ro space construction, namely,
		\begin{equation*}
			\mathscr{C}^{(r)} X \coloneqq \left\{ f \in L_0 \colon \mathscr{H}^{(r)}(f) \in X \right\}
		\end{equation*}
		equipped with the norm $\norm{f}_{\mathscr{C}^{(r)} X} \coloneqq \norm{\mathscr{H}^{(r)}(f)}_X$. In other words, $\mathscr{C}^{(r)} X$
		is the optimal domain $[\mathscr{H}^{(r)},X]$. Of course, the space $\mathscr{C}^{(1)} X$ coincide, up to the equality
		of norms, with the usual Ces{\` a}ro space $\mathscr{C}X$.
		
		Let $1 \leqslant r < \infty$. Moreover, let $Z$ be a rearrangement invariant space with the Fatou property. Suppose that Hardy's operator
		$\mathscr{H}$ is bounded on $X^{(1/r)}$. We claim that
		\begin{equation} \label{EQ: representation CrX}
			\mathscr{C}^{(r)}Z = \Bigl( \bigoplus_{j \in \mathbb{J}} L_q(\Delta_j) \Bigr)_{\mathbf{E}(Z)\left( w^{1/r} \right)}.
		\end{equation}
		Indeed, due to Theorems~\ref{Thm: Tandori sequence representation} and \ref{Thm: Tandori sequence representation}, we have
		\begin{align*}
			\mathscr{C}^{(r)}Z
				& \equiv \left[ \mathscr{C}\left( X^{(1/r)} \right) \right]^{(r)} \\
				& = \left[ \Bigl( \bigoplus_{j \in \mathbb{J}} L_1(\Delta_j) \Bigr)_{\mathbf{E}\left( X^{(1/r)} \right)(w)} \right]^{(r)} \\
				& = \Bigl( \bigoplus_{j \in \mathbb{J}} L_q(\Delta_j) \Bigr)_{\mathbf{E}(Z)\left( w^{1/r} \right)}
		\end{align*}
		and \eqref{EQ: representation CrX} follows.
		
		Now, let $1 < p,q < \infty$ with $1/p + 1/q = 1$. Moreover, let $X$ be a rearrangement invariant space with the Fatou property.
		Suppose that Hardy's operator $\mathscr{H}$ is bounded\footnote{Somewhat tedious but elementary calculations show that $1/\alpha_{M(L_p,X)} = 1/\alpha_X - 1/p$.
		Thus, the assumption that Hardy's operator $\mathscr{H}$ is bounded on $M(L_p,X)^{(1/q)}$ can be paraphrased as $1/\alpha_{M(L_p,X)} < 1/q = 1 - 1/p$.
		But this simply means that $\alpha_X > 1$, that is, that Hardy's operator $\mathscr{H}$ is bounded on $X$.}
		on $M(L_p,X)^{(1/q)}$.
		Then, thanks to \eqref{EQ: M(Lp,CX) =} and \eqref{EQ: representation CrX}, we have
		\begin{equation*}
			M(L_p,\mathscr{C}X) = \Bigl( \bigoplus_{j \in \mathbb{J}} L_q(\Delta_j) \Bigr)_{\mathbf{E}(M(L_p,X))(w^{1/q})} = \mathscr{C}^{(q)} M(L_p,X).
		\end{equation*}
		In particular, since $M(L_p,L_p)^{(1/q)} \equiv L_{\infty}$, so
		\begin{equation*}
			G_q = M(L_p,Ces_p) = \mathscr{C}^{(q)} L_{\infty}
		\end{equation*}
		and
		\begin{equation*}
			g_q = M(\ell_p,ces_p) = \mathscr{C}^{(q)} \ell_{\infty}.
		\end{equation*}
	
		The moral of this short story is that the construction $X \rightsquigarrow \mathscr{C}^{(q)} X$ provides a language
		in which the space $M(L_p,\mathscr{C}X)$ finds an elegant description.
		\demo
	\end{remark}

	Dual result to Theorem~\ref{THM: factorization CX with Lp} is as follows.

	\begin{theorem}[Factorization of Lebesgue spaces through Tandori spaces] \label{THM: Lp factorized through Tandori X}
		{\it Let $1 \leqslant p,q \leqslant \infty$ with $1/p + 1/q = 1$. Further, let $X$ be a r.i. space with the Fatou property such that Hardy's operator
		$\mathscr{H}$ is bounded on $X$. Suppose that $L_p$ factorizes through $X$, that is, $L_p = X \odot M(X,L_p)$. Then}
		\begin{equation*}
			M(\widetilde{X},L_p) = \Bigl( \bigoplus_{j \in \mathbb{J}} L_p(\Delta_j) \Bigr)_{\mathbf{E}(M(L_q,X^{\times}))(w^{1/p})},
		\end{equation*}
		{\it where $w^{1/p}(j) \coloneqq 2^{-j/p}$ for $j \in \mathbb{J}$. Moreover, we have the following factorization}
		\begin{equation*}
			L_p = \widetilde{X} \odot M(\widetilde{X},L_p).
		\end{equation*}
	\end{theorem}
	\begin{proof}
		To begin with, let us note that for $p=1$ the above result is nothing else but an immediate consequence of Lozanovski{\u \i}'s factorization theorem.
		On the other hand, for $p = \infty$, we have
		\begin{align*}
			\widetilde{X} \odot M(\widetilde{X},L_{\infty})
				& \equiv \widetilde{X} \odot M(\widetilde{X},\widetilde{L_{\infty}}) \\
				& \equiv \widetilde{X} \odot \widetilde{M(X,L_{\infty})} \quad (\text{by Theorem~\ref{Cor: multipliers between Cesaro and Tandori}}) \\
				& = \widetilde{X \odot M(X,L_{\infty})} \quad (\text{using Proposition~\ref{PROP: Tandori komutuje z produktem}}) \\
				& = \widetilde{L_{\infty}} \equiv L_{\infty}.
		\end{align*}
		So, without loss of generality, we can assume that $1 < p < \infty$.
		
		Keeping this in mind, we have
		\begin{align*}
			M(\widetilde{X},L_p)
				& = M \left( L_p^{\times}, \left[ \widetilde{X} \right]^{\times} \right) \quad (\text{since $M(X,Y) \equiv M(Y^{\times},X^{\times})$}) \\
				& = M\left( L_q,\mathscr{C}(X^{\times}) \right) \quad (\text{by \cite[Theorem~D]{KMS07} or \cite[Theorem~2]{LM15a}}) \\
				& = \Bigl( \bigoplus_{j \in \mathbb{J}} L_p(\Delta_j) \Bigr)_{\mathbf{E}(M(L_q,X^{\times}))(w^{1/p})} \quad (\text{using (\ref{EQ: M(Lp,CX) =})}),
		\end{align*}
		where $w^{1/p}(j) \coloneqq 2^{-j/p}$ for $j \in \mathbb{J}$ and $1/p + 1/q = 1$. In consequence, using Theorems~\ref{Thm: Tandori sequence representation}
		and \ref{Thm: Tandori function representation}, we get
		\begin{align*}
			\widetilde{X} \odot M(\widetilde{X},L_p)
				& = \Bigl( \bigoplus_{j \in \mathbb{J}} L_{\infty}(\Delta_j) \Bigr)_{\mathbf{E}(X)}
						\odot \Bigl( \bigoplus_{j \in \mathbb{J}} L_p(\Delta_j) \Bigr)_{\mathbf{E}(M(L_q,X^{\times}))(w^{1/p})} \\
				& = \Bigl( \bigoplus_{j \in \mathbb{J}} L_p(\Delta_j) \Bigr)_{\mathbf{E}(X) \odot \mathbf{E}(M(L_q,X^{\times}))(w^{1/p})}
						\quad (\text{by Propositions~\ref{PROP: Pointwise multipliers of amalgams}})\\
				& = \Bigl( \bigoplus_{j \in \mathbb{J}} L_p(\Delta_j) \Bigr)_{\mathbf{E}(X \odot M(L_q,X^{\times}))(w^{1/p})} \quad (\text{by Proposition~\ref{PROP: E komutuje z produktem}}) \\
				& = \Bigl( \bigoplus_{j \in \mathbb{J}} L_p(\Delta_j) \Bigr)_{\mathbf{E}(X \odot M(X,L_p))(w^{1/p})} \\
				& = \Bigl( \bigoplus_{j \in \mathbb{J}} L_p(\Delta_j) \Bigr)_{\mathbf{E}(L_p)(w^{1/p})} \quad (\text{since $L_p = X \odot M(X,L_p)$}) \\
				& = \Bigl( \bigoplus_{j \in \mathbb{J}} L_p(\Delta_j) \Bigr)_{\ell_p(\mathbb{J})} \equiv L_p \quad (\text{using \eqref{EQ : L_p = suma prosta} and \eqref{EQ : ELpw = lp}}),
		\end{align*}
		The proof is finished.
	\end{proof}

	\begin{remark}[On factorization of duals] \label{REMARK Faktoryzacja = Faktoryzacja duali}
		It is worth noting that Theorem~\ref{THM: factorization CX with Lp} and Theorem~\ref{THM: Lp factorized through Tandori X}
		are in fact equivalent\footnote{Provided we additionally assume that the space $X$ has non-trivial Boyd indices (and not just the lower one).}.
		The reason for this is the following simple observation that for two Banach ideal spaces, say $X$ and $Y$,
		both with the Fatou property, we have
		\begin{equation} \label{EQ: FAK vs DUAL FAK}
			Y = X \odot M(X,Y) \quad \text{ if, and only if, } \quad X^{\times} = Y^{\times} \odot M(Y^{\times},X^{\times}).
		\end{equation}
		The details of proof looks as follows. Suppose that $Y = X \odot M(X,Y)$. Then, we have
		\begin{align*}
			X^{\times}
				& \equiv M(X,L_1) \\
				& \equiv M(X,Y \odot Y^{\times}) \quad (\text{by Lozanovski{\u \i}'s factorization theorem}) \\
				& = M(X,X \odot M(X,Y) \odot Y^{\times}) \quad (\text{since $Y = X \odot M(X,Y)$}) \\
				& \equiv M(L_{\infty},M(X,Y) \odot Y^{\times}) \quad (\text{using the cancellation property \eqref{CANCELLATION PROPERTY}}) \\
				& \equiv M(X,Y) \odot Y^{\times} \quad (\text{because $M(L_{\infty},Z) \equiv Z$}) \\
				& \equiv Y^{\times} \odot M(Y^{\times},X^{\times}) \quad (\text{since $M(X,Y) \equiv M(Y^{\times},X^{\times})$}).
		\end{align*}
		The proof of the reverse implication is entirely analogous. Thus, using \eqref{EQ: FAK vs DUAL FAK} together with Corollary~\ref{COR: blocking technique CX},
		we have
		\begin{equation*}
			\mathscr{C}X = L_p \odot M(L_p, \mathscr{C}X) \quad \text{ if, and only if, } \quad L_q = \widetilde{X} \odot M(\widetilde{X},L_q),
		\end{equation*}
		where $1 < p,q < \infty$ and $1/p + 1/q = 1$.
		This gives an alternative proof of Theorem~\ref{THM: Lp factorized through Tandori X} (or, if we want to look at the whole situation from the opposite side,
		of Theorem~\ref{THM: factorization CX with Lp}).
		\demo
	\end{remark}

	For $L_p$'s, Theorem~\ref{THM: Lp factorized through Tandori X} reduces to the following conclusion (see \cite[Proposition~2]{AM09}).

	\begin{corollary}[G. Bennett, 1996/S. V. Astashkin and L. Maligranda, 2009] \label{COR: faktoryzacja Lp i lp}
		{\it Let $1 \leqslant p < \infty$. Then, we have}
		\begin{equation*}
			\ell_p = \widetilde{\ell_p} \odot g_p \quad \textit{ and } \quad L_p = \widetilde{L_p} \odot G_p.
		\end{equation*}
	\end{corollary}
	\begin{proof}
		Fix $1 < p,q < \infty$ with $1/p + 1/q = 1$. Then
		\begin{align*}
			G_p 
				& \equiv (Ces_{\infty})^{(p)} \\
				& = M(L_q,Ces_q) \quad (\text{from the proof of Corollary~\ref{COR: B & AM factorization}}) \\
				& = M(\widetilde{L_p},L_p) \quad (\text{since $Ces_q^{\times} = \widetilde{L_p}$}),
		\end{align*}
		Now, it is enough to invoke Theorem~\ref{THM: Lp factorized through Tandori X}.
	\end{proof}

	\subsection{Even more factorizations} \label{SUBSECTION: Faktoryzacje CX przez CX i inne}
	Inspired by Kolwicz, Le{\' s}nik and Maligranda's recent results from \cite{KLM14} and \cite{KLM19},
	but this time leaving aside the factorization of Hardy's inequality itself,
	we want to show here a few more factorizations.
	Note also that some special cases of the below results can be already deduced from \cite{Be96}.
	
	\begin{theorem}[Factorization of Ces{\` a}ro spaces through Ces{\` a}ro spaces] \label{THEOREM: Factorization CX through CY}
		{\it Let $X$ and $Y$ be two r.i. spaces with the Fatou property such that Hardy's operator $\mathscr{H}$ is bounded on both spaces $X$ and $Y$.
		Suppose that the space $\mathbf{E}(Y)$ can be factorized through $\mathbf{E}(X)$, that is, $\mathbf{E}(Y) = \mathbf{E}(X) \odot M(\mathbf{E}(X),\mathbf{E}(Y))$.
		Then we have the following factorization}
		\begin{equation*}
			\mathscr{C}Y = \mathscr{C}X \odot M(\mathscr{C}X,\mathscr{C}Y).
		\end{equation*}
		{\it Moreover, for any quasi-concave function $\varphi$ with $1 < \gamma_{\varphi} \leqslant \delta_{\varphi} < \infty$, we have}
		\begin{equation*}
			\mathscr{C}Y = \mathscr{C}M_{\varphi} \odot M(\mathscr{C}M_{\varphi},\mathscr{C}Y)
				\quad \textit{ and } \quad \mathscr{C}\Lambda_{\varphi} = \mathscr{C}X \odot M(\mathscr{C}X,\mathscr{C}\Lambda_{\varphi}).
		\end{equation*}
	\end{theorem}
	\begin{proof}
		Due to Proposition~\ref{PROP: EX <-> carrier}, we have
		\begin{equation} \label{EQ: EMfi = ELinftyfi}
			\mathbf{E}(M_{\varphi}) = \mathbf{E}\left( \left[ L_{\infty}(\varphi) \right]^{\bigstar} \right)
				= \mathbf{E}(L_{\infty}(\varphi)) \equiv \ell_{\infty}(w_{\varphi}),
		\end{equation}
		where $w_{\varphi}(j) \coloneqq \varphi(2^{j+1})$ for $j \in \mathbb{J}$. Thus, we have
		\begin{align*}
			M(\mathscr{C} M_{\varphi}, \mathscr{C} Y)
				& = M\biggl( \Bigl( \bigoplus_{j \in \mathbb{J}} L_1(\Delta_j) \Bigr)_{\mathbf{E}(M_{\varphi})(w)},\Bigl( \bigoplus_{j \in \mathbb{J}} L_1(\Delta_j) \Bigr)_{\mathbf{E}(Y)(w)} \biggr)
						\quad (\text{by Theorem~\ref{Thm: Tandori function representation}}) \\
				& = \Bigl( \bigoplus_{j \in \mathbb{J}} M(L_1(\Delta_j),L_1(\Delta_j)) \Bigr)_{M(\mathbf{E}(M_{\varphi})(w),\mathbf{E}(Y)(w))}
						\quad (\text{using Theorem~\ref{Thm: komutowanie M z cdot}}) \\
				& \equiv \Bigl( \bigoplus_{j \in \mathbb{J}} L_{\infty}(\Delta_j) \Bigr)_{M(\mathbf{E}(M_{\varphi}),\mathbf{E}(Y))}
						\quad (\text{since $M(L_1,L_1) \equiv L_{\infty}$}) \\
				& = \Bigl( \bigoplus_{j \in \mathbb{J}} L_{\infty}(\Delta_j) \Bigr)_{M(\ell_{\infty}(w_{\varphi}),\mathbf{E}(Y))} \quad (\text{by \eqref{EQ: EMfi = ELinftyfi}}),
		\end{align*}
		where $w(j) = 2^{-j}$ for $j \in \mathbb{J}$. Let us now focus on $M(\ell_{\infty}(w_{\varphi}),\mathbf{E}(Y))$ for a moment. It is straightforward to see that
		\begin{align*}
			M(\ell_{\infty}(w_{\varphi}),\mathbf{E}(Y))
				& = M(\ell_{\infty}(\mathbb{J}),\mathbf{E}(Y))(1/w_{\varphi}) \\
				& = M(\mathbf{E}(Y)^{\times},\ell_1(\mathbb{J}))(1/w_{\varphi}) \\
				& = \mathbf{E}(Y)^{\times \times}(1/w_{\varphi}) \\
				& = \mathbf{E}(Y)(1/w_{\varphi}),
		\end{align*}
		where the last equality is due to the Fatou property of the space $Y$. In consequence,
		\begin{equation} \label{EQ : M(CMvarphi,CY) = }
			M(\mathscr{C} M_{\varphi}, \mathscr{C}Y) = \Bigl( \bigoplus_{j \in \mathbb{J}} L_{\infty}(\Delta_j) \Bigr)_{\mathbf{E}(Y)(1/w_{\varphi})}.
		\end{equation}
		To sum up, we have
		\begin{align*}
			\mathscr{C} M_{\varphi} \odot M(\mathscr{C} M_{\varphi}, \mathscr{C} Y)
				& = \Bigl( \bigoplus_{j \in \mathbb{J}} L_1(\Delta_j) \Bigr)_{\mathbf{E}(M_{\varphi})(w)}
					\odot \Bigl( \bigoplus_{j \in \mathbb{J}} L_{\infty}(\Delta_j) \Bigr)_{\mathbf{E}(Y)(1/w_{\varphi})} \quad (\text{by \eqref{EQ : M(CMvarphi,CY) = }}) \\
				& = \Bigl( \bigoplus_{j \in \mathbb{J}} L_1(\Delta_j) \odot L_{\infty}(\Delta_j) \Bigr)_{\mathbf{E}(L_{\infty}(\varphi))(w) \odot \mathbf{E}(Y)(1/w_{\varphi})} \\
				& \quad \quad \quad \quad \quad \quad \quad \quad \quad \quad \quad (\text{using Theorem~\ref{PROP: Pointwise multipliers of amalgams} and Proposition~\ref{PROP: EX <-> carrier}}) \\
				& = \Bigl( \bigoplus_{j \in \mathbb{J}} L_1(\Delta_j) \Bigr)_{\left[ \ell_{\infty}(\mathbb{J}) \odot \mathbf{E}(Y) \right](w)} \quad (\text{by \eqref{EQ: EMfi = ELinftyfi}}) \\
				& \equiv \Bigl( \bigoplus_{j \in \mathbb{J}} L_1(\Delta_j) \Bigr)_{\mathbf{E}(Y)(w)} \quad (\text{since $L_{\infty} \odot Z \equiv Z$}) \\
				& = \mathscr{C} Y.
		\end{align*}
		The proof is thus completed.
	\end{proof}

	By duality, we have

	\begin{theorem}[Factorization of Tandori spaces through Tandori spaces] \label{THEOREM: factorization Tandori przez Tandori}
		{\it Let $X$ and $Y$ be two r.i. spaces with the Fatou property. Suppose that the space $\mathbf{E}(Y)$ can be factorized through
		$\mathbf{E}(X)$, that is, $\mathbf{E}(Y) = \mathbf{E}(X) \odot M(\mathbf{E}(X),\mathbf{E}(Y)$. Then we have the following factorization}
		\begin{equation*}
			\widetilde{Y} = \widetilde{X} \odot M(\widetilde{X},\widetilde{Y}).
		\end{equation*}
		{\it Moreover, for any quasi-concave function $\varphi$ with $1 < \gamma_{\varphi} \leqslant \delta_{\varphi} < \infty$, we have}
		\begin{equation*}
			\widetilde{Y} = \widetilde{M_{\varphi}} \odot M(\widetilde{M_{\varphi}},\widetilde{Y})
			\quad \textit{ and } \quad \widetilde{\Lambda_{\varphi}} = \widetilde{X} \odot M(\widetilde{X},\widetilde{\Lambda_{\varphi}}).
		\end{equation*}
	\end{theorem}
	\begin{proof}
		One can simply imitate the proof of Theorem~\ref{THEOREM: Factorization CX through CY}. We leave the easy details to the interested reader.
	\end{proof}

	We can also say something about factorization of Ces{\' a}ro spaces through Tandori spaces and {\it vice versa}.

		\begin{theorem}[Mixed type factorization] \label{THEOREM: Mixed factorizations} 
		{\it Let $X$ and $Y$ be two r.i. spaces with the Fatou property such that Hardy's operator $\mathscr{H}$ is bounded on $Y$.
		Suppose that the space $\mathbf{E}(Y)$ can be factorize through $\mathbf{E}(X)$, that is, $\mathbf{E}(Y) = \mathbf{E}(X) \odot M(\mathbf{E}(X),\mathbf{E}(Y))$.
		Then we have the following factorizations}
		\begin{equation*}
			\mathscr{C}Y = \widetilde{X} \odot M(\widetilde{X},\mathscr{C}Y) \quad \textit{ and } \quad \widetilde{Y} = \mathscr{C}X \odot M(\mathscr{C}X,\widetilde{Y}).
		\end{equation*}
		{\it Moreover, for any quasi-concave function $\varphi$ with $1 < \gamma_{\varphi} \leqslant \delta_{\varphi} < \infty$, we have}
		\begin{equation*}
			\mathscr{C}Y = \widetilde{M_{\varphi}} \odot M(\widetilde{M_{\varphi}},\mathscr{C}Y),
			\quad \mathscr{C}\Lambda_{\varphi} = \widetilde{X} \odot M(\widetilde{X},\mathscr{C}\Lambda_{\varphi}),
		\end{equation*}
		\begin{equation*}
			\widetilde{Y} = \mathscr{C}M_{\varphi} \odot M(\mathscr{C}M_{\varphi},\widetilde{Y})
			\quad \textit{ and } \quad \widetilde{\Lambda_{\varphi}}  = \mathscr{C}X \odot M(\mathscr{C}X,\widetilde{\Lambda_{\varphi}}).
		\end{equation*}
	\end{theorem}
	\begin{proof}
		We leave a not-so-complicated proof of this result as an exercise for the patient reader.
		Again, the details do not differ much from the proof of Theorem~\ref{THEOREM: Factorization CX through CY}.
	\end{proof}

	\subsection{So why we are doing better anyway?} \label{REMARK: faktoryzacja z E lepsza niz bez E}
		So far, we have not provided a concrete example to clearly demonstrate that our results about multipliers, products,
		and factorization from Sections~\ref{SECTION: Products and factors: Grosse-Erdmann's style}
		and \ref{SUBSECTION: Faktoryzacje CX przez CX i inne} actually improve upon existing results.
		The time has come to change this.
		
		\begin{example} \label{PROPOSITION: faktoryzacja z E lepsza niz bez E}
			{\it Let $X$ and $Y$ be two r.i. spaces with the Fatou property. Then the assumption that the space $\mathbf{E}(Y)$
			factorizes through $\mathbf{E}(X)$ is essentially weaker than the corresponding assumption that the space $Y$ factorizes
			through $X$.}
		\end{example}
		\begin{proof}
			Suppose that the space $Y$ factorizes through $X$, that is, $Y = X \odot M(X,Y)$.
			Then, in view of Lemma~\ref{PROP: E komutuje z produktem}, we have
			\begin{equation*}
				\mathbf{E}(Y) = \mathbf{E}(X \odot M(X,Y)) = \mathbf{E}(X) \odot \mathbf{E}(M(X,Y)).
			\end{equation*}
			Moreover, using the cancelation property for multipliers (see Theorem~\ref{THM: cancellation property for multipliers}), we get
			\begin{align*}
				M(\mathbf{E}(X),\mathbf{E}(Y))
				& = M\bigl( \mathbf{E}(X) \odot \ell_{\infty}, \mathbf{E}(X) \odot \mathbf{E}(M(X,Y)) \bigr) \\
				& = M(\ell_{\infty}(\mathbb{J}),\mathbf{E}(M(X,Y))) \\
				& = \mathbf{E}(M(X,Y)).
			\end{align*}
			Consequently,
			\begin{equation*}
				\mathbf{E}(Y) = \mathbf{E}(X) \odot \mathbf{E}(M(X,Y)) = \mathbf{E}(X) \odot M(\mathbf{E}(X),\mathbf{E}(Y)),
			\end{equation*}
			that is, the space $\mathbf{E}(Y)$ can be factorized through $\mathbf{E}(X)$.
			
			It remains to show that is possible to find two rearrangement invariant spaces, say $E$ and $F$, such that $F \neq E \odot M(E,F)$,
			but at the same time $\mathbf{E}(F) = \mathbf{E}(E) \odot M(\mathbf{E}(E),\mathbf{E}(F))$.
			To this end, let us first recall how the factorization of Lorentz spaces $L_{p,q}$ looks like.
			Fix $1 < p,q,r,s < \infty$ with $p = r$ and $q > s$.
			Then, as Theorem~\ref{THEOREM: Factorization of Lpq} teaches us, we have
			\begin{equation*}
				L_{p,q} \odot M(L_{p,q},L_{r,s}) \neq L_{r,s}.
			\end{equation*}
			However, due to Proposition~\ref{PROP: EX <-> carrier},
			\begin{equation} \label{EQ: ELpq=LqWpq}
				\mathbf{E}(L_{p,q})
					\equiv \mathbf{E}\left( \left[ L_q ( t^{1/p-1/q} ) \right]^{\bigstar} \right)
					= \mathbf{E}\left( L_q ( t^{1/p-1/q} ) \right)
					\equiv \ell_q(W_{p,q}),
			\end{equation}
			where $W_{p,q}(j) \coloneqq \int_{\Delta_j} t^{1/p-1/q} dt$ for $j \in \mathbb{Z}$. Similarly,
			\begin{equation} \label{EQ: ELPQ=LQWPQ}
				\mathbf{E}(L_{r,s}) = \ell_s(W_{r,s}),
			\end{equation}
			where $W_{r,s} \coloneqq \int_{\Delta_j} t^{1/r-1/s} dt$ for $j \in \mathbb{Z}$. In consequence,
			\begin{align*}
				\mathbf{E}(L_{p,q}) \odot M(\mathbf{E}(L_{p,q}),\mathbf{E}(L_{r,s}))
					& = \ell_{q}(W_{p,q}) \odot M(\ell_{q}(W_{p,q}),\ell_{s}(W_{r,s})) \quad (\text{by \eqref{EQ: ELpq=LqWpq} and \eqref{EQ: ELPQ=LQWPQ}}) \\
					& = (\ell_{q} \odot M(\ell_{q},\ell_{s}))(W_{r,s}) \\
					& = \ell_{s}(W_{r,s}) \quad (\text{since $q > s$}) \\
					& = \mathbf{E}(L_{r,s}) \quad (\text{again, by \eqref{EQ: ELPQ=LQWPQ}}),
			\end{align*}
			The proof is complete.
		\end{proof}

	\section{{\bf Open ends}} \label{SECTION: Open ends}
	
	Let us list and discuss some problems which arise from this paper.
	
	\subsection{Factorization of inequalities, not quite} We have already mentioned Bennett's and Astashkin and Maligranda's results on factorization
	of the HLR-inequality many times (see Corollary~\ref{COR: B & AM factorization}; cf. \cite[Theorem~1.5]{Be96} and \cite[Proposition~1]{AM09}).
	Let us do it one last time: for $1 < p,q < \infty$ with $p + q = pq$, we have
	\begin{equation*}
		\ell_p \odot g_q = ces_p \quad \text{ and } \quad L_p \odot G_q = Ces_p.
	\end{equation*}
	Since the spaces $g_q$ and $G_q$ are nothing else but $M(\ell_p,ces_p)$ and, respectively, $M(L_p,Ces_p)$, so it is perfectly natural to expect
	that in general, say, for any r.i. space $X$ with the Fatou property and non-trivial Boyd indices, we have
	\begin{equation} \label{QUESTION: X x M(X,CX) = CX}
		X \odot M(X,\mathscr{C}X) = \mathscr{C}X.
	\end{equation}
	This, in turn, comes down to computing the space $M(X,\mathscr{C}X)$. We do not know how to do this in general, but we have a few thoughts on this problem.
	
	First, things get a little nicer when we restrict ourselves to the class of Lorentz spaces $\Lambda_{\varphi}$. Indeed,
	\begin{align*}
		\Lambda_{\varphi} \odot M(\Lambda_{\varphi},\mathscr{C}\Lambda_{\varphi})
			& = \Lambda_{\varphi} \odot M(\Lambda_{\varphi}, L_1(w)) \quad \text{(since $\mathscr{C}\Lambda_{\varphi} = L_1(w)$)}\\
			& \equiv \Lambda_{\varphi} \odot M(\Lambda_{\varphi}, L_1)(w) \quad \text{(since $M(X,Y(w)) \equiv M(X,Y)(w)$)}\\
			& = \Lambda_{\varphi} \odot M_{\psi}(w) \quad \text{(because $M(\Lambda_{\varphi}, L_1) \equiv \Lambda_{\varphi}^{\times} = M_{\psi}$ with $\psi(t) = t/\varphi(t)$)}\\
			& \equiv (\Lambda_{\varphi} \odot M_{\psi})(w) \quad \text{(since $X(w) \odot Y \equiv X \odot Y(w) \equiv (X \odot Y)(w)$)}\\
			& \equiv L_1(w) \quad \text{(by Lozanovski{\u \i}'s factorization theorem)} \\
			& = \mathscr{C}\Lambda_{\varphi}.
	\end{align*}
	The same argument works for Lorentz sequence spaces $\lambda_{\varphi}$.
	
	Secondly, the space $M(X,\mathscr{C}X)$ almost coincides with $X^{\times}$ in the following sense: Due to \cite[Proposition~2.2]{ALM19}
	(cf. \cite[Lemma~4.1]{KKM21}), for any $0 < a < b < \infty$, we have
	\begin{equation*}
		R_{[a,b]}\mathscr{C}X \coloneqq \left\{ f \in \mathscr{C}X \colon \supp f \subset [a,b] \right\} = R_{[a,b]}L_1 \equiv L_1([a,b]).
	\end{equation*}
	In consequence,
	\begin{equation*}
		M(R_{[a,b]}X,R_{[a,b]}\mathscr{C}X) = R_{[a,b]}M(X,\mathscr{C}X) = R_{[a,b]}X^{\times}.
	\end{equation*}
	For example, since
	\begin{equation*}
		R_{[a,b]} G_q \equiv R_{[a,b]} \left( Ces_{\infty} \right)^{(q)} \equiv \left[ R_{[a,b]} Ces_{\infty} \right]^{(q)}
			= \left[ L_1([a,b]) \right]^{(q)} \equiv L_q([a,b]) \equiv \left[ L_p([a,b]) \right]^{\times},
	\end{equation*}
	so in the \enquote{middle} of $G_q$, life goes on as in $L_q$.
	
	\subsection{Factorization of generalized weighted spaces}
	Inspired by \cite[Remark~2]{AM09}, Barza, Marcoci and Marcoci \cite{BMM18} proposed a factorization of the weighted spaces $(L_p)_w$ and $\mathscr{C}(L_p)_w$
	in the style of Corollary~\ref{COR: faktoryzacja Lp i lp} and Corollary~\ref{COR: B & AM factorization}, respectively (see \cite[Theorems~2.1 and 3.1]{BMM18}).
	It would be interesting to have an abstract analogues of these results built upon generalized weighted spaces $X_w$.
	
	\subsection{Some problems with interpolation}
	Let $\mathbf{F}$ be an interpolation functor. Being optimistic, can we somehow describe the space $\mathbf{F}(\widetilde{X},X)$
	or $\mathbf{F}(X,X^{\downarrow})$? Or, in a more realistic version, how to identify the space $\widetilde{X}^{1-\theta}X^{\theta}$ and $X^{1-\theta}(X^{\downarrow})^{\theta}$?
	We know how to do it, for example, when $X$ is just $L_p$, but the general case remains a complete mystery (cf. \cite[Problem~3]{LM15a}).
	We believe that that the seemingly simpler case of the complex method is the key.
	
	\subsection{Multipliers between abstract Ces{\' a}ro sequence spaces} \label{QUESTION: M(CX,CY) if M(X,Y) = Linfty}
	Let $X$ and $Y$ be two r.i. sequence spaces. Suppose that $M(X,Y) = \ell_{\infty}$ but $X \neq Y$. The question left unanswered in Section~\ref{SECTION: Examples a'la Bennett}
	concerns the description of space $M(\mathscr{C}X,\mathscr{C}Y)$ in this setting. The fact that $M(X,Y) = \ell_{\infty}$
	implies that there is no chance for factorization, because $X \odot M(X,Y) = X \odot \ell_{\infty} = X \neq Y$.
	Although we can do this for Orlicz sequence spaces $M(ces_M,ces_N)$ (see Example~\ref{przyklad Orlicze infty}), we have no idea how to do this in general.
	
	\subsection{Multipliers between $Ces_M$ and $Ces_N$} \label{QUESTION: Multipliers Cesaro--Orlicz function}
	Let $M$ and $N$ be two Young functions. Suppose that $M(L_M,L_N) = L_{\infty}$ but $L_M \neq L_N$. How to identify the space $M(Ces_M,Ces_N)$ in this case?
	Based on Example~\ref{przyklad Orlicze infty}, it seems reasonable to assume that the space $M(Ces_M,Ces_N)$ coincide with the weighted $L_{\infty}$ space.
	However, the techniques we have developed fail in this case, so some new ideas are probably necessary.
	
	\subsection{Missing analogs for spaces on the probability measure} \label{SUBSEC: CX on PROB MES SP}
	Let $X$ be a r.i. function space on the unit interval equipped with the Lebesgue measure.
	Virtually all our results (as well as many other authors; see, for example, \cite{LM16} and \cite{KLM19}) remain open in this case.
	This is probably the most interesting of the problems listed here. However, as the differences in the description of the dual to ${\mathscr C}X$
	and discrepancies in the interpolation structure clearly show (see \cite{AM09}, \cite{AM13}, \cite{LM15a} and \cite{LM16}), it is probably also the most difficult one.
	
	\subsection{The best-possible constants strike back} Already Grosse--Erdmann clearly emphasized the qualitative character of the results obtained by
	the blocking technique and on many occasions raised the question about their hidden quantitative content, that is, about the best-possible constants
	involved in the crucial inequalities (see, for example, \cite[pp.~3--4 and p.~104]{GE98}). This question remains almost\footnote{For example, Masty{\l}o and Sinnamon
	were able to show that the $\mathscr{K}$-divisibility constant for the couple $(\widetilde{L_1},L_{\infty})$ is one (see \cite{MS17} for details; cf. \cite{HS23}).}
	completely valid in the context of our results.
	Although problems of this type seem to be the least pressing, it is an interesting topic with a classic twist and a sporty flair. However, we do
	not even know how exactly (that is, up to the equality of norms) describe the space of pointwise multipliers between two Orlicz spaces (to the best of
	our knowledge, the Banach--Mazur distance between $M(L_M,L_N)$ and $L_{M \ominus N}$ is less than or equal to 8).
	
	\appendix

	\section{{\bf A few comments and bibliographical notes}} \label{SECTION: Appendix - bibliographical notes}
	
	Although bibliography we have collected may seem quite substantial, it is hardly complete or comprehensive in any sense
	and we do not even try to make it so. On the one hand, some traces of the topics discussed here date back to the late 1940's.
	On the other, multitude of intertwined threads and possible points of view seems overwhelming. Our intention was rather
	to show a wide spectrum of themes, to clearly highlight some connections, and to firmly ground the undertaken research in the existing classics.
	
	\begin{remark}
		The space $Ces_{\infty}$ were introduced, but under somewhat different name $K$, already in 1948 by Korenblyum, Kre{\u \i}n and Levin
		(see \cite{KKL48}; cf. \cite{LZ65} and \cite{Wn99}).
		\demo
	\end{remark}

	\begin{remark}
		K{\' a}roly Tandori in his 1955 paper \cite{Ta55} and, independently, Wilhelmus A. J. Luxemburg
		and Adriaan C. Zaanen in the mid 1960's \cite{LZ65} found a description of the K{\" o}the dual to
		$Ces_{\infty}$. They were able to show that the space $(Ces_{\infty})^{\times}$ can be identified, even up to
		the equality of norms, with $\widetilde{L_1}$. Knowing this, it is also easy to see that
		\begin{equation*}
			( \widetilde{L_1} )^{\times} \equiv (Ces_{\infty})^{\times \times} \equiv Ces_{\infty},
		\end{equation*}
		where the last equality is due to the Fatou property of the space $Ces_{\infty}$.
		\demo
	\end{remark}

	\begin{remark}
		The space $\widetilde{\ell_1}$ appeared as early as 1949 in Arne Beurling's work \cite{Beu49}. Roughly speaking, because of his interest
		in the problem of the absolute convergence of contracted Fourier series, he introduced the space $A^{*} = A^{*}(\mathbb{T})$ (nowadays, the so-called
		{\it Beurling space}) as the space of all measurable functions, say $f$, living on the unit circle $\mathbb{T}$ such that $f$ belongs to $L_1(\mathbb{T})$
		and their Fourier coefficients satisfy the condition
		\begin{equation*}
			 \norm{ \left\{ \widehat{f}(j) \right\}_{j \in \mathbb{Z}} }_{\widetilde{\ell_1}}
			 	= \sum_{n = 0}^{\infty} \sup\limits_{\abs{j} \geqslant n} \abs{\widehat{f}(j)}
			 	< \infty
		\end{equation*}
		(precisely, see formula (2.5) on page 227 in \cite{Beu49}).
		\demo
	\end{remark}

	\begin{remark}
		When it comes to the description of the dual space of $\widetilde{\ell_1}$, however, it was given in 1957 by Andrzej Alexiewicz in his seemingly
		little-known paper \cite{Ale57}. He was not only able to prove that the K{\" o}the dual (in fact, also the topological dual, because the space $\widetilde{\ell_1}$
		is separable) to $\widetilde{\ell_1}$ coincide, up to the equality of norms, with $ces_{\infty}$ (see \cite[Theorem~3]{Ale57}), but also to generalize
		this result to the case of weighted spaces $\widetilde{\ell_1(w)}$ (see \cite[Theorem~2]{Ale57} for more details). Again, thanks to the Fatou
		property of the space $\widetilde{\ell_1}$, Alexiewicz's result implies, among other things, that
		\begin{equation*}
			(ces_{\infty})^{\times} \equiv ( \widetilde{\ell_1} )^{\times \times} \equiv \widetilde{\ell_1}.
		\end{equation*}
		\demo
	\end{remark}
	
	More information and references can be found, for example, in \cite{AM09}, \cite{AM14}, \cite{Be96}, \cite{GE98}, \cite{KMP07}, \cite{LM15a} and \cite{ORS08}.

\end{document}